\documentclass[10pt,reqno]{article} 
\usepackage{amsfonts,amsmath,layout}
\usepackage{mathrsfs}  
\usepackage{soul}
\usepackage{amssymb,mathabx,amsfonts,amsthm,amscd,stmaryrd,dsfont,esint,upgreek,constants,todonotes}
\DeclareFontFamily{OT1}{pzc}{}
\DeclareFontShape{OT1}{pzc}{m}{it}{<-> s * [1.10] pzcmi7t}{}
\DeclareMathAlphabet{\mathpzc}{OT1}{pzc}{m}{it}

\usepackage{color} 

\definecolor{Red}{cmyk}{0,1,1,0.2}

\newcommand{\N}{\mathbb N}

\newcommand{\R}{\mathbb R}

\def\R{\mathbb R}
\def\N{\mathbb N}

\def\ep{\epsilon}

\newcommand{\be}{\begin{equation}}
\newcommand{\ee}{\end{equation}}
\def\1{{\bf 1}}

\newtheorem{Theorem}{Theorem}[section]
\newtheorem{Definition}[Theorem]{Definition}
\newtheorem{Proposition}[Theorem]{Proposition}
\newtheorem{Lemma}[Theorem]{Lemma}
\newtheorem{Corollary}[Theorem]{Corollary}
\newtheorem{Remark}[Theorem]{Remark}

\begin{document}

\title{A class of germs arising from homogenization\\ in traffic flow on junctions}
\author{\renewcommand{\thefootnote}{\arabic{footnote}}
  P. Cardaliaguet\footnotemark[1], N. Forcadel\footnotemark[2], R. Monneau\footnotemark[1] ~\footnotemark[3]}
\footnotetext[1]{CEREMADE, UMR CNRS 7534, Universit\'e Paris Dauphine-PSL,
Place de Lattre de Tassigny, 75775 Paris Cedex 16, France. }
\footnotetext[2]{INSA Rouen Normandie, Normandie Univ, LMI UR 3226, F-76000 Rouen, France.}
\footnotetext[3]{CERMICS, Universit\'e Paris-Est, Ecole des Ponts ParisTech, 6-8 avenue Blaise Pascal, 77455 Marne-la-Vall\'ee Cedex 2, France.}

\maketitle

\begin{abstract}
We {consider traffic flows described by conservation laws. We study}  a 2:1 junction  (with two incoming roads and one outgoing road) or a 1:2 junction  (with one incoming road and two outgoing roads).
At the mesoscopic level, the priority law at the junction is given by traffic lights, which are periodic in time and the traffic can also be slowed down by periodic in time flux-limiters.

After a  long time, and at large scale in space, we intuitively expect an effective junction condition to emerge.
Precisely, we perform a rescaling in space and time, to pass from the mesoscopic scale to the macroscopic {one}.
At the limit of the rescaling, we show rigorous homogenization of the problem and identify the effective junction condition, which belongs to a general class of germs (in the terminology of \cite{AKR11,FMR22b,MFR22}). The identification of this germ and of a characteristic subgerm which determines the whole germ, is the first key result of the paper.

The second key result of the paper is the construction of a family of correctors whose values at infinity are related to {each} element of the characteristic subgerm.
This construction is indeed explicit  at the level of some mixed Hamilton-Jacobi equations for concave Hamiltonians (i.e. fluxes). The explicit solutions are found in the spirit of representation formulas for optimal control problems.
%
\end{abstract}


\tableofcontents

\section{Introduction} 

In this section  we introduce  the problem and the main notations,  assumptions and results of the paper.  We start with a  foreword in which we explain the goal of the paper. Then  we introduce the notions of germs and  our two main models (mesoscopic and macroscopic). We give our main results and compare them with the literature. We finally describe the organization of the paper.

\subsection{Foreword}\label{s1.1}

The goal of the paper  is to  understand and to  justify effective junction conditions for macroscopic models of traffic flows arising by homogenization of mescoscopic models.
We concentrate here  on junctions involving two incoming roads and a single outgoing one (referred later on as 2:1 junctions), or the opposite: one incoming road and two outgoing ones (referred as 1:2 junctions). On each road, the equation satisfied by the density is a scalar conservation law of the form 
$$
\partial_t \rho+ \partial_x(f(\rho))=0,
$$
where the concave flux function $f$  can depend on the road. At the junction point we require of course a Rankine-Hugoniot condition, as well as relations between the incoming and outgoing fluxes, which define what is called a germ. For the mesoscopic model, the germ is an oscillating function of time, which can be interpreted as periodic in time traffic lights (or more generally flux limiters).  For instance, for the 1:2 junction, traffic lights regulate the traffic, dispatching the vehicles in one of the two exit branches. For 2:1 junction, the traffic lights give the priority rules.

Looking at long time behavior and on large space scale, we show that the oscillating germ for the mesoscopic model homogenizes in an effective (and homogeneous)  germ for the macroscopic model. On the branches, the PDEs satisfied by the densities are the same for the macroscopic model  and the mesoscopic model; only the junction condition (the germ) changes.   Our homogenization procedure naturally  introduces a general class of germs for conservation laws on 1:2 and 2:1 junctions. The guess and the study of those germs (Theorem \ref{th:1})  is the first key contribution of this paper. The second key contribution is the  rigorous justification of the homogenization by the construction of suitable correctors (Theorem \ref{thm.main2:1} for 2:1 junctions and Theorem \ref{thm.main} for 1:2 junctions).


For the mesoscopic model, we manage to reduce the junction condition to a 1:1 junction, involving at each time  one incoming road and one outgoing road only.  1:1 junctions  are well understood and justified \cite{AGDV11, AKR11, AuPe05, BaJe97, GNPT07, To00, To01}; they are  known to arise by homogenization of microscopic models of follow-the-leader type \cite{CFarma, CFsiam21, rigorousLWR, FSZ18, FoSa20, GaImMo15} and there is an equivalence between the approach through the germ theory for 1:1 junctions and the one using Hamilton-Jacobi (HJ) equations on such  junctions  \cite{CFGM}. We will make an extensive use of this equivalence (still in the case 1:1) in the construction of correctors. Our new junction conditions (for 2:1 and 1:2 junctions) arise rigorously by mixing these very natural 1:1 junctions. 
Let us  underline that the mesoscopic models we consider possess an $L^1-$contraction property, and, as expected, this is also the case for our limit models after homogenization. Note however that, in the literature, there exists some junction models which do not possess this $L^1-$contraction property\footnote{For instance traffic flows on 1:2 junctions in which the positive  proportion of the traffic entering each outgoing road is fixed,  are never $L^1-$contractions.}.

 For our mesoscopic models, we use the approach through germs developed in  \cite{AKR11}. This approach, which relies on the notion of trace developed by Panov \cite{Ps07} (see also \cite{V01}), consists in requiring  that the trace of the solution at the junction belongs to a set, the germ. As recalled in Subsection \ref{subsec.germ},  the fact that the germ is ``maximal'' ensures the uniqueness of the solution to the conservation law and its stability. Existence, on the other hand, comes from the ``completeness''  of this germ. 
\medskip

As explained above, the paper partially relies (for the construction of correctors) on the formulation of traffic flows in terms of Hamilton-Jacobi on a 1:1 junction. HJ equation on junctions have been discussed in many works \cite{ACCT13, AOT15, CaMa13, IM17, LiSo16,LS2,ShCa13}; see also the recent monograph \cite{BaCh}. The central notion of  flux limiters, used throughout this paper, has been  developed in \cite{IM17}. Questions of homogenization in this framework are discussed in \cite{AcTc15, CFarma, IM17, FSZ18, FoSa20, GaImMo15}. In contrast with the approach developed here, these papers rely on a comparison principle. Homogenization of scalar conservation laws has been less considered in the literature: see \cite{Da09, Ew92, Se92}, and, as far as we know, never for problems on a junction.

\bigskip

Now a few words about the techniques of proof are in order. Let us first underline that, for technical reasons, we mainly work throughout the paper in the case of 1:2 junctions; the maybe more interesting problem of junctions of type 2:1 is handled by a simple change of variables in Subsection \ref{proof2:1}. Second, and in contrast with most homogenization results we are aware of on the topic and quoted above, the homogenization does not rely directly on a comparison principle for some Hamilton-Jacobi formulation on the junction: indeed the limit problem {\it cannot } naturally be formulated in terms of pure HJ equations with some general comparison principle at the HJ level.

The homogenization must therefore  be proved directly at the level of the scalar conservation laws. The construction of correctors for each element of the homogenized germ seems to be a difficult task in general. For this reason we first show the existence of a subset of the  germ, called a characteristic subgerm, which  determines the whole germ (Lemma \ref{lem.reduction}). This characteristic subgerm will be then used to guide  the construction of correctors. Indeed, to each element of the characteristic subgerm, we associate a corrector whose values at infinity are given by the values of this element  (Theorem \ref{thm.corrector}). This construction  uses explicit solutions for suitable HJ equations with concave Hamiltonians in the flavor of the Lax-Oleinik formula. The explicit solutions are guessed in the spirit of representation formulas in optimal control theory on junctions \cite{IM17}. The proof of homogenization is then achieved thanks to Kato's inequality and germ's theory developed in \cite{AKR11,MFR22}.
\bigskip

Note that the mesoscopic model can itself be thought as the limit of a microscropic model taking the form of a follow-the-leader model on a junction, as discussed in \cite{CHM20} for instance. However the rigorous derivation of the macroscopic model from a microscopic one seems a very challenging question. Another open problem is the analysis of junctions involving {four branches or more}, which seems to require new ideas.

\subsection{Standing notation and assumptions}

The following assumptions are in force throughout the paper. 

Let $\mathcal R^0=(-\infty,0)\times \{0\}$ be the incoming branch, $\mathcal R^{j}= (0,\infty)\times \{j\}$ for $j=1,2$ being the outgoing ones. We consider the set $\mathcal R=\bigcup_{j=0}^2 \mathcal R^{j}\cup\{0\}$ with the topology of three half lines glued together at the origin $0$.

Let $a^j<b^j<c^j$ for $j\in\{0,1,2\}$. We make the following assumptions on the fluxes for some $\delta>0$: 

\be\label{hypflux}
\begin{array}{c}
\text{For $j\in\{0,1,2\}$, the flux $f^j:[a^j, c^j]\to \R$ is of class $C^2$, with $(f^j)''\le -\delta<0$ on $[a^j,c^j]$,} \\
\text{increasing on $[a^j,b^j]$ and decreasing on $[b^j,c^j]$, with $f^j(a^j)=f^j(c^j)=0$.}
\end{array}
\ee
We set
\begin{equation}\label{eq::e*1}
f^j_{\max}:=\max_{[a^j,c^j]} f^j=f^j(b^j)>0
\end{equation}
and define the nondecreasing envelope of $f^j$
\begin{equation}\label{eq::e*2}
f^{j,+}(p):=\left\{\begin{array}{ll}
f^j(p)&\quad \mbox{for}\quad p\in [a^j,b^j]\\
f^j(b^j) &\quad \mbox{for}\quad p\in [b^j,c^j]\\
\end{array}\right.
\end{equation}
and its nonincreasing envelope
\begin{equation}\label{eq::e*3}
f^{j,-}(p):=\left\{\begin{array}{ll}
f^j(b^j)&\quad \mbox{for}\quad p\in [a^j,b^j]\\
f^j(p) &\quad \mbox{for}\quad p\in [b^j,c^j].\\
\end{array}\right.
\end{equation}
Throughout the paper, the set $I^1$ (respectively $I^2$) denotes the time {sets} on which the branch 1 (resp. the branch 2) is active in the mesoscopic model. The sets $I^1$ and $I^2$ form a partition of $\R$, each $I^k$, $k=1,2$, being periodic and of period $1$ and locally the union of a finite number of intervals: 
\be\label{hypIj}
\begin{array}{c}
I^1\cup I^2 = \R, \; I^1\cap I^2 =\emptyset, \\
 \text{$I^j$ is periodic of period $1$ and consists locally in a finite number of intervals, $j=1,2$.}
\end{array}
\ee
The flux limiter in the mesoscopic model is a time dependent map $A:\R\to \R$, such that
\be\label{hypfluxlimiterA}
\begin{array}{c}
\text{$A:\R\to \R$ is piecewise constant, periodic of period $1$ and such that}\\
0\leq A(t)\leq \left\{\begin{array}{ll}
\min\{f^0_{\max}, f^1_{\max}\} & {\rm on}\; I^1,\\
\min\{f^0_{\max}, f^2_{\max}\}  & {\rm on}\; I^2.
\end{array}\right.
\end{array}
\ee

\subsection{Entropy pairs and germs}\label{subsec.germ}

We now introduce the notion of germs, following \cite{AKR11,MFR22}. Germs define the junction conditions and play a central role in this paper. Let us recall that the pair (entropy, entropy flux) is given, for $p,\bar p \in \R$, by
$$
\eta(\bar p,p)=|p-\bar p| , \qquad q^j(\bar p,p) = \mbox{sign}(p-\bar p)(f^j(p)-f^j(\bar p)).
$$

We define the box
\begin{equation}\label{eq::e**1}
Q:=[a^0,c^0]\times [a^1,c^1]\times [a^2,c^2]
\end{equation}
and the subset of $Q$ satisfying Rankine-Hugoniot condition
\begin{equation}\label{eq::e**2}
Q^{RH}:=\left\{P=(p^0,p^1,p^2)\in Q,\quad f^0(p^0)=f^1(p^1)+f^2(p^2)\right\}
\end{equation}
\begin{Definition}\label{defi:1}{\bf (dissipation, germ, maximality)}\\
\noindent {\bf i) (Dissipation)}\\
For $P=(p^0,p^1,p^2)$, $\bar P=(\bar p^0,\bar p^1,\bar p^2) \in Q$, we define the dissipation by
$$D(\bar P,P):= q^0(\bar p^0,p^0)-\left\{q^1(\bar p^1,p^1)+q^2(\bar p^2,p^2)\right\}=\mbox{IN}-\mbox{OUT}$$
\noindent {\bf ii) (Germ)}\\
Consider a set $G\subset Q$. We say that $G$ is a germ (for dissipation $D$) if
$$\left\{\begin{array}{ll}
G\subset Q^{RH}&\quad \mbox{\bf (Rankine-Hugoniot)}\\
D(\bar P,P)\ge 0\quad \mbox{for all}\quad \bar P,P\in G &\quad \mbox{\bf (dissipation)}\\
\end{array}\right.$$
\noindent {\bf iii) (Maximal set)}\\
Let $G\subset Q$ be a set. We say that $G$ is maximal (for the dissipation $D$ relatively to the box $Q$) if for every $P\in Q$, we have
$$\left(D(\bar P,P)\ge 0\quad \mbox{for all}\quad \bar P\in G\right)\quad \Longrightarrow\quad P\in G.$$
\end{Definition}

\subsection{The mesoscopic problem}\label{sec-meso}

We are interested in a problem with one incoming branch and two outgoing ones; a periodic traffic light regulates the traffic, dispatching the vehicles in one of the two exit branches, slowing down the traffic or stopping it at the junction. On the time-intervals $I^1$, cars coming from road $0$ can enter road $1$ only, while on the time-intervals $I^2$ cars coming from road $0$ can enter road $2$ only. The traffic can also be limited on the junction by the flux limiter $A$, which is time dependent, {but piecewise constant.} For instance,   time intervals on which $A(t)=0$ correspond to periods where the traffic light stops completely the traffic at the junction.  
$$\left\{\begin{array}{l}
\left.\begin{array}{l}
\mbox{traffic from branch $0$ to branch $1$ with limiter $A(t)$,}\\
\mbox{no traffic entering in branch $2$}\\
\end{array}\right\} \quad \mbox{on the time-interval $I^1$,}\\
\\
\left.\begin{array}{l}
\mbox{traffic from branch $0$ to branch $2$  with limiter $A(t)$,}\\
\mbox{no traffic entering in branch $1$}\\
\end{array}\right\} \quad \mbox{on the time-intervals $I^2$.}
\end{array}\right.$$

\begin{figure}[!ht]
\begin{center}
\includegraphics[width=0.7\textwidth]{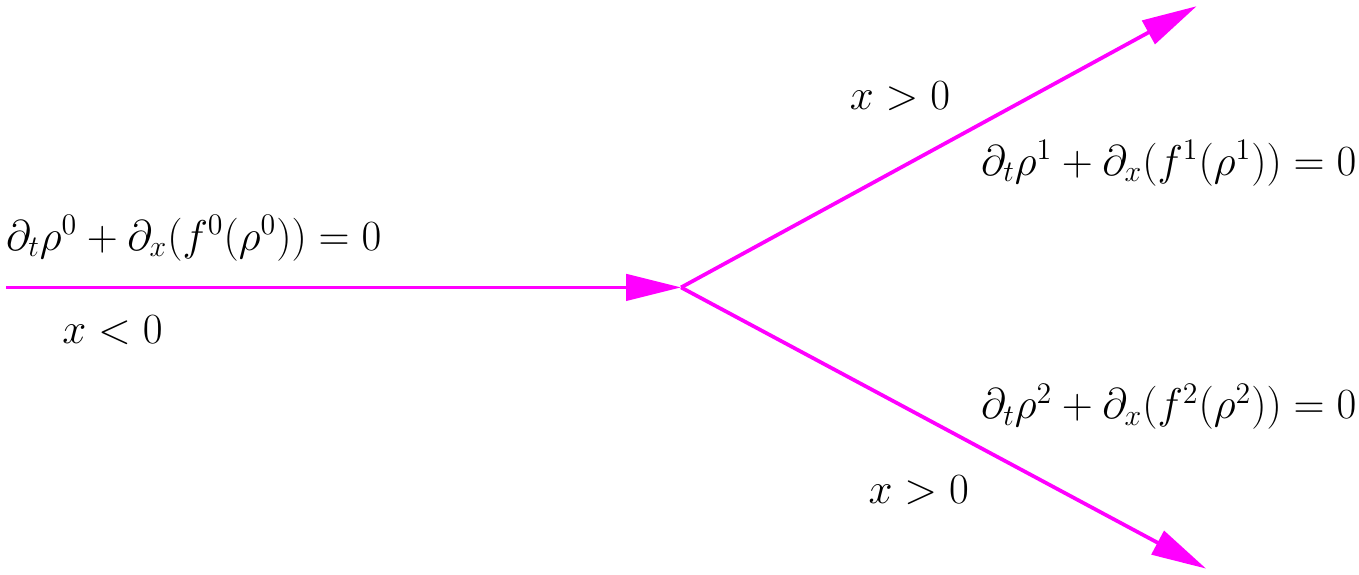}
\caption{Divergent 1:2 junction}
\label{Fdiv}
\end{center}
\end{figure}

Let $\rho^{j}$ ($j=0,1,2$) be the density of vehicles. Then $\rho=(\rho^0,\rho^1,\rho^2)$ solves 
\be\label{eq.meso}
\begin{array}{lllll}
(i)&\rho^j\in [a^j,c^j] &\qquad \text{ a.e. on}\; &(0,\infty)\times \mathcal R^j, & j=0,1,2\\
(ii)& \partial_t \rho^j +\partial_x (f^j(\rho^j))= 0 &\qquad \text{ on}\; &(0,\infty)\times \mathcal R^j, & j=0,1,2 \\
(iii) & (\rho^0(t,0^-), \rho^1(t,0^+), \rho^2(t,0^+))\in \mathcal G(t) &\qquad \text{ for a.e.}\;  &t\in (0,\infty),&
\end{array}
\ee
where the time dependent germ  $t {\mapsto}\mathcal G(t)$ is the piecewise constant  set-valued map given by
\be\label{defGt}
\mathcal G(t) = \mathcal G_{\Lambda_k} (t) \qquad \text{on}\; I^k, \; k=1,2,
\ee
and
\be\label{germGL1bis}
{\mathcal G}_{\Lambda_1}(t)=\left\{P=(p^0,p^1,p^2)\in Q,\quad \left|\begin{array}{l}
f^2(p^2)=0\\
\min(A(t), f^{0,+}( p^0),f^{1,-}(p^1))=f^0(p^0)=f^1(p^1)
\end{array}\right.\right\}, 
\ee
\be\label{germGL2bis}
{\mathcal G}_{\Lambda_2}(t)=\left\{P=(p^0,p^1,p^2)\in Q,\quad \left|\begin{array}{l}
f^1(p^1)=0\\
\min(A(t), f^{0,+}(p^0),f^{2,-}(p^2))=f^0(p^0)=f^2(p^2)
\end{array}\right.\right\}.
\ee
Recall that the assumption on the time intervals $I^k$, $k=1,2$, and the flux limiter $A$ are given in \eqref{hypIj} and \eqref{hypfluxlimiterA} respectively. 
The  notation $\mathcal G_{\Lambda_k}$ is justified in Section \ref{sec:germ} below, where we also explain that the ${\mathcal G}_{\Lambda_k}(t)$ are maximal germs  for each $t\in \R$ (Lemma \ref{lem::c15}).  The germ{s} ${\mathcal G}_{\Lambda_1}(t)$ and ${\mathcal G}_{\Lambda_2}(t)$ are very natural from a traffic flow point of view. Indeed,  during the {time-interval $I^1$ (for instance)}, the flux on the road $2$ is null and we consider only a 1:1 junction between the incoming road $0$ and the outgoing road $1$. In this situation, the description of the germ is well understood and take the above form (see \cite{CFarma} for the derivation of the junction condition in terms of Hamilton-Jacobi equations and \cite{CFGM} for a reformulation in term of scalar conservation laws). Let us recall that, for any $j=0,1,2$, the $L^\infty$ map $\rho^j$, being a solution to the scalar conservation law $\partial_t \rho^{j} +\partial_x (f^{j}(\rho^{j}))= 0$, has a strong trace (see Theorem \ref{th::r7homo}) at $x=0$ in the sense of Panov \cite{Ps07},  because the fluxes are strongly concave in the sense of (\ref{hypflux}).
\medskip

We say that a function $v$ is a standard Krushkov entropy solution of $\partial_t v+\partial_x(f(v))=0$ on $(0,+\infty)_t\times (0,+\infty)_x$ with initial condition $\bar v$, if for every $C^1_c([0,+\infty)_t\times (0,+\infty)_x)$ function $\varphi\ge 0$, we have
$$\int_{(0,+\infty)_t} \ \int_{(0,+\infty)_x} \ |v-c| \varphi_t + \left\{\mbox{sign}(v-c)\right\}\cdot (f(v)-f(c)) \varphi_x + \int_{\left\{0\right\}\times (0,+\infty)_x} |\bar v-c|\varphi \ge 0\quad \mbox{for all}\quad c\in \R$$

The next lemma states that equation \eqref{eq.meso} is well-posed and defines  a semigroup of contraction in $L^1$. 

\begin{Lemma}{\bf (Existence, uniqueness, $L^1$-contraction on the junction)}\label{lem.existeqmeso}\\ 
Assume \eqref{hypflux}, \eqref{hypIj} and \eqref{hypfluxlimiterA}. Given an initial condition $\bar \rho= (\bar \rho^{j})_{j=0, 1,2}$ in $L^\infty(\mathcal R)$ with $\bar \rho^{j} \in [a^{j}, c^{j}]$ a.e., there exists a unique entropy solution to \eqref{eq.meso}, in the sense that $\rho^{j}$ is a standard Krushkov entropy solution of $\partial_t \rho^{j} +\partial_x (f^{j}(\rho^{j}))= 0 $ on $(0,\infty)\times \mathcal R^{j}$ with $\rho^{j}(0,\cdot)=\bar \rho^{j}$ a.e.,
and such that the traces $(\rho^0(t,0^-), \rho^1(t,0^+), \rho^2(t,0^+))$ belong to the set $\mathcal G(t)$ for a.e. $t\in (0,\infty)$. 

In addition, if $\rho$ is a solution  to \eqref{eq.meso} associated with the initial condition $\bar \rho$ and $\rho_1$ is a solution to \eqref{eq.meso} associated with the initial condition $\bar \rho_1$, then Kato's inequality holds: 
\be\label{Kato}
\sum_{j=0}^2 \int_0^\infty \int_{\mathcal R^{j}} |\rho^{j}- \rho^{j}_1|\phi^{j}_t+ \left\{\mbox{sign}(\rho^{j}- \rho^{j}_1)\right\}\cdot (f^{j}(\rho^{j})- f^{j}(\rho_1^{j})) \partial_x\phi^{j} + \sum_{j=0}^2\int_{\mathcal R^{j}} |\bar \rho^{j}- \bar \rho^{j}_1| \phi^{j}(0,x) \geq 0
\ee
for any continuous nonnegative test function $\phi:[0,\infty)\times \mathcal R\to [0,\infty)$ with a compact support and such that $\phi^j:=\phi_{|[0,+\infty)\times (\mathcal R^j\cup \left\{0\right\})}$ is $C^1$ for any $j=0,1,2$.
\end{Lemma}

The proof of Lemma \ref{lem.existeqmeso} is postponed to Subsection  \ref{proofs1:2}. Let us underline that equation \eqref{eq.meso} almost fits the usual existence and uniqueness framework of conservation laws on a junction, as discussed in \cite{AKR11}, as only one outgoing branch is active at any time. 

\subsection{The macroscopic problem} 

We expect the limit problem to be of the same form as the mesoscopic problem, but with an autonomous germ $\mathcal G$. The limit scalar conservation law should take the form:  
\be\label{eq.macro}
\begin{array}{lllll}
(i) & \rho^j\in [a^j,c^j]&\qquad \text{ a.e. on}\; &(0,\infty)\times \mathcal R^j, &\; j=0,1,2,\\
(ii) & \partial_t \rho^j +(f^j(\rho^j))_x= 0 &\qquad \text{ on}\; &(0,\infty)\times \mathcal R^j, &\; j=0,1,2,\\
(iii) & (\rho^0(t,0), \rho^1(t,0), \rho^2(t,0))\in \mathcal G& \qquad \text{ for a.e.}\;  &t\in (0,\infty),&
\end{array}
\ee

Here the set $\mathcal G$ is the limit germ and is the main unknown of our problem. We now define the notion of solution for equation \eqref{eq.macro}, following \cite{FMR22b,MFR22}.  

\begin{Definition}{\bf (Entropy solution of (\ref{eq.macro}))}\label{def.solwithgerm}\\ 
Given a maximal germ $\mathcal G\subset Q$ and an initial condition $\bar \rho\in L^\infty(\mathcal R)$ such that $\bar \rho^j\in [a^j, c^j]$ a.e. for $j=0,1,2$, we say that a map  $\rho\in L^\infty((0,\infty)\times \mathcal R)$ is an entropy solution of \eqref{eq.macro} if, for any $j=0,1,2$, $\rho^j$ is a Kruzkhov entropy solution of \eqref{eq.macro}-(ii) on $\mathcal R^j$, if {its trace at $t=0$ is $\bar \rho$} and if, its trace $\rho(\cdot,0)=(\rho^0(\cdot,0^-),\rho^1(\cdot,0^+),\rho^2(\cdot,0^+))$ at $x=0$ belongs to $\mathcal G$: 
$$
\rho(t,0)\in \mathcal G\qquad a.e. \; t\geq0.
$$
\end{Definition} 

Following \cite{FMR22b,MFR22}, and because the germ $\mathcal G$ is maximal, the last condition in Definition \ref{def.solwithgerm} is equivalent to the following entropy inequality: 
$$
\sum_{j=0}^2\left\{ \int_0^\infty \int_{\mathcal R^j} \eta(u^j-\rho^j)\partial_t \phi^j+ q^j(u^j,\rho^j) \partial_x\phi^j +\int_{\mathcal R^j} \eta(u^j,\bar \rho^j)\phi^j(0,x)\right\} \geq 0
$$
for any $u=(u^j) \in \mathcal G$ and any continuous nonnegative test function $\phi:[0,\infty)\times \mathcal R\to [0,\infty)$ with a compact support and such that $\phi^j:=\phi_{|[0,+\infty)\times (\mathcal R^j\cup \left\{0\right\})}$ is $C^1$ for any $j=0,1,2$. 

Let us also point out that the  entropy solution  $\rho$ of \eqref{eq.macro} is in $C^0([0,+\infty), L^1_{loc}(\mathcal R))$: this is an easy consequence of the classical continuity in $L^1_{loc}$ of bounded entropy solution of scalar conservation laws on the line (see \cite[Theorem 6.2.2, {Lemma 6.3.3}]{Da05}) and of finite speed of propagation arguments. 

\subsection{Main result: the homogenization}

We are interested in the homogenization of \eqref{eq.meso}. Namely, given an initial condition $\bar \rho_0$, we want to understand the behavior as $\ep\to0$ of  the solution $\rho^\ep= (\rho^{\ep,0},\rho^{\ep,1},\rho^{\ep,2})$ to
\be\label{eq.mesoep}
\begin{array}{lllll}
(i) & \rho^{\varepsilon,j}\in [a^j,c^j] &\qquad \text{ a.e. on }\; &(0,\infty)\times \mathcal R^j, &\; j=0,1,2\\
(ii) & \partial_t \rho^{\ep,j} +\partial_x(f^j(\rho^{\ep,j}))= 0 &\qquad \text{ on}\; &(0,\infty)\times \mathcal R^j, &\; j=0,1,2\\
(iii)& (\rho^{\ep,0}(t,0), \rho^{\ep,1}(t,0), \rho^{\ep,2}(t,0))\in \mathcal G(t/\ep)&\qquad \text{ for a.e.}\;  &t\in (0,\infty),&\\
(iv)& \rho^\ep(0,\cdot)= \bar \rho_0 &\qquad \text{ on}\;  &\left\{0\right\}\times \mathcal R,&\\
\end{array}
\ee

{ Our main homogenization result is the following:
\begin{Theorem}{\bf (Homogenization of the 1:2 junction)}\label{thm.main}\\ 
Assume that \eqref{hypflux}, \eqref{hypIj} and \eqref{hypfluxlimiterA} hold.
Then there exists a maximal germ $\mathcal G_{\bar{\Lambda}}\subset Q$, such that the following holds true.
Let the initial data $\bar \rho_0=(\bar \rho_0^i)\in L^\infty(\mathcal R)$ be such that $\bar \rho^i_0\in [a^i,c^i]$ a.e. for $i=0,1,2$. 
Then the solution $\rho^\ep$ of (\ref{eq.mesoep}) converges  in $L^1_{loc}([0,\infty)\times \mathcal R)$ to the unique entropy solution $\rho$  to 
\be\label{eq.homo}
\begin{array}{lllll}
(i) & \rho^j \in [a^j,c^j]  &\qquad \text{ a.e. on}\; &(0,\infty)\times \mathcal R^j,& \; j=0,1,2 \\
(ii)& \partial_t \rho^j +\partial_x(f^j(\rho^j))= 0 &\qquad \text{ on}\; &(0,\infty)\times \mathcal R^j, & \; j=0,1,2\\
(iii)& (\rho^0(t,0), \rho^1(t,0), \rho^2(t,0))\in \mathcal G_{\bar{\Lambda}} &\qquad \text{ for a.e.}\;  &t\in (0,\infty),&\\
(iv)& \rho(0,\cdot)= \bar \rho_0&\qquad \text{ on}\;  &\left\{0\right\}\times \mathcal R,&\\
\end{array}
\ee
\end{Theorem}}

Let us point out that Theorem \ref{thm.main} itself implies the existence of a solution to \eqref{eq.homo}, which is not obvious otherwise. This shows in particular that the germ $\mathcal G_{\bar \Lambda}$ is complete in the terminology of \cite{AKR11,MFR22}. The germ $\mathcal G_{\bar \Lambda}$ is described in Subsection \ref{subsec.germmacro}.

In order to prove the theorem, we need to build suitable correctors of the equation, associated to elements of the germ. For this, the point is that we will not have to do it for all elements of the germ $\mathcal G_{\bar \Lambda}$, but only for a subset of it (which will indeed determine the whole germ $\mathcal G_{\bar \Lambda}$, as we will see later on). This subset, denoted by $E_{\bar \Lambda}$, is called a characteristic subgerm and is given in the following expression (where the continuous, nondecreasing maps $p^0\to \hat p^j_{p^0}$ for $j=1,2$ are introduced in \eqref{defpkp0BIS}):
\be\label{defG0}
\begin{array}{ll}
E_{\bar \Lambda}:=  & \Bigl\{(p^0,p^1,p^2)\in Q^{RH} \; \text{such that one of the following conditions holds:}\\
& \quad  \text{(i)} \;  p^j=\hat p^j_{p^0} , \; j=1,2,\; f^0(p^0)=f^{0,+}(p^0)\leq \int_0^1 A(t)dt,\\
& \quad \text{(ii)}\; p^2=c^2, \; f^0(p^0)= f^{0,-}(p^0)= \int_0^1{\bf 1}_{I^1}(t)A(t)dt = f^1(p^1)=f^{1,+}(p^1),  \\
& \quad \text{(iii)}\; p^1=c^1, \; f^0(p^0)= f^{0,-}(p^0)= \int_0^1{\bf 1}_{I^2}(t)A(t)dt = f^2(p^2)=f^{2,+}(p^2),\\
&\quad  \text{(iv)}\; p^j=c^j, \;j=0,1,2\;  \Bigr\}.
\end{array}
\ee

Case $(i)$ corresponds to a situation in which the traffic is fluid on all branches at the macroscopic level, and fluid on the exit branches at the mesoscopic level. In case $(ii)$, the outgoing branch 2 is completely congested and the traffic is stopped on this branch. The traffic reduces to a classical {1:1} junction, the only difficulty being that the traffic is congested at the macroscopic level on the incoming branch and fluid (but saturated by the flux limiter $A$) on the {outgoing} {branch 1}. Case $(iii)$ is symmetric, exchanging the role of the outgoing roads. The last case, Case $(iv)$, is particularly simple since it corresponds to a situation in which the traffic is completely congested (and the velocity of the traffic is null everywhere).  

The following lemma states that the germ $\mathcal G_{\bar{\Lambda}}$ is a sort of closure of $ E_{\bar \Lambda}$: 

\begin{Lemma}{\bf ($E_{\bar \Lambda}$ generates $\mathcal G_{\bar \Lambda}$)}\label{lem.reduction}\\ 
Assume that \eqref{hypflux}, \eqref{hypIj} and \eqref{hypfluxlimiterA}  hold. We have $E_{\bar \Lambda}\subset \mathcal G_{\bar{\Lambda}}$ and $E_{\bar \Lambda}$ generates $\mathcal G_{\bar{\Lambda}}$: namely, for any $U\in Q$, 
$$
\Bigl( \; D(U,\bar U) \geq 0 \qquad \forall \bar U\in E_{\bar \Lambda}\;\Bigr) \qquad \Longrightarrow \qquad U\in \mathcal G_{\bar{\Lambda}} . 
$$
\end{Lemma}

The two main ingredients of the proof of Theorem \ref{thm.main} are the correct guess of the effective germ $\mathcal G_{\bar \Lambda}$ (with its generation property given in Lemma \ref{lem.reduction}) and  the construction of a corrector for each element {of} $ E_{\bar \Lambda}$:

\begin{Theorem}{\bf (Existence of correctors with prescribed values at infinity)}\label{thm.corrector}\\ 
Assume that \eqref{hypflux}, \eqref{hypIj} and \eqref{hypfluxlimiterA}  hold. For any $p=(p^0,p^1,p^2)\in E_{\bar \Lambda}$, there exists an entropy solution $u_p=(u_p^i)\in L^\infty(\R\times \mathcal R)$ of \eqref{eq.meso} which is $1$-periodic in time and a constant $C>0$ such that for all $M\geq C$
\be\label{keypptcorr}
 \|u^0_p-p^0\|_{L^\infty(\R\times (-\infty,-M))}+ \|u^i_p-p^i\|_{L^\infty(\R\times (M,\infty))}\leq CM^{-1},\qquad  i=1,2. 
\ee
If, in addition,  $p$ is as in (i) in the definition \eqref{defG0} of  $E_{\bar \Lambda}$, then 
$$
u^0_p= p^0 \; \text{on} \; \R\times (-\infty,-C).
$$
\end{Theorem} 

The definition of the germ $\mathcal G_{\bar \lambda}$, the proof of its maximality as well as the proof of Lemma \ref{lem.reduction} are given in Subsection \ref{subsec.germmacro}. 
The proofs of Theorem \ref{thm.main} (convergence part) and Theorem \ref{thm.corrector} are postponed to the last section (Subsection \ref{proofs1:2}).

\subsection{Homogenization for 2:1 junctions} \label{subsec.2:1}

We complete the section by the analysis of homogenization on 2:1 junctions: as already pointed out, this case is more realistic in terms of applications. The junction is now described by the two incoming branches $\check{\mathcal R}^j=(-\infty,0)\times \{j\}$, $j=1,2$, and the outgoing branch $\check{\mathcal R}^0=(0,\infty)\times \{0\}$. We set $\check{\mathcal R}= \bigcup_{j=0}^2 \check{\mathcal R}^j\cup\{0\}$. 

The mesoscopic model we are interested in concerns a junction with a periodic traffic light which regulates the traffic. As before the time-interval $\R$ is split into the $1-$periodic sets $I^1$ and $I^2$, each $I^k$ consisting locally in a finite number of intervals. On the time-intervals $I^1$, only cars coming from road $1$ are allowed to enter the junction and the road $0$, while on the time-intervals $I^2$ only cars coming from road $2$ can enter road $0$. The traffic is also  limited on the junction by a flux limiter $A=A(t)$. To summarize:  
$$\left\{\begin{array}{l}
\left.\begin{array}{l}
\mbox{traffic from branch $1$ to branch $0$ with flux limiter $A(t)$,}\\
\mbox{no traffic exiting  branch $2$}\\
\end{array}\right\} \quad \mbox{on the time intervals $I^1$,}\\
\\
\left.\begin{array}{l}
\mbox{traffic from branch $2$ to branch $0$ with flux limiter $A(t)$,}\\
\mbox{no traffic exiting  branch $1$}\\
\end{array}\right\} \quad \mbox{on the time intervals $I^2$}
\end{array}\right.$$
(see figure \ref{Fconv}).  
\begin{figure}[!ht]\label{fig2}
\begin{center}
\includegraphics[width=0.7\textwidth]{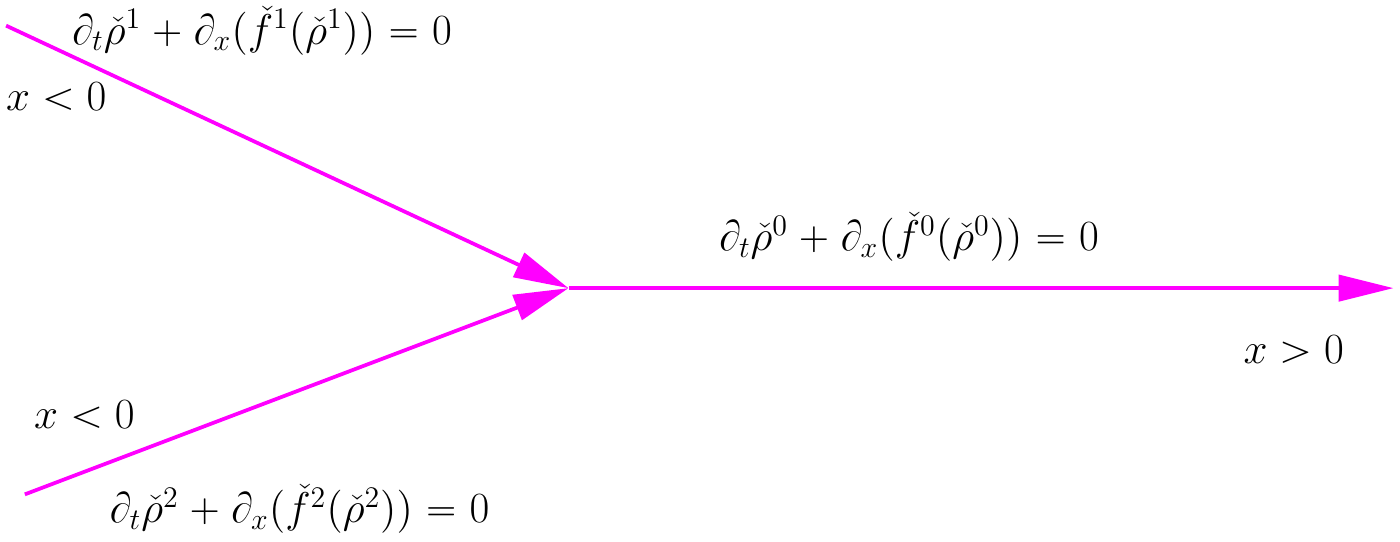}
\caption{Convergent 2:1 junction}
\label{Fconv}
\end{center}
\end{figure}
We fix $\ep>0$ a scaling parameter. 
In this model the scaled densities $\check \rho^\ep= (\check \rho^{\ep,0}, \check \rho^{\ep,1}, \check \rho^{\ep,2})$ solve the conservation law:  
\begin{equation}\label{eq::c20bis}
\left\{\begin{array}{llll}
\check \rho^{\ep,j}\in [\check a^j, \check c^j]&\quad \mbox{a.e on.}&\quad  (0,+\infty)\times \check {\mathcal R}^j,&\quad j=0,1,2\\
\partial_t \check \rho^{\ep,j}+\partial_x (\check f^j(\check \rho^{\ep,j}))=0&\quad \mbox{on}&\quad (0,+\infty)\times \check {\mathcal R}^j,&\quad j=0,1,2\\
(\check \rho^{\ep,0}(t,0^+),\check \rho^{\ep,1}(t,0^-),\check \rho^{\ep,2}(t,0^-))\in \check {\mathcal G}(t/\ep)&\quad \mbox{for a.e.}&\quad t\in (0,+\infty),&\\
\check \rho^\ep(0,\cdot)= \check{\bar \rho} &\quad \mbox{on}&\quad \left\{0\right\}\times \check {\mathcal R}.&
\end{array}\right.
\end{equation}
The fluxes $\check f^j$ satisfy condition (\ref{hypflux}) with $\check a^j,\check b^j,\check c^j$ in place of $a^j,b^j,c^j$, and $\check f^{j,\pm}$ are defined similarly as in (\ref{eq::e*2}), (\ref{eq::e*3}). The time periodic maximal germ $\check{\mathcal G}$ of period equal to $1$ is given by
\begin{equation}\label{eq::c21bis}
\check{\mathcal G}(t):=\left\{\begin{array}{llll}
\check {\mathcal G}^1(t) &\quad \mbox{on}\quad I^1\\
\check {\mathcal G}^2(t) &\quad \mbox{on}\quad I^2\\
\end{array}\right.
\end{equation}
and 
$$
\check{\mathcal G}^1(t)= \{ (p^0,p^1,p^2)\in Q, \;  \check f^2(p^2)=0,\quad \min\left\{A(t),\check f^{1,+}(p^1),\check f^{0,-}(p^0)\right\}=\check f^1(p^1)=\check f^0(p^0)\},$$
$$\check{\mathcal G}^2(t)= \{ (p^0,p^1,p^2)\in Q, \;  \check f^1(p^1)=0,\quad \min\left\{A(t), \check f^{2,+}(p^2),\check f^{0,-}(p^0)\right\}=\check f^2(p^2)=\check f^0(p^0)\}.
$$
As in the previous parts, $I^1$ and $I^2$ form a partition of $\R$ satisfying \eqref{hypIj}, and the flux limiter $A:\R\to \R$ is a periodic, piecewise constant map such that \eqref{hypfluxlimiterA} holds. Finally the initial condition $\check{\bar \rho}= (\check{\bar \rho}^j)\in L^\infty(\check{\mathcal R})$ satisfies $\check{\bar \rho}^j\in [\check a^j, \check c^j]$ a.e.{.}

\begin{Theorem}{\bf (Homogenization of the 2:1 junction)}\label{thm.main2:1}\\ 
Under the previous assumptions,
for any $\ep>0$ there exists a unique entropy solution to \eqref{eq::c20bis} and, as $\ep\to 0^+$ the  solution $(\check \rho^\ep)$  to \eqref{eq::c20bis} converges in $L^1_{loc}({[0,\infty)}\times \mathcal R)$ to the unique entropy solution $\check \rho$ of the homogenized problem 
\be\label{eq.macro2:1}
\left\{\begin{array}{llll}
\check \rho^j\in [\check a^j,\check c^j]&\quad \mbox{a.e. on}&\quad (0,+\infty)\times \check {\mathcal R}^j,&\quad j=0,1,2\\
\partial_t \check \rho^j+\partial_x (\check f^j(\check \rho^j))=0&\quad \mbox{on}&\quad (0,+\infty)\times \check {\mathcal R}^j,&\quad j=0,1,2\\

(\check \rho^0(t,0^+),\check \rho^1(t,0^-),\check \rho^2(t,0^-))\in {\mathcal G}_{\check f,\bar \Lambda}^-& \quad \mbox{for a.e.}&\quad t\in (0,+\infty),&\\
\check \rho(0,\cdot)= \check{\bar \rho} &\quad \mbox{on}&\quad \left\{0\right\}\times \check {\mathcal R}.&
\end{array}\right.
\ee
where the maximal germ $ {\mathcal G}_{\check f,\bar \Lambda}^-$ is defined explicitly in \eqref{eq::c1bis} below with $\bar \Lambda$ given in Subsection  \ref{subsec.germmacro}. 
\end{Theorem}

The proof of {this} theorem is given in Subsection \ref{proof2:1}.

\subsection{Review of the literature}

Conservation laws (CL) on junctions (and their application to traffic flows) have attracted a lot of attention: see for instance the monograph \cite{GP1}  and the survey paper \cite{BCGHP}. A large part of the literature is concerned with conservation laws on 1:1 junctions, involving one flux function for the incoming road and a possibly  different one on the outgoing road, see  \cite{AGDV11, AKR11, AuPe05, BaJe97, GNPT07, To00, To01}.
It turns out that the approach through the germ theory for 1:1 junctions is strongly linked with Hamilton-Jacobi (HJ) equations on such  junctions (still in the 1:1 case, see \cite{CFGM}). Combining both approaches gives a rough picture of this 1:1 setting: in a nutshell, the junction condition reduces to a flux limiter (a scalar),  the conservation law is an $L^1-$contraction and is equivalent to the HJ approach {at the level of the antiderivative. Let us also underline that the Hamilton-Jacobi equation possesses itself an $L^\infty-$contraction property.} In conclusion, this 1:1 framework is now relatively well understood.

The situation is completely different for junctions involving at least 3 branches. Indeed, although many works have been devoted to such junctions (see for instance \cite{ACD, FMR22b, GP2, GR20, H22, MFR22, To20}), the problem is still poorly understood and the general picture is far from clear. For instance, if the germ approach of \cite{AKR11}  has been recently extended to general junctions in \cite{FMR22b, MFR22} (and we strongly use this extension in the paper), there are still few examples of germs which are  maximal and complete; one of the outcome of our paper is to describe a new class of such germs (note however that a particular case was previously discussed in \cite{To20}).  On the other hand,  models involving more than 3 branches seem far richer than the 1:1 set-up: for instance our junction condition (in terms of germs) can be parametrized by a whole family of increasing functions (in contrast with the 1:1 set-up where there is just a single parameter). Another difference with the 1:1 setting is that 2:1 and 1:2 junctions are not always  $L^1-$contractions. And last, the equivalence between CL and HJ is lost in general: the limit models for 1:2 and 2:1 junctions discussed in this paper do not seem to fit a HJ framework. \medskip

It is interesting to compare our class of germs (that we call here the class of traffic light germs, TL-germs  in brief) with some {of} the known germs in the literature on junctions (see in particular \cite{GR20}).
We only consider 1:2 junctions because a reversed germ is automatically constructed for 2:1 junctions, by reversion transform.
In \cite{To20},  the author  defines a germ which is a special case of TL-germs for very special functions satisfying moreover $f^0_{max}=f^1_{max}+f^2_{max}$ with $\hat \lambda^j(\lambda)=\theta^j \lambda$ for $j=1,2$. In the pioneering work \cite{HR}, the authors introduced a class of germs, by the maximization of some entropy at the junction. It has been only very recently proved in \cite{H22} that those germs are $L^1$-contractant. We do not know what is the relationship between this class of germs and the class of TL-germs, even if the intersection of the two classes is empty or not.

The vanishing viscosity germ studied in \cite{ACD} can be either or not a TL-germ, depending on the flux functions.
For instance, for $f^0=f$, $f^1=\alpha_1 f$, $f^2=\alpha_2 f$, it is possible to show that the vanishing viscosity germ is a TL-germ if and only if $\alpha_1+\alpha_2\le 1$. 
\medskip

Hamilton-Jacobi germs (HJ-germs in brief) were  defined in \cite{IMZ} and studied in \cite{IM17}. These HJ-germs are the same (going from the HJ level to the level of conservation laws) as the ones defined previously in the monograph \cite{GP1} for divergent junctions, and a single ingoing road. These germs are a particular case of $\mathcal {R S}_2$ germs in \cite{GP2}, where the authors also show that the total variation of the fluxes is bounded by a constant if it is the case for the initial data. This allows them to show the existence of a solution. The uniqueness seems an open question in general (at least at the direct level of conservation laws). Notice that for $N\ge 3$ branches (like 1:2 junctions), it is easy to check that HJ-germs are never $L^1$-contractant germs {(see \cite{CFGM}).}

In the monograph \cite{GP1}, the authors  introduce in particular a germ for 2:1 junctions which is the same (by reversion) as the one called $\mathcal {R S}_1$ in the article \cite{GP2} for junctions 1:2. It is defined for $f^i =f$ for $i=0,1,2$, and it is possible to show that it is not in the class of what we call here TL-germs.
The existence of a solution is shown in \cite{GP2}, but the uniqueness seems open. We do not know if these germs have the $L^1$-contraction property or not.

\subsection{Organization of the paper}
In Section \ref{sec:germ},  we provide some key results concerning the germs discovered in this paper. Section \ref{sec:correctors} is devoted to the  construction of correctors. The proof of the main homogenization results, Theorem \ref{thm.main} and Theorem \ref{thm.main2:1}, are given in Section \ref{sec:homog}.

\section{Germs for divergent 1:2 junctions} \label{sec:germ}

In this section, we introduce a new general class of sets, prove that these sets are maximal germs, and show how the different germs encountered in the main results enter into this general framework. 

In contrast with the rest of the paper, in this section we only use a weaker assumption than (\ref{hypflux}), namely
\be\label{hypflux-bis}
\begin{array}{c}
\text{For $j\in\{0,1,2\}$, for $a^j<b^j<c^j$, the function $f^j:[a^j, c^j]\to \R$ is continuous,}\\
\text{increasing on $[a^j,b^j]$ and decreasing on $[b^j,c^j]$, with $f^j(a^j)=f^j(c^j)=0$,}
\end{array}
\ee
and we use the same notation $f^{j,\pm}$ as defined in  (\ref{eq::e*2}), (\ref{eq::e*3}). We start the section with a description of the general class of germs used throughout the paper and explain their main properties. We illustrate this notion by showing that the germs introduced for the mesoscopic model do fit  this general framework. Then we present the germs found through the homogenization procedure and give several examples. We complete the section by the proof of the main properties of our class of germs.

\subsection{A general family of germs}

\subsubsection{The main result on germs}

In this section  we investigate a general class of germs on 1:2 junctions. This family is described through  a set of parameters
$$
\Lambda=\Bigl\{ \bar \lambda^0, \bar \lambda^1, \bar \lambda^2, \hat \lambda^1, \hat \lambda^2\Bigr\}
$$ 
satisfying the following conditions
\begin{equation}\label{eq::c12}
\left\{\begin{array}{ll}
\bar \lambda^j\in [0,f^j_{\max}] &\quad \mbox{for}\quad j=0,1,2\\
\bar \lambda^0=\bar \lambda^1+\bar \lambda^2&\\
\mbox{the maps}\quad \hat \lambda^k: [0,f^0_{\max}]\to [0,\bar \lambda^k]\quad \mbox{are continuous nondecreasing}&\quad  \mbox{for}\quad  k=1,2\\
\hat \lambda^k(0)=0,\quad \hat \lambda^k(\bar \lambda^0)=\bar \lambda^k&\quad \mbox{for}\quad k=1,2\\
\hat \lambda^1(\lambda)+\hat \lambda^2(\lambda)=\min(\lambda,\bar \lambda^0)&\quad \mbox{for}\quad \lambda\in [0,f^0_{\max}].
\end{array}\right.
\end{equation}

The germ ${\mathcal G}_\Lambda$ is defined from $\Lambda$ as follows:
\begin{equation}\label{eq::c1}
{\mathcal G}_\Lambda:={\mathcal G}_{f,\Lambda}=\left\{P=(p^0,p^1,p^2)\in \R^3,\quad \left|\begin{array}{ll}
a^j\le p^j\le c^j ,&\quad j=0,1,2\\
\\
0\le f^j(p^j)\le\bar \lambda^j ,&\quad j=0,1,2\\
\\
f^0(p^0)=f^1(p^1)+f^2(p^2)&\\
\\
f^{k,+}(p^k)\ge \hat \lambda^k(f^{0,+}(p^0)),&\quad k=1,2\\
\end{array}\right.\right\}.
\end{equation}

\begin{Theorem}\label{th:1}{\bf (Germ for divergent 1:2 junction)}
Under assumptions (\ref{hypflux-bis}) and  (\ref{eq::c12}),  let us consider the set ${\mathcal G}_\Lambda$ defined in (\ref{eq::c1}).
Then 
\begin{itemize}
\item[(i)]  ${\mathcal G}_\Lambda$ is a maximal germ, 
\item[(ii)] ${\mathcal G}_\Lambda$ is determined by its subset 
\begin{equation}\label{eq::e100}
E^+_{\Lambda}:=\Gamma\cup \left\{P_1,P_2,P_{3}\right\},
\end{equation}
where the curve $\Gamma$ and the points $P_1,P_2,P_3$ are defined below in (\ref{eq::c5}) and (\ref{eq::c6}) respectively. This means that, for any $P\in Q$, 
$$
\Bigl[ D(\bar P, P) \geq 0\qquad \forall \bar P\in E^+_{\Lambda} \Bigr] \qquad \Longrightarrow \qquad P\in {\mathcal G}_\Lambda.
$$
\end{itemize}
\end{Theorem}

In order to describe the curves $\Gamma$ and the points $P_i$ (for $i=1,2,3$), let us first introduce the roots of $f^{j, \pm}(\cdot)=\lambda$ for $j=0,1,2$:
$$\left\{\begin{array}{l}
\left[a^j,b^j\right] \ni u^j_+(\lambda):=r \quad \mbox{such that}\quad f^{j,+}(r)=\lambda\in \left[0,f^j_{\max}\right]\\
\left[b^j,c^j\right]\ni u^j_-(\lambda):=r  \quad \mbox{such that}\quad f^{j,-}(r)=\lambda\in \left[0,f^j_{\max}\right].
\end{array}\right.$$
We will also use later the notation $u^j_\pm=(f^{j,\pm})^{-1}$.
Then 
\begin{equation}\label{eq::c5}
\Gamma:=\left\{P=(u^0_+(\lambda),u^1_+(\lambda^1),u^2_+(\lambda^2))\quad \mbox{with}\quad \lambda^k:=\hat \lambda^k(\lambda)
\quad \mbox{for}\quad k=1,2\quad \mbox{and}\quad \lambda\in [0,\bar \lambda^0] \right\}
\end{equation}
and 
\begin{equation}\label{eq::c6}
\left\{\begin{array}{l}
P_0:=(u^0_+(0),u^1_+(0),u^2_+(0))=(a^0,a^1,a^2)\in \Gamma\\
P_3:=(u^0_-(0),u^1_-(0),u^2_-(0))=(c^0,c^1,c^2)\\
\\
P_1:=(u^0_-(\bar \lambda^1),u^1_+(\bar \lambda^1),u^2_-(0))\\
P_2:=(u^0_-(\bar \lambda^2),u^1_-(0),u^2_+(\bar \lambda^2)).\\
\end{array}\right.
\end{equation}
Heuristically, the curve $\Gamma$ corresponds to a situation in which all the branches are fluids, while 
$$\left\{\begin{array}{llllll}
P_0=&(&\mbox{empty road},&\ \mbox{empty road},&\ \mbox{empty road}&)\quad \in \Gamma\\
P_3=&(& \mbox{fully congested},&\ \mbox{fully congested},&\ \mbox{fully congested}&)\\
\\
P_1=&(&\mbox{congested},&\ \mbox{fluid and saturated},&\ \mbox{fully congested}&)\\
P_2=&(&\mbox{congested},&\ \mbox{fully congested},&\ \mbox{fluid and saturated}&)
\end{array}\right.$$
where ``fully congested'' means that the road is with a maximal density of vehicles (hence with zero velocity). On the other hand, ``fluid and saturated'' means that the outgoing road is still fluid, but that we can not increase the flux passing through the junction point.\\

The proof of Theorem \ref{th:1} is postponed to Subsection \ref{sec:2.2}.\medskip

Let us now explain how the germs introduced for the mesoscopic model and the homogenized germ introduced for the macroscopic model fit into the framework just described. 

\subsubsection{Germs in the mesoscopic model}

We check here that the sets $\mathcal G_{\Lambda_k}(t)$ (for $t\in I^k$ and $k=1,2$) introduced in \eqref{germGL1bis} and \eqref{germGL2bis} respectively, are  of the form \eqref{eq::c1} for suitable sets $\Lambda_k(t)$. For $t\in I_1$, the set $\Lambda_1(t)=(\bar \lambda^0_1(t),\bar \lambda^1_1(t),\bar \lambda^2_1(t),\hat \lambda^1_1(t, \cdot),\hat \lambda^2_1(t, \cdot))$ is given by 
\begin{equation}\label{eq::c13-1}
\left\{\begin{array}{ll}
\bar \lambda^0_1(t)=\bar \lambda^1_1(t)=A(t), 
\quad \bar \lambda^2_1=0&\\
\hat \lambda^1_1(t,\lambda)=\min(\lambda,A(t))& \quad \mbox{for}\quad \lambda\in [0,f^0_{\max}]\\
\hat \lambda^2_1(t,\lambda)=0& \quad \mbox{for}\quad \lambda\in [0,f^0_{\max}].
\end{array}\right.
\end{equation}
For $t\in I^2$, the set $\Lambda_2(t)=(\bar \lambda^0_2(t),\bar \lambda^1_2(t),\bar \lambda^2_2(t),\hat \lambda^1_2(t,\cdot),\hat \lambda^2_2(t,\cdot))$ is defined symmetrically, exchanging the indices $1$ and $2$:
\begin{equation}\label{eq::c13-2}
\left\{\begin{array}{ll}
\bar \lambda^0_2(t)=\bar \lambda^2_2(t)=A(t),\quad \bar \lambda^1_2=0&\\
\hat \lambda^1_2(t,\lambda)=0& \quad \mbox{for}\quad \lambda\in [0,f^0_{\max}]\\
\hat \lambda^2_2(t,\lambda)=\min(\lambda,A(t))& \quad \mbox{for}\quad \lambda\in [0,f^0_{\max}].
\end{array}\right.
\end{equation}

The next lemma claims that the germs $\mathcal G_{\Lambda_k(t)}$ (for $k=1,2$) associated with the $\Lambda_k(t)$ through definition (\ref{eq::c1}), coincide precisely with the germs $\mathcal G_{\Lambda_k}(t)$ introduced in \eqref{germGL1bis} and \eqref{germGL2bis} respectively for the mesoscopic model: 
\begin{Lemma}\label{lem::c15}{\bf (Characterization of the maximal germs ${\mathcal G}_{\Lambda_k}$)}\\ For any $k=1,2$ and  any $t\in I^k$, the set ${\mathcal G}_{\Lambda_k(t)}$, defined through (\ref{eq::c1}) from the sets $\Lambda_k(t)$ is a maximal germs and coincides with the set ${\mathcal G}_{\Lambda_k}(t)$ introduced in \eqref{germGL1bis} (for $k=1$) and \eqref{germGL2bis} (for $k=2$).  
\end{Lemma}

\begin{proof} 
The proof is elementary. By symmetry, we can only do it for ${\mathcal G}_{\Lambda_1(t)}$ for $t\in I^1$. Notice that $\Lambda_1(t)$ satisfies (\ref{eq::c12}).
Hence $\mathcal G_{\Lambda_1(t)}$ is a maximal germ, from Theorem \ref{th:1}.\\
If $P=(p^k)_{k=0,1,2}$ belongs to  ${\mathcal G}_{\Lambda_1(t)}$ or to ${\mathcal G}_{\Lambda_1}(t)$, we have
$$\lambda:=f^0(p^0)=f^1(p^1)\in [0,\bar \lambda^0_1(t)]$$
and then
$$p^0\in \left\{u^0_\pm(\lambda)\right\},\quad p^1\in \left\{u^1_\pm(\lambda)\right\}.$$
This gives $2\times 2$ cases. Examining all cases in details (it is slightly tedious to do it for both expressions), we can check in both expressions that all cases are possible except the following case which is excluded by both expressions
$$p^0=u^0_-(\lambda),\quad p^1=u^1_+(\lambda) \quad \mbox{for}\quad \lambda\in [0,\bar \lambda^0_1(t)).$$
Hence the two expressions coincide and the {lemma} holds true. 
\end{proof}

\subsubsection{The homogenized germ in the macroscopic model}\label{subsec.germmacro}

We now turn to the homogenized germ. This germ is naturally associated with the correctors introduced in the next section. It happens however that it can be built independently: we present this construction here. We also give several examples in which the germ can be explicitly computed (Propositions \ref{prop.exmacrogerm}, \ref{lem::f21} and \ref{lem::f4}).  

The homogenized germ $\mathcal G_{\bar{\Lambda}}$ introduced in Theorem \ref{thm.main} is defined through the set of parameters 
$$
\bar{\Lambda}= \Bigl\{ \bar \lambda^0, \bar \lambda^1, \bar \lambda^2, \hat \lambda^1, \hat \lambda^2\Bigr\}
$$
by  relation \eqref{eq::c1} that we recall: 
\begin{equation}\label{eq::c1BIS}
{\mathcal G}_{\bar \Lambda}:=\left\{P=(p^0,p^1,p^2)\in Q^{RH},\quad \left|\begin{array}{ll}
0\le f^j(p^j)\le\bar \lambda^j ,&\quad j=0,1,2\\
\\
f^{k,+}(p^k)\ge \hat \lambda^k(f^{0,+}(p^0)),&\quad k=1,2\\
\end{array}\right.\right\}.
\end{equation}
In $\bar{\Lambda}$, the effective limiters $\bar \lambda^0, \bar \lambda^1, \bar \lambda^2$ are given by
\be\label{deflambdakkk}\left\{\begin{array}{l}
\displaystyle \bar \lambda^k:=\int_0^1 {\bf 1}_{I^k}(t) A(t)dt \quad \mbox{for}\quad k=1,2,\\
\;\\
\displaystyle \bar \lambda^0:=\int_0^1 A(t)dt =\bar \lambda^1+\bar \lambda^2 \le f^0_{\max}.
\end{array}\right.
\ee
For $\lambda\in [0, \bar \lambda^0]$, let $p^0= (f^{0,+})^{-1}(\lambda)$. Note that $p^0$ satisfies the inequality 
\be\label{condcondf0p0}
f^{0,+}(p^0)=f^0(p^0)\leq \int_0^1 A(t) dt =\bar \lambda^0.
\ee
We introduce the 1-periodic map\footnote{For simplicity we use the same expression  $F_\lambda$ and $F_{p^0}$ although the relationship between $\lambda$ and $p^0$ is the equality $p^0= (f^{0,+})^{-1}(\lambda)$: the first notation makes more sense in the present section, while the second one will be used throughout Section \ref{sec:correctors} on the construction of correctors.} $F_\lambda=F_{p^0} :\R\to [0,f^0_{\max}]$ as 
\be\label{defFp0}
\begin{array}{ll}
\forall t\in \R\qquad 
F_\lambda(t)= F_{p^0}(t) & = \displaystyle \left\{\begin{array}{ll}
\displaystyle \lambda= f^0(p^0) &\displaystyle \text{if } \frac{1}{t-t_1}\int_{t_1}^t A(s)ds \geq f^0(p^0),\; \forall t_1<t,\\
A(t) & \text{otherwise,}
\end{array}\right.
\end{array}
\ee
and set, for $k=1,2$,  
\be\label{defhatlambda}
\hat \lambda^k(\lambda)= \int_0^1 F_{\lambda}(t) {\bf 1}_{I^k}(t)dt .
\ee
We  extend the functions $\hat \lambda^k$ up to $f^0_{\max}$ by
$$
\hat \lambda^k(\lambda):=\bar \lambda^k\quad \mbox{for}\quad \lambda\in[ \bar \lambda^0, f^0_{\max}],\quad k=1,2.
$$
Finally, we set, for $k=1,2$,  
\be\label{defpkp0BIS}
\hat p^k_{p^0}= u^k_+(\hat \lambda^k(f^0(p^0))), \qquad \forall p^0\in [a^0,b^0]\; {\rm with} \; f^0(p^0)\leq \bar \lambda^0,
\ee
where we recall the notation $u^k_{\pm}= (f^{k, \pm})^{-1}$. 

The interpretation of these quantities  is the following: we  show in Lemma \ref{lem.w0} below that  $F_{p^0}$ is the flux at the junction $x=0$ of the $1-$periodic corrector taking value $p^0$ at $-\infty$ (or, equivalently, having a flux $\lambda=f^0(p^0)$ at $-\infty$). Proposition \ref{prop.correctorfluid}  shows that the $\hat p^k_{p^0}$ are  the densities at $+\infty$ and on the branch $k$ of this corrector. Hence the quantities $\hat \lambda^k(\lambda)$ are the fluxes at $+\infty$ of the  time periodic corrector with a flux $\lambda$ at $-\infty$. 

\begin{Remark}{{\bf (An obstacle problem)}}  The  flux  $F_\lambda$ at the junction can be recovered by an obstacle problem. More precisely, one can show that
$$F_\lambda=\lambda+A-(\Phi_\lambda)'\quad \mbox{a.e. on}\quad \R$$
where $\Phi_\lambda$ is a viscosity solution to the following obstacle problem
$$\min(\Phi_\lambda-B, (\Phi_\lambda)'-\lambda)=0\quad \mbox{on}\quad \R,\quad B'=A$$
such that $\Phi_\lambda-B$ is $1$-periodic. Moreover $\Phi_\lambda$ is unique for $\lambda\in [0,\bar \lambda^0)$ and we have
$$\{\Phi_\lambda=B\}\subset\{F_\lambda=\lambda\}\qquad{\rm and}\qquad \{\Phi_\lambda>B\}\subset\{F_\lambda=A\}.$$
Finally, we have the following representation for $\Phi_\lambda$:
\begin{equation}\label{eq::f16}
\Phi_\lambda(t):=\sup_{\tau\ge 0} \left\{B(t-\tau)+\lambda \tau\right\}.
\end{equation}
{See Lemma \ref{lem.analysew00} for related results on $\psi_{p^0}=\Phi_\lambda-B$.}
\end{Remark}

We then have the following properties
\begin{Lemma}\label{lem.hatlambdak}{\bf {(Properties of the fluxes $\hat \lambda^k$)}}\\
For $k=1,2$, $\bar \lambda^k\leq f^k_{\max}$ and the map $\lambda {\mapsto} \hat \lambda^k(\lambda)$ is continuous and nondecreasing on $ [0, f^0_{\max}]$ with 
\be\label{aezlqkjrsd}
 \hat \lambda^1(\lambda)+  \hat \lambda^2(\lambda)=\lambda \qquad \forall \lambda\in [0,\bar \lambda^0]
 \ee
and
\be\label{aezlqkjrsdBIS}
 0\le \hat \lambda^k(\lambda)\le \bar \lambda^k=\hat \lambda^k(\bar \lambda^0), \qquad k=1,2, \;  \forall \lambda\in [0,\bar \lambda^0].
\ee
\end{Lemma}

\begin{proof}{\bf Step 0: preliminaries.} Let us first note for later use that 
\be\label{FleqA}
F_{p^0}(t)\leq A(t) \qquad {\rm a.e.}.
\ee
Indeed, let $t$ be a point of continuity of $A$. Then either $F_{p^0}(t)= A(t)$, or $\frac{1}{t-t_1}\int_{t_1}^t A(s)ds \geq f^0(p^0)$ for any $t_1<t$. In this later case, 
\be\label{eq:n1}
A(t) = \lim_{t_1\to t^-} \frac{1}{t-t_1}\int_{t_1}^t A(s)ds \geq f^0(p^0) = F_{p^0}(t), 
\ee
which shows \eqref{FleqA}. 

Fix $k=1,2$. On $I^k$, we have $A(t)\leq f^k_{\max}$ by assumption on $A$. Thus
$$
\bar \lambda^k = \int_0^1 A(t){\bf 1}_{I^k}(t)dt \leq  
f^k_{\max}  | [0,1] \cap I^k  | \leq f^k_{\max}.
$$

Let us set 
$$
\forall t\in \R, \qquad \psi_{p^0}(t ) = \max_{t_1\leq t} \Bigl\{ \int_{t_1}^t (f^0(p^0)-A(s))ds\Bigr\} .
$$
We explain in Lemma \ref{lem.analysew00} below that $\psi_{p^0}$ is nonnegative, Lipschitz continuous, $1-$periodic and satisfies
\be\label{psiprimp0}
\psi_{p^0}'(t) = \left\{\begin{array}{ll}
f^0(p^0)-A(t) & \text{if } \psi_{p^0}(t)>0 \\
0 & \text{if }\psi_{p^0}(t)=0
\end{array}\right.
\qquad \text{a.e.}
\ee
and
\be\label{refpsip0Fp0}
\psi_{p^0}'(t) = f^0(p^0)-F_{p^0}(t)  \qquad {\rm a.e.}. 
\ee
Moreover, by the definition of $\psi_{p^0}$,  for any $t\in \R$, $\psi_{p^0}(t)=0$ is equivalent to saying that $\frac{1}{t-t_1} \int_{t_1}^t A(s)ds \geq f^0(p^0)$ for any $t_1$, and thus, as explained in \eqref{eq:n1}, one has $A\geq f^0(p^0)$ a.e. on $\{\psi_{p^0}=0\}$. 
\medskip

\noindent{\bf Step 1: $\hat \lambda^k$ are nondecreasing.}
Fix $k=1,2$. We now prove that the $\hat \lambda^k$ are non decreasing on $[0,\bar \lambda^0]$ (it is constant on $[\bar \lambda^0,\lambda^0_{\max}]$). If $0\leq \lambda\leq \bar \lambda\leq \bar \lambda^0$, then 
$$
a^0\leq p^0:=   u_+^0(\lambda) \leq  u_+^0(\bar \lambda)=: \bar p^0\leq b^0.
$$
Hence, by the definition of $\psi_{p^0}$ and $\psi_{\bar p^0}$, $\psi_{p^0}\leq \psi_{\bar p^0}$ and therefore $\{\psi_{p^0}>0\}\subset \{\psi_{\bar p^0}>0\}$. Recalling \eqref{psiprimp0}, \eqref{refpsip0Fp0} and the facts that $f^0(p^0)\leq f^0(\bar p^0)$ and that  $A\geq f^0(p^0)$ a.e. on $\{\psi_{p^0}=0\}$, we have  
\begin{align*}
F_{p^0}(t) & = f^0(p^0)- \psi_{p^0}'(t) = {\bf 1}_{\{\psi_{p^0}>0\}}(t) A(t)+ {\bf 1}_{\{\psi_{p^0}=0\}}(t)f^0(p^0) \\
& =  {\bf 1}_{\{\psi_{\bar p^0}>0\}}(t) A(t)+ {\bf 1}_{\{\psi_{\bar p^0}=0\}}(t)f^0( p^0) + {\bf 1}_{\{\psi_{\bar p^0}>0, \  \psi_{p^0}=0\}} (f^0(p^0)-A(t))\\
& \leq  {\bf 1}_{\{\psi_{\bar p^0}>0\}}(t) A(t)+ {\bf 1}_{\{\psi_{\bar p^0}=0\}}(t)f^0(\bar p^0)  = F_{\bar p^0}(t).
\end{align*}
Recalling \eqref{defhatlambda}  then shows that $\hat \lambda^k$ is nondecreasing. \medskip

\noindent{\bf Step 2:  $\hat \lambda^k$ is continuous in $[0,\bar \lambda^0]$.} We assume that $\lambda^n$ converge to $\lambda$ in $[0, \bar \lambda^0]$. Let $p^{0,n}= u_-^0(\lambda^n)$ and $p^0= u_-^0(\lambda)$. Then $(p^{0,n})$ converges to $p^0$ and 
$(\psi_{p^{0,n}})$ converges uniformly to $\psi_{p^0}$. Using assumption \eqref{hypIj}, we can write the set $I^k\cap [0,1]$ into a finite union of disjoint intervals $((t^j_1, t^j_2))_{j=1,\dots, J_k}$ up to a set of measure $0$. Then \eqref{refpsip0Fp0} shows that  
$$
\int_0^1 F_{p^{0,n}}(t){\bf 1}_{I^k}(t)dt = \sum_{j=1}^{J_k} \int_{t_{j_1}}^{t_{j_2}}(- \psi'_{p^{0,n}}(t)+f^0(p^{0,n}))dt = 
\sum_{j=1}^{J_k} (\psi_{p^{0,n}}(t^j_1)-\psi_{p^{0,n}}(t^j_2)+f^0(p^{0,n})(t^j_2-t^j_1))
$$
converges to 
$$
\sum_{j=1}^{J_k} (\psi_{p^{0}}(t^j_1)-\psi_{p^{0}}(t^j_2)+f^0(p^{0})(t^j_2-t^j_1)) = \int_0^1 F_{p^{0}}(t){\bf 1}_{{I^k}}dt.
$$
By \eqref{defhatlambda} this shows the continuity of $\hat \lambda^k$ in $[0,\bar \lambda^0]$. \medskip

\noindent{\bf Step 3: proof of \eqref{aezlqkjrsd} and  \eqref{aezlqkjrsdBIS}.}  By \eqref{defhatlambda} and \eqref{refpsip0Fp0} we have, for $ \lambda\in [0,\bar \lambda^0]$,
$$
 \hat \lambda^1(\lambda)+  \hat \lambda^2(\lambda)=\int_0^1 F_{p^0}(t)dt = \int_0^1 ( f^0(p^0)- \psi_{p^0}'(t) )dt = f^0(p^0)=\lambda, 
$$
since $\psi_{p^0}$ is periodic.  This is \eqref{aezlqkjrsd}. By \eqref{FleqA}, $F_{p^0}(t)\leq A(t)$ a.e.. Hence by \eqref{defhatlambda}, $\hat \lambda^k(\lambda)\leq \bar \lambda^k$ for any $\lambda\in [0,\bar \lambda^0]$. For $\lambda=\bar \lambda^0$, we then have 
$$
\bar \lambda^0= \bar \lambda^1+\bar \lambda^2\geq  \hat \lambda^1(\bar \lambda^0)+  \hat \lambda^2(\bar \lambda^0)= \bar \lambda^0 , 
$$
which shows that the inequalities $\bar \lambda^k\geq  \hat \lambda^k(\bar \lambda^0)$ are actually equalities. This is \eqref{aezlqkjrsdBIS}. Let us finally remark that $\hat \lambda^k(\lambda)=\bar \lambda ^k${,} $\forall \lambda \ge \bar \lambda ^0$. Hence $\hat \lambda ^k$ is also continuous in $[0,f^0_{\max}]$ (recall that it is continuous in $[0,\bar \lambda ^0]$ by Step 2).
\end{proof}

In the proof of Lemma \ref{lem.hatlambdak} we used the following result: 

\begin{Lemma} [Analysis of $\psi_{p^0}$]\label{lem.analysew00} Fix $p^0\in [a^0,b^0]$ such that \eqref{condcondf0p0} holds and let
$$
\forall t\in \R, \qquad \psi_{p^0}(t ) {:=} \max_{t_1\leq t} \Bigl\{ \int_{t_1}^t (f^0(p^0)-A(s))ds\Bigr\} .
$$
Then $\psi_{p^0}$ is nonnegative, Lipschitz continuous, $1-$periodic and satisfies, a.e., 
$$
\psi_{p^0}'(t) = f^0(p^0)- F_{p^0}(t)= \left\{\begin{array}{ll}
f^0(p^0)-A(t) & \text{if } \psi_{p^0}(t)>0 \\
0 & \text{if }\psi_{p^0}(t)=0.
\end{array}\right.
$$
In addition, $\psi_{p^0}'(t)+A(t)-f^0(p^0)\geq0$ a.e.. 
\end{Lemma}

\begin{proof} Note, choosing $t_1=t-1$ as a competitor and using  \eqref{condcondf0p0}, that $\psi_{p^0}\geq 0$. Moreover,  by  \eqref{condcondf0p0} and periodicity of $A$, the maximum in $t_1$ in the definition of $\psi_{p^0}$ can be chosen in $[t-1,t]$.  By periodicity of $A$, $\psi_{p^0}$ is $1-$periodic. Moreover, as 
$$
\psi_{p^0}(t ) = \max_{t_1\in \R} \Bigl\{ \int_{t_1\wedge t}^t (f^0(p^0)-A(s))ds\Bigr\} , 
$$
where the integrand is bounded, $\psi_{p^0}$ is also Lipschitz continuous as the supremum of uniformly Lipschitz {continuous} quantities.

Let us now compute $\psi_{p^0}'$. On $\{\psi_{p^0}=0\}$,  we have $\psi_{p^0}'=0$ a.e.. 
Let  $t$ be a point of derivability of $\psi_{p^0}$ with $\psi_{p^0}(t)>0$ and such that $t$ is a point of continuity of $A$. If $\hat t_1$ is optimal in the definition of $\psi_{p^0}$, then $\hat t_1<t$ because $\psi_{p^0}(t)>0$. Hence, for $|h|>0$ small, 
$$
\psi_{p^0}(t+h) \geq \int_{\hat t_1}^{t+h} (f^0(p^0)-A(s))ds = \psi_{p^0}(t)+ \int_t^{t+h} (f^0(p^0)-A(s))ds, 
$$
which implies that $\psi_{p^0}'(t)= f^0(p^0)-A(t)$. So we have proved that 
$$
\psi_{p^0}'(t) = \left\{\begin{array}{ll}
f^0(p^0)-A(t) & \text{if } \psi_{p^0}(t)>0 \\
0 & \text{if }\psi_{p^0}(t)=0
\end{array}\right.
\qquad \text{a.e..}
$$
On the other hand equality $\psi_{p^0}(t)=0$ is equivalent to saying that, for any $t_1<t$,  
\be\label{eq:n2}
\frac{1}{t-t_1}\int_{t_1}^t A(s))ds \geq f^0(p^0).
\ee
 Comparing \eqref{defFp0} with the previous equality  shows that 
$\psi_{p^0}'(t) = f^0(p^0)-F_{p^0}(t)$ a.e.. 

Finally, we have seen in \eqref{eq:n1} that
 $A\geq f^0(p^0)$ a.e. on $\{\psi_{p^0}=0\}$, which shows the last claim. 
\end{proof}

\begin{proof}[Proof of Theorem \ref{thm.main}: $\mathcal G_{\bar \Lambda}$ is a maximal germ] 
By Lemma \ref{lem.hatlambdak},  $\bar{\Lambda}$ satisfies \eqref{eq::c12}, which implies by Theorem \ref{th:1} that $\mathcal G_{\bar{\Lambda}}$ is a maximal germ.
\end{proof}

\begin{proof}[Proof of Lemma \ref{lem.reduction}] 
Let us set 
\begin{align*}
\Gamma&:=\left\{U=(u^0_+(\lambda),u^1_+(\lambda^1),u^2_+(\lambda^2))\quad \mbox{with}\quad \lambda^k:=\hat \lambda^k(\lambda)
\quad \mbox{for}\quad k=1,2\quad \mbox{and}\quad \lambda\in [0,\bar \lambda^0] \right\} \\
&\;  = \left\{U= (p^0, \hat p^1_{p^0},  \hat p^2_{p^0}), \qquad p^0\in [a^0, b^0]\; {\rm with} \; f^0(p^0)\leq \bar \lambda^0\right\}
\end{align*}
and 
$$
\left\{\begin{array}{l}
P_0:=(a^0,a^1,a^2)\in \Gamma\\
P_3:=(c^0,c^1,c^2)\\
\\
P_1:=(u^0_-(\bar \lambda^1),u^1_+(\bar \lambda^1),c^2)\\
P_2:=(u^0_-(\bar \lambda^2),c^1,u^2_+(\bar \lambda^2))\\
\end{array}\right.
$$
Then $E_{\bar \Lambda}$ defined in \eqref{defG0} is equal to
$$
E_{\bar \Lambda} = \Gamma\cup\{P_1, P_2, P_3\},
$$
the curve $\Gamma$ corresponding to case (i) in \eqref{defG0}, $P_1$ to case (ii), $P_2$ to case (iii) and $P_3$ to case (iv). Therefore $E_{\bar \Lambda} $ generates $\mathcal G_{\bar \Lambda}$ by Theorem \ref{th:1}-(ii).
\end{proof}

\bigskip

\paragraph{Three explicit examples.} 
We complete this part by three explicit computations. In the first one, there is no flux limiter (and hence no stop); the homogenized germ is then quite straightforward. The second one involves one stop only and no other flux limiter; it shows that the order (stop-road 1-road 2 or stop-road 2-road 1) influences the homogenized germ, even if the flux function is the same for both exit roads. The last one gives a hint of the class of germs that can be reached through our homogenization procedure.\bigskip  

\noindent{\it Example 1: the case where the traffic is never limited.}
We assume that
\be\label{hyp:A(t)max}
A(t)=\min\{f^0_{\max}, f^j_{\max}\} \qquad \text{for $t\in I^j$, $j=1,2$}
\ee
and that the sets $I^1$ and $I^2$ are as simple as possible:
\be\label{hyp:Iksimple}
\text{Up to a translation in time, the restriction of $I^k$ to $[0,1]$ is a single interval. }
\ee
Under these assumptions, we can compute explicitly $\bar \Lambda$. 

\begin{Proposition}\label{prop.exmacrogerm} Assume \eqref{hyp:A(t)max} and \eqref{hyp:Iksimple}. Let us set $\theta^k= \left| I^k\cap [0,1]\right|$ (for $k=1,2$). 
Then 
$$\left\{\begin{array}{l}
\displaystyle \bar \lambda^k:=\theta^k \min\{f^0_{\max}, f^j_{\max}\} \quad \mbox{for}\quad k=1,2,\\
\;\\
\displaystyle \bar \lambda^0:=\theta^1 \min\{f^0_{\max}, f^1_{\max}\} +\theta^2 \min\{f^0_{\max}, f^2_{\max}\} .
\end{array}\right.$$
Letting $\displaystyle\theta^k_*:=\frac{\bar \lambda^k}{\bar \lambda^0}$ (for $k=1,2$), 
the curves $\hat \lambda^1, \hat \lambda^2$ are given by
$$\left\{\begin{array}{ll}
\left\{\begin{array}{l}
\hat \lambda^1(\lambda):=\max(\theta^1\lambda,\lambda-\bar \lambda^2)\\
\hat \lambda^2(\lambda):=\min(\theta^2\lambda,\bar \lambda^2)\\
\end{array}\right|\quad \mbox{for}\quad \lambda\in  [0,\bar \lambda^0] & \text{if $\theta^2\ge \theta^2_*$,}\\ 
\; &\\
\left\{\begin{array}{l}
\hat \lambda^1(\lambda):=\min(\theta^1\lambda,\bar \lambda^1)\\
\hat \lambda^2(\lambda):=\max(\theta^2\lambda,\lambda-\bar \lambda^1)\\
\end{array}\right|\quad \mbox{for}\quad \lambda\in  [0,\bar \lambda^0] & \text{if $\theta^2< \theta^2_*$.} 
\end{array}\right. 
$$
\end{Proposition}

\begin{proof}  The computation of the $\bar \lambda^k$ ($k=0,1,2$) is immediate. Let us now compute  the $\hat \lambda^k$ ($k=1,2$). To fix the ideas, we assume that $\theta^2\ge \theta^2_*$, the other case being treated in a symmetric way. Without loss of generality, we also assume that $0<\theta^1<1$, since otherwise the problem reduces to a problem with a single outgoing road. We set $\phi^k= \min\{f^0_{\max}, f^k_{\max}\}$, $k=1,2$. Note that $\theta^2\ge \theta^2_*$ is equivalent to saying that $\phi^1\geq \phi^2$. Fix $\lambda \in [0, \bar \lambda^0]$ and let $p^0=u_+^0(\lambda)$. 

Let us first assume that $\lambda\in [0, \phi^2]$. Recalling that $\theta_2=1-\theta_1$, we have $\max(\theta^1\lambda,\lambda-\bar \lambda^2)= \theta^1 \lambda$ and $\min(\theta^2\lambda,\bar \lambda^2)= \theta^2\lambda$. On the other hand, in this case, the map $F_{p^0}$ defined in \eqref{defFp0} is constant and equal to $\lambda = f^0(p^0)$.   
Then, for $k=1,2$, 
$$
 \hat \lambda^k(\lambda)= \int_0^1 F_{p^0(\lambda)}(t) {\bf 1}_{I^k}(t)dt=\theta^k \lambda. 
 $$

Let us now suppose that $\lambda\in (\phi^2, \bar \lambda^0]$. Then $\max(\theta^1\lambda,\lambda-\bar \lambda^2)= \lambda-\bar \lambda^2$ and $\min(\theta^2\lambda,\bar \lambda^2)= \bar \lambda^2$. To compute $F_{p^0}$, we assume without loss of generality that $I^1\cap [0,1)= [0,\theta^1)$ while $I^2\cap [0,1)= [\theta^1,1)$. 
Since $\lambda\le \bar \lambda ^0=\theta^1\phi^1+\theta^2\phi^2$ and $\lambda>\phi^2$, we deduce that $\lambda<\phi^1$. Hence the minimum over $t_1$ of $\frac 1{t-t_1}\int_{t_1}^t(A(s)-\lambda)ds$ is reached for $t_1=-(1-\theta^1)=-\theta^2$ if $t\in[0,\theta^1)$.
Then, by \eqref{defFp0},
$$
F_{p^0}(t) = \left\{ \begin{array}{ll}
f^0(p^0) & \text{if $t\in [\frac{\theta^2(\lambda-\phi^2)}{\phi^1-\lambda}, \theta^1)$}\\
A(t) & \text{otherwise}
\end{array}\right. \; \text{(mod 1)},
$$
so that 
$$
\hat \lambda^1(\lambda)= \int_0^1 F_{p^0(\lambda)}(t) {\bf 1}_{I^1}(t)dt =\int_0^{\theta^1} F_{p^0(\lambda)}(t) dt= \frac{\theta^2(\lambda-\phi^2)}{\phi^1-\lambda}\phi^1+ (\theta^1-\frac{\theta^2(\lambda-\phi^2)}{\phi^1-\lambda}) \lambda  = \theta^2(\lambda-\phi^2) +\theta^1\lambda= \lambda-\bar \lambda^2,
$$
while 
$$
\hat \lambda^2(\lambda)=\int_0^1 F_{p^0(\lambda)}(t) {\bf 1}_{I^2}(t)dt= \theta^2 \phi^2=\bar \lambda^2. 
$$
\end{proof}
 
 \noindent{\it Example 2: one stop followed successively by two exits.}
Consider now the case where $f^1=f^2=f$, and $f^0$ may be different. We set $A^0:=\max(f^0_{\max},f_{\max})$.
We also assume that for $\tau_0=\theta^0$, $\tau_1=\theta^0+\theta^1$, $\tau_2=\theta^0+\theta^1+\theta^2=1$ with $\theta^i >0$, we have
$$A(t)=\left\{\begin{array}{ll}
0 &\quad \mbox{on}\quad [0,\tau_0)=I^0\\
A^0 &\quad \mbox{on}\quad [\tau_0,\tau_1)=I^1\\
A^0 &\quad \mbox{on}\quad [\tau_1,\tau_2)=[\tau_1,1)=I^2\\
\end{array}\right.$$
 In other worlds,  all the incoming vehicles from road $0$, go on road $j$ during the time interval $I^j$ for $j=1,2$, while they are all stopped at the junction during the time interval $I^0$.

We then have the following result
\begin{Proposition}\label{lem::f21}{\bf (Flux computation with one stop followed successively by two exits)}
Under the previous assumptions, we have for $\lambda\in [0,\bar \lambda^0]$
$$\left\{\begin{array}{l}
\hat \lambda^1(\lambda)=\min\{\lambda (\theta^0+\theta^1),\bar \lambda^1\}\\
\hat \lambda^2(\lambda)=\max\{\lambda \theta^2,\lambda-\bar \lambda^1\}\\
(\hat \lambda^1+\hat \lambda^2)(\lambda)=\lambda
\end{array}\right.$$
with
$$\bar \lambda^0:=A^0(\theta^1+\theta^2),\quad \bar \lambda^1:=A^0\theta^1,\quad \bar \lambda^2=A^0\theta^2.$$
Moreover, if $\theta^1=\theta^2$, then we have $\bar \lambda^1=\bar \lambda^2$ and
$$\hat \lambda^1 > \hat \lambda^2 \quad \mbox{on}\quad (0,\bar \lambda^0)$$
with equality at both end points of the interval $(0,\bar \lambda^0)$.
\end{Proposition}

\begin{Remark}\label{rem::f23}
{The result of Proposition \ref{lem::f21} in the special case $\theta^1=\theta^2$, means that the order (stop-road 1-road 2) matters with respect to the order (stop-road-2-road 1).  The road which receives the traffic just after the stop, will have a higher passing flux than the other one.\\
After reversion, this corresponds  to a convergent 2:1 junction where the outgoing road 0 is congested.
Then road $1$ (just after the stop) will evacuate more easily than road $2$, its vehicles onto the road $0$. This happens because the stop created some free space on road $0$ just after the junction.
This last interpretation is much more intuitive here.}
\end{Remark}

\begin{proof}
For $t\in [0,1]$, let $B(t)=\max(0, A^0(t-\theta^0))$  and extend $B$ to $\R$ by
$B(t+1)=B(t)+A^0(\theta^1+\theta^2)$ such that $B'=A$. For any $\lambda\in [0,\bar \lambda^0]$, where $\bar \lambda^0:=\int_{[0,1]} A=A^0(\theta^1+\theta^2)$, define $\Phi_\lambda$ as in 
\eqref{eq::f16}
and $t=t_\lambda\in (\theta^0,1]$ such that
$$\lambda t = B(t)\qquad {\rm i.e}\qquad t_\lambda-\theta^0=\frac{\lambda \theta^0}{A^0-\lambda}.$$
We then have, using that $\{\Phi_\lambda=B\}\subset\{F_\lambda=\lambda\}$,
$$\hat \lambda^j(\lambda)=\int_{I^j} F_\lambda=\int_{I^j\cap \left\{\Phi_\lambda> B\right\}} A+\int_{I^j\cap \left\{\Phi_\lambda=B\right\}} \lambda=\int_{I^j} A+
\int_{I^j\cap \left\{\Phi_\lambda=B\right\}} (\lambda-A).$$
Since $\left\{\Phi_\lambda=B\right\}\cap [0,1]=[t_\lambda,1]$, we deduce that
$$\hat \lambda^1(\lambda)=A^0\theta^1+\int_{[\tau_1\wedge t_\lambda,\tau_1]} (\lambda-A^0)$$
Let us set $\lambda_*=\frac{A^0\theta^1}{\theta^0+\theta^1}$ such that $t_{\lambda_*}=\tau_1$. This implies that, for $\lambda\in [0,\lambda_*]$, we have
$$\hat \lambda^1(\lambda)=A^0\theta^1+(\lambda-A^0)(\tau_1-t_\lambda)=\lambda \tau_1$$
and then
$$\hat \lambda^1(\lambda)=\min(\lambda \tau_1,\bar \lambda^1)\quad \mbox{on}\quad [0,\bar \lambda^0],\quad \mbox{with}\quad \bar \lambda^1:=\lambda_*\tau_1=A^0\theta^1.$$
Similarly, using that
$$\hat \lambda^2(\lambda)=\int_{[t_\lambda \vee \tau_1,1]} (\lambda-A^0)+A^0\theta^2$$
we can show that
$${\hat \lambda^2(\lambda)=\max(\lambda \theta^2,\lambda-\bar \lambda^1)\quad \mbox{on}\quad [0,\bar \lambda^0]}$$
This ends the proof.
\end{proof}

\begin{Remark}{{\bf (Bounds on the derivatives of $\hat \lambda^j$)}}
A natural question is the characterization of the functions $\hat \lambda^j$ that can be constructed by homogenization. In fact, the derivative of these {fluxes} has to be bounded between $0$ and $1$. More precisely, one can show that
$$1-g^2(\lambda) \ge (\hat \lambda^1)'(\lambda) \ge g^1(\lambda)\ge 0 \quad \mbox{a.e. for}\quad \lambda\in [0,\bar \lambda^0]$$
(and symmetrically for $\hat \lambda^2$) with
$$g^j(\lambda):=|\left\{F_\lambda=\lambda\right\}\cap I^j|$$
Moreover $g^j\in L^\infty([0,\bar \lambda^0])$ has a monotone nonincreasing representant in the class of $L^\infty$ functions.
We can show that this also implies that if there exists some $\lambda_1\in (0,\bar \lambda^0)$ such that the derivative vanishes
$$(\hat \lambda^j)'(\lambda_1)=0$$
then $\hat \lambda^j=const$ on $[\lambda_1,\bar \lambda^0]$.
Moreover each $\hat \lambda^j$ is sandwiched in between a concave function and a convex function.\\

\end{Remark}

\noindent {\it Example 3: concave flux $\hat \lambda^1$.} We now explain how to compute $\hat \lambda^1$ from $A$ when $A$ has a particular structure {and is  assumed to be continuous}.\\

\begin{Proposition}\label{lem::f4}{\bf (The case of $\hat \lambda^1$  concave {and $A$ continuous})}\\
Given $0<t_1-t_0<1$, assume furthermore that $A:\R \to [0,+\infty)$ (still $1$-periodic) is $C^1$, decreasing on $[t_0,t_1]$, and increasing on $[t_1,t_0+1]$.
Given $\bar \lambda^0:=\int_{[0,1]} A$, consider $\bar t_0\in [t_0,t_1]$ such that $A(\bar t_0)= {\bar \lambda^0}$.
Assume now that 
$$A'<0\quad \mbox{on}\quad [\bar t_0,t_1)\qquad {\rm and}\qquad  I^1=[\bar t_0,t_1]\quad \mbox{mod. 1}$$
Then up to translate $A$, we can assume that $\bar t_0=0$, and we have
\begin{equation}\label{eq::b201}
1>(\hat \lambda^1)'(\lambda) = \left\{\begin{array}{ll}
(A_{|[0,t_1]})^{-1}(\lambda)& \quad \mbox{if}\quad \lambda\in (A(t_1),A(0)]\\
|I^1|=t_1 & \quad \mbox{if}\quad \lambda\in [0,A(t_1))\\
\end{array}\right.
\end{equation}
The function $\hat \lambda^1$ is $C^1$ and concave on $[0,\bar \lambda^0]$. Moreover $\hat \lambda^1$ is linear on $[0,A(t_1)]$, and $C^2$ strictly concave on $(A(t_1),\bar \lambda^0]$. We also have $(\hat \lambda^1)'(\bar \lambda^0)=\bar t_0=0$ when $0=\bar t_0< t_1$.\\
\end{Proposition}

\begin{proof}
We first notice that for $\lambda\in [0,A(t_1)]$, we have $F_\lambda=\lambda$ and $\hat \lambda^1(\lambda)=|I_1| \lambda$.\\
For $\lambda\in [A(t_1),A(\bar t_0)]$, we define $t_\lambda\in [\bar t_0,t_1]$ such that
$A(t_\lambda)=\lambda$. Arguing as in the proof of Proposition \ref{lem::f21}, we have%
$$\beta(\lambda):=\hat \lambda^1(\bar \lambda^0)-\hat \lambda^1(\lambda)=\int_{\bar t_0}^{t_\lambda} (A-\lambda)$$
Because $\lambda \mapsto t_\lambda=(A_{|[\bar  t_0,t_1]})^{-1}(\lambda)=A^{-1}(\lambda)$ is $C^1$ on $(A(t_1),A(\bar t_0)]$, we see for later use that $\beta$ is  also $C^1$, and is moreover nonincreasing. Now for $t:=t_\lambda$, we have $A(t)=\lambda$ and 
$$\beta(\lambda)=\int_{\bar t_0}^{A^{-1}(\lambda)} (A(s)-A(t))ds$$
i.e.
$$(\beta\circ A)(t)=\int_{\bar t_0}^{t} (A(s)-A(t))ds$$
Taking the derivative, and dividing by $A'(t)<0$, and up to assume that $\bar t_0=0$, we get
$$(-\beta')\circ A = Id_{[\bar t_0,t_1)}$$
with $-\beta'=(\hat \lambda^1)'$. This implies that
$$(\hat \lambda^1)'=(A_{|[\bar t_0,t_1]})^{-1} \quad \mbox{on}\quad (A(t_1),A(\bar t_0)]$$
\end{proof}

{\begin{Remark}\label{rem::b200}
1) Notice that we can also prove a sort of recriprocal result.
Given any $C^2$ concave function $\hat \lambda^1:[0,\bar \lambda^0]\to [0,+\infty)$ with $(\hat \lambda^1)''<0$ on $(0,\bar \lambda^0)$
and $\hat \lambda^1(0)=(\hat \lambda^1)(\bar \lambda^0)=0< (\hat \lambda^1)'(0)<1$, we can cook-up a suitable $1$-periodic function $A$ with $A(t_1)=0$. Everything can be done such that $\hat \lambda^1$ is associated to $A$ as in Proposition \ref{lem::f4} (except that  $A$ is constant on $(t_1,t_0+1)$ and possibly discontinuous at $t_0$ and $t_1$).\\
2) Notice also that in this remark and in Proposition \ref{lem::f4}, the function $A$ is not piecewise constant, as it is assumed in our homogenization result. Nevertheless, an approximation {of} such $A$ by a sequence of piecewise constant functions is always possible, and then relation (\ref{eq::b201})
is still valid, once it is correctly interpreted (where $\hat \lambda^1$ is continuous and piecewise linear).
Then any concave $\hat \lambda^1$ as in point 1), can then be obtained as limits of homogenized $\hat \lambda^1$ of piecewisely approximated functions $A$.
\end{Remark}}

\subsection{Proof of Theorem \ref{th:1}}\label{sec:2.2}
This subsection is devoted to the proof of Theorem \ref{th:1}. Starting with a lemma describing how the dissipation condition can be violated (Lemma \ref{lem::22}), we prove that $\mathcal G_\Lambda$ is maximal and generated by $E^+_{\Lambda}$ (Lemma \ref{lem::r36}) and then that it is  a germ (Lemma \ref{lem::r37}). 

\subsubsection{A technical lemma}

We consider $P=(p^0,p^1,p^2)$ and $\bar P=(\bar p^0,\bar p^1,\bar p^2)$ with $P,\bar P \in Q^{RH}$, i.e. such that we have the Rankine-Hugoniot relations
$$\left\{\begin{array}{l}
f^0(p^0)=f^1(p^1)+f^2(p^2)\\
f^0(\bar p^0)=f^1(\bar p^1)+f^2(\bar p^2).\\
\end{array}\right.$$
Defining
\begin{equation}\label{eq::c9}
\left\{\begin{array}{ll}
F^0:=f(\bar p^0)-f(p^0),\quad s^0:=\mbox{sign}(\bar p^0-p^0)\\
F^1:=f(\bar p^1)-f(p^1),\quad s^1:=\mbox{sign}(\bar p^1-p^1)\\
F^2:=f(\bar p^2)-f(p^2),\quad s^2:=\mbox{sign}(\bar p^2-p^2)\\
\end{array}\right.
\end{equation}
we get
$$D(\bar P,P)=s^0F^0-\left\{s^1F^1+s^2F^2\right\}\quad \mbox{with}\quad F^0=F^1+F^2$$
and $s^j=0$ implies $F^j=0$.

\begin{Lemma}\label{lem::22}{\bf (Violated dissipation for divergent 1:2 junction)}\\
Let us consider the dissipation
$$D:=s^0F^0-\left\{s^1F^1+s^2F^2\right\}\quad \mbox{with}\quad \left\{\begin{array}{ll}
F^0=F^1+F^2&\\
s^j\in \left\{0,\pm 1\right\}&\quad \mbox{for}\quad j=0,1,2\\
s^j=0 \quad \mbox{implies}\quad F^j=0 &\quad \mbox{for}\quad j=0,1,2.
\end{array}\right.$$
Then $D<0$ if and only if
$$\left\{\begin{array}{ll}
&s^0F^0<0,\quad s^1=s^2\not= s^0 \quad \mbox{weakly}\\
\mbox{or}&\\
&s^1F^1>0,\quad s^0=s^2\not= s^1 \quad \mbox{weakly}\\
\mbox{or}&\\
&s^2F^2>0,\quad s^0=s^1\not= s^2 \quad \mbox{weakly}\\
\end{array}\right.$$
where 
$$s^1=s^2\not= s^0 \quad \mbox{weakly} \quad \Longleftrightarrow \quad s^0\not=0,\quad s^1s^2\ge 0,\quad s^0s^1\le 0,\quad s^0s^2\le 0.$$
\end{Lemma}

\begin{proof}
The proof is technical but elementary.
Up to change $(F^0,F^1,F^2)$ in $(F^0,-F^1,-F^2)$, we can assume that
$$D=s^0F^0+s^1F^1+s^2F^2\quad \mbox{with}\quad F^0+F^1+F^2=0$$
and we want to show that $D<0$ if and only if
$$\left\{\begin{array}{lll}
&\mbox{\bf (0)}&s^0F^0<0,\quad s^1=s^2\not= s^0 \quad \mbox{weakly}\\
\mbox{or}&&\\
&\mbox{\bf (1)}&s^1F^1<0,\quad s^0=s^2\not= s^1 \quad \mbox{weakly}\\
\mbox{or}&&\\
&\mbox{\bf (2)}&s^2F^2<0,\quad s^0=s^1\not= s^2 \quad \mbox{weakly}.\\
\end{array}\right.$$
\noindent {\bf Step 1: (0),(1) or (2) imply $D<0$}\\
We only consider the case (0) (the other cases being symmetric).\\
This means that we have
$$s^0F^0<0,\quad s^0\not=0,\quad s^1s^2\ge 0,\quad s^0s^1\le 0,\quad s^0s^2\le 0$$
and we distinguish several cases.\\
\noindent {\bf Case 1.a: $s^1=0=s^2$.}
Then $F^1=0=F^2$ and $D=s^0F^0<0$.

\noindent {\bf Case 1.b: $s^1=0\not=s^2$.}
Then $F^1=0$ and then $F^2=-F^0$ and also $s^2=-s^0$. We get $D=2s^0F^0<0$.\\
\noindent {\bf Case 1.c: $s^1\not =0=s^2$.}
This is symmetric to case 1.b.\\
\noindent {\bf Case 1.d: $s^1\not =0$, $s^2\not=0$.}
Then $s^1=s^2=-s^0$, and $F^1+F^2=-F^0$ gives $D=2s^0F^0<0$.\newline

We conclude that $D<0$ in all cases of Step 1.\\

\noindent {\bf Step 2: if we do not have (0),(1) nor (2) then $D\ge 0$}\\
If $s^jF^j\ge 0$ for all $j=0,1,2$, then $D\ge 0$. Then assume that at least one such term is negative. By symmetry, we can assume that
$$s^0F^0<0.$$
Notice also that if all the $s^j$ for $j=0,1,2$ have the same sign (with value in $\left\{0,\pm1\right\}$), then $D=0$ (because $F^0+F^1+F^2=0$).
Then we can assume that the $s^j$ do not have all the same sign.\\
Moreover recall  that we don't have (0). Hence we can assume in particular  that
$$\left\{\begin{array}{l}
s^0F^0<0\\
s^0\not=0\quad \mbox{and} \quad \left(s^1s^2< 0\quad \mbox{or} \quad s^0s^1> 0\quad \mbox{or}\quad s^0s^2> 0\right)\\
s^0,s^1,s^2\quad \mbox{do not have all the same sign}.
\end{array}\right.$$
We distinguish several cases.\\

\noindent {\bf Case 2.a: $s^0s^1>0$.} 
If $s^2\not=0$, then $s^1=s^0=-s^2$ and $F^0+F^1=-F^2$ which gives $D=2s^2F^2\ge 0$ because case (2) is also excluded.
If $s^2=0$, then $F^2=0$ and $F^1=-F^0$ which implies $D=0$.

\noindent {\bf Case 2.b: $s^0s^2>0$.} This case is symmetric of case 2.a.

\noindent {\bf Case 2.c: $s^1s^2<0$}. 
If $s^0=s^1$, then $s^0=s^1=-s^2$ and $F^0+F^1=-F^2$. This implies that $D=2s^2F^2\ge 0$, because (2) does not hold.
If $s^0=s^2$, then we obtain, in a symmetric way, that $D\ge 0$.\newline

We conclude that $D\ge 0$ in all cases of Step 2.\\
This completes the proof of the lemma.
\end{proof}

\subsubsection{Maximality}

\begin{Lemma}\label{lem::r36}{\bf (Maximality of ${\mathcal G}_\Lambda$)} We work under the assumptions of Theorem \ref{th:1}. Let us consider  a set  $G\subset Q$ satisfying the dissipation condition
$$D(\bar P,P)\ge 0 \quad \mbox{for all}\quad \bar P,P\in G.$$
Let
$$E^+_\Lambda:=\Gamma\cup \left\{P_1,P_2,P_3\right\} \quad \mbox{defined in (\ref{eq::e100})}.$$
If $E^+_\Lambda \subset G$,
then we have
$$G\subset {\mathcal G}_\Lambda.$$
This implies in particular that ${\mathcal G}_\Lambda$ is maximal.
\end{Lemma}

\begin{proof}
We choose $P\in G$ and we will test it with
$$\bar P\in \Gamma\cup \left\{P_1,P_2,P_3\right\}$$
using the dissipation condition $D(\bar P,P)\ge 0$ in order to show that $P\in {\mathcal G}_\Lambda$.\\
We write
$$P=(p^0,p^1,p^2),\quad \bar P=(\bar p^0,\bar p^1,\bar p^2)$$
We use notation (\ref{eq::c9}) for the fluxes $F^j$ for $j=0,1,2$.\\
\noindent {\bf Step 1: recovering Rankine-Hugoniot condition}\\
We choose $\bar P:=P_3$. 
Because for all $P\in Q=[a^0,c^0]\times [a^1,c^1]\times [a^2,c^2]$, we have $p^j\le \bar p^j=c^j$ for all $j=0,1,2$, 
we get
$$0\le D(\bar P,P)=F^0-(F^1+F^2),\quad f^0(\bar p^0)=f^1(\bar p^1)+f^2(\bar p^2),$$
which implies
\begin{equation}\label{eq:01}
f^0(p^0)-\left\{f^1(p^1)+f^2(p^2)\right\}\le 0.
\end{equation}
We now choose $\bar P:=P_0$. 
Because for all $P\in Q$, we have $p^j\ge \bar p^j=a^j$ for all $j=0,1,2$, 
we get
$$-\left\{F^0-(F^1+F^2)\right\}\ge 0,\quad f^0(\bar p^0)=f^1(\bar p^1)+f^2(\bar p^2),$$
which implies
\begin{equation}\label{eq:02}
f^0(p^0)-\left\{f^1(p^1)+f^2(p^2)\right\}\ge 0.
\end{equation}
Combining \eqref{eq:01} and \eqref{eq:02}, we get the Rankine-Hugoniot relation and then $P\in Q^{RH}$.\newline

\noindent {\bf Step 2: getting flux limiters}\\
\noindent {\bf Step 2.1: $0\le f^1(p^1)\le \bar \lambda^1$.} 
We set $\bar P:=P_1=(p^0_-(\bar \lambda^1), p^1_+(\bar \lambda^1), p^2_-(0))$. Assume by contradiction that
$$\lambda^1:=f^1(p^1)> \bar \lambda^1=f^1(\bar p^1).$$
Using Rankine-Hugoniot relation and the facts that $f^2\ge 0$ and $f^2(\bar p^2)=0$, we get
$$\lambda:=f^0(p^0)> \bar \lambda^1=f^0(\bar p^0).$$
Using that $\bar p^1\in [a^1,b^1]$ and that $\bar p^0\in [b^0,c^0]$, we deduce that
$$p^1> \bar p^1,\quad p^0< \bar p^0.$$
Then we get the table
$$\begin{tabular}{|c||c||c|c|}\hline
& $k=0$& 1 & 2\\\hline
$s^k$& $\boxed{>0}$ & $\boxed{<0}$& $$ \\\hline
$F^k$& $\boxed{<0}$  & $\boxed{<0}$ & $\boxed{\le 0}$ \\\hline
$s^kF^k$& $<0$  & $>0$ & $$ \\\hline
\end{tabular}$$
with the convention that the boxed inequalities are the known ones, and the unboxed inequalities are the deduced ones.\\

Hence whatever is the value of $s^2$, we deduce from Lemma \ref{lem::22} that $D<0$ either from $s^0F^0<0$ or from $s^1F^1>0$ (depending on the value of $s^2$).
Contradiction. \newline

\noindent {\bf Step 2.2: $0\le f^2(p^2)\le \bar \lambda^2$.} Choosing $\bar P:=P_2$, we get the result in a symmetric way.\newline

\noindent {\bf Step 2.3: conclusion}\\
From Rankine-Hugoniot relation, we deduce that
$$0\le f^0(p^0)\le \bar \lambda^0:=\bar \lambda^1+\bar \lambda^2,$$
which, combining with Steps 2.1 and 2.2, implies the limiters
$$0\le f^j(p^j)\le \bar \lambda^j\quad \mbox{for}\quad j=0,1,2.$$

\noindent {\bf Step 3: getting key inequalities defining ${\mathcal G}_\Lambda$.}\\
\noindent {\bf Step 3.1: $f^{1,+}(p^1)\ge \hat \lambda ^1(f^{0,+}(p^0))$.}\\
Assume by contradiction that
$$f^{1,+}(p^1)< \hat \lambda^1(f^{0,+}(p^0)).$$
We choose $\bar \lambda =\min(\bar \lambda^0,f^{0,+}(p^0))$ and we define $\bar P=(\bar p^0, \bar p^1,\bar p^2):=(u^0_+(\bar \lambda),u^1_+(\bar \lambda^1), u^2_+(\bar \lambda^2))$ with $\bar \lambda ^k=\hat \lambda ^k(\bar \lambda)$. This implies in particular that
$$f^0(\bar p^0)=\bar \lambda\ge f^0(p^0)=:\lambda.$$
Hence (recalling that $\hat \lambda^1$ is nondecreasing)
$$\lambda^1:=f^1(p^1)\le f^{1,+}(p^1)< \hat \lambda^1(f^{0,+}(p^0))\le\hat \lambda^1(\bar \lambda)=\bar \lambda^1=f^1(\bar p^1)=f^{1,+}(\bar p^1)$$
and then
$$p^1\in [a^1,b^1],\quad p^1<\bar p^1.$$
Then we get the table
$$\begin{tabular}{|c||c||c|c|}\hline
& $k=0$& 1 & 2\\\hline
$s^k$& $$ & $\boxed{>0}$& $$ \\\hline
$F^k$& $\boxed{\ge 0}$  & $\boxed{>0}$ & $$ \\\hline
$s^kF^k$& $$  & $>0$ & $$ \\\hline
\end{tabular}$$
In order to go further, we have to distinguish cases.\\
\noindent {\bf Case A: $\lambda< \bar \lambda$.}
Then
$$\bar \lambda=f^0(\bar p^0)=\min(\bar \lambda^0,f^{0,+}(p^0))> f^0(p^0)=\lambda$$
and
$$\bar p^0=u^0_+(\bar \lambda)\le u^0_+(\bar \lambda^0)<p^0$$
i.e.
$$\begin{tabular}{|c||c||c|c|}\hline
& $k=0$& 1 & 2\\\hline
$s^k$& $<0$ & $\boxed{>0}$& $$ \\\hline
$F^k$& $\boxed{> 0}$  & $\boxed{>0}$ & $$ \\\hline
$s^kF^k$& $<0$  & $>0$ & $$ \\\hline
\end{tabular}$$
Hence whatever is the value of $s^2$, we deduce from Lemma \ref{lem::22} that $D<0$ either from $s^0F^0<0$ or from $s^1F^1>0$ (depending on the value of $s^2$).
Contradiction.\\
\noindent {\bf Case B: $\lambda= \bar \lambda$.}
Then, 
we have with $\lambda^k=f^k(p^k)$ and $\bar \lambda^k=f^k(\bar p^k)$ for $k=1,2$
$$\lambda^1< \bar \lambda^1\quad \mbox{and}\quad \lambda^1+\lambda^2=\lambda=\bar \lambda=\bar \lambda^1+\bar \lambda^2$$
Hence
$$\lambda^2> \bar \lambda^2$$
i.e.
$$f^{2,+}(p^2)\ge f^2(p^2)> f^2(\bar p^2)=f^{2,+}(\bar p^2)$$
and then
$$\bar p^2<p^2.$$
We can almost complete the table
$$\begin{tabular}{|c||c||c|c|}\hline
& $k=0$& 1 & 2\\\hline
$s^k$& $$ & $\boxed{>0}$& $<0$ \\\hline
$F^k$& $\boxed{= 0}$  & $\boxed{>0}$ & $<0$ \\\hline
$s^kF^k$& $=0$  & $>0$ & $>0$ \\\hline
\end{tabular}$$
Again we deduce from Lemma \ref{lem::22} that $D<0$ using $s^2F^2>0$ or  $s^1F^1>0$ (depending on the sign of $s^0$). Contradiction.\newline

We get a contradiction in all the cases and so 
$$f^{1,+}(p^1)\ge  \hat \lambda^1(f^{0,+}(p^0)).$$
\noindent {\bf Step 3.2: $f^{2,+}(p^2)\ge  \hat \lambda^2(f^{0,+}(p^0)).$}
Proceeding symmetrically to Step 3.1, we get the result.\newline

\noindent {\bf Step 3.3: conclusion}\\
Finally, this shows that
$P\in {\mathcal G}_\Lambda$
and completes the proof of the lemma.
\end{proof}

\subsubsection{Germ property}

\begin{Lemma}\label{lem::r37}{\bf (Germ property of ${\mathcal G}_\Lambda$)}
Under the assumptions of Theorem \ref{th:1}, the set ${\mathcal G}_\Lambda$ defined by \eqref{eq::c1} is a germ.
\end{Lemma}

\begin{proof}
By construction, we have ${\mathcal G}_\Lambda \subset Q^{RH}$, and then we only have to show that\footnotemark[3]
\begin{equation}\label{eq::c11}
D(\bar P,P)\ge 0\quad \mbox{for all}\quad \bar P,P\in {\mathcal G}_\Lambda.
\end{equation}
\footnotetext[3]{The proof of inequality (\ref{eq::c11}) is a short proof. Still it is quite difficult to guess that proof from scratch (and also the expression of the germ $\mathcal G_{\Lambda}$) and it needs a lot of tries.
Notice that each component of $P$ and $\bar P$ can be either in the nondecreasing (i.e. fluid) or nonincreasing (i.e. congested) part of the flux.
A first (tedious) proof was done distinguishing $(2^3)^2=64$ cases, and using  a much more complicate (and equivalent) expression of ${\mathcal G}_\Lambda$.
Finally, the proof we give here is easy to follow  line by line but is absolutely not intuitive.}

Assume by contradiction that there exists $\bar P,P\in {\mathcal G}_\Lambda$ such that
$$D(\bar P,P)<0.$$
Then from Lemma \ref{lem::22}, we have two cases. Either
$$s^0F^0<0\quad \mbox{and}\quad s^0\not=s^1=s^2\mbox{ weakly},$$
or (up to exchange the indices $1$ and $2$), we have
$$s^1F^1>0\quad \mbox{and}\quad s^1\not=s^0=s^2\mbox{ weakly}.$$
\noindent {\bf Case A: $s^0F^0<0\quad \mbox{and}\quad s^0\not=s^1=s^2\mbox{ weakly}$.}\\
Up to exchange $P$ and $\bar P$, this means that
$$F^0<0,\quad s^0=1,\quad s^1\le 0,\quad s^2\le 0,$$
i.e.
$$\bar p^0>p^0,\quad \bar p^1\le p^1,\quad \bar p^2\le p^2,\quad f^0(\bar p^0)< f^0(p^0)\le \bar \lambda^0.$$
Hence
$$\bar p^0>u^0_-(\bar \lambda^0).$$
Recall that
$$f^{1,+}(\bar p^1)\ge \hat \lambda^1(f^{0,+}(\bar p^0)),\quad f^{2,+}(\bar p^2)\ge  \hat \lambda^2(f^{0,+}(\bar p^0))$$
and in particular
$$f^{0,+}(\bar p^0)\ge \bar \lambda^0,\quad f^{1,+}(\bar p^1)\ge\bar \lambda^1,\quad f^{2,+}(\bar p^2)\ge\bar \lambda^2$$
where we have used the fact that $\hat \lambda^k(f^0_{\max}) = \hat \lambda^k(\bar \lambda^0)=\bar \lambda^k$ for $k=1,2$.
Therefore, since $f^k(\bar p^k)\le \bar \lambda^k$, we have
$$\left\{\begin{array}{l}
\bar p^1\in \left\{u^1_+(\bar \lambda^1)\right\}\cup [u^1_-(\bar \lambda^1),c^1]\\
\\
\bar p^2\in \left\{u^2_+(\bar \lambda^2)\right\}\cup [u^2_-(\bar \lambda^2),c^2].
\end{array}\right.$$
This implies that
$$f^1(p^1)\le f^1(\bar p^1),\quad f^2(p^2)\le f^2(\bar p^2)$$
and then
$$f^0(p^0)=f^1(p^1)+f^2(p^2) \le f^1(\bar p^1)+f^2(\bar p^2)=f^0(\bar p^0)< f^0(p^0).$$
Contradiction.\\

\noindent {\bf Case B: $s^1F^1>0\quad \mbox{and}\quad s^1\not=s^0=s^2\mbox{ weakly}$.}\\
Up to exchange $P$ and $\bar P$, this means that
$$F^1>0,\quad s^1=1,\quad s^0\le 0,\quad s^2\le 0, $$
i.e.
$$f^1(\bar p^1)> f^1(p^1),\quad \bar p^1> p^1,\quad \bar p^0\le p^0,\quad \bar p^2\le p^2.$$
Recall also that
$$\left\{\begin{array}{l}
f^{1,+}(p^1)\ge \hat \lambda^1(f^{0,+}(p^0)),\quad f^{2,+}(p^2)\ge  \hat \lambda^2(f^{0,+}(p^0))\\
\\
f^{1,+}(\bar p^1)\ge \hat \lambda^1(f^{0,+}(\bar p^0)),\quad f^{2,+}(\bar p^2)\ge  \hat \lambda^2(f^{0,+}(\bar p^0)).
\end{array}\right.$$
\noindent {\bf Case B.1: $p^0\ge u^0_+(\bar \lambda^0)$.}
Then
$$f^{1,+}(p^1)\ge \hat \lambda^1(f^{0,+}(p^0)) = \bar \lambda^1$$
and
$$p^1\ge u^1_+(\bar \lambda^1)$$
which implies
$$f^1(\bar p^1)\le f^1(p^1).$$
Contradiction.\\
\noindent {\bf Case B.2: $p^0< u^0_+(\bar \lambda^0)$.}
Hence we have
$$\bar p^0\le p^0< u^0_+(\bar \lambda^0)$$
and then
$$F^0\le 0.$$
Using the fact that $F^1>0$, we get
$F^2<0$.
This implies that
$$\left\{\begin{array}{ll}
\bar p^1> p^1,&\quad \bar p^2< p^2\\
f^1(\bar p^1)> f^1(p^1),&\quad f^2(\bar p^2)< f^2(p^2).
\end{array}\right.$$
Hence
$$p^1< u^1_+(\bar \lambda^1),\quad \bar p^2< u^2_+(\bar \lambda^2).$$
Moreover
$$\left\{\begin{array}{l}
f^1(p^1)=f^{1,+}(p^1)\ge \hat \lambda^1(f^{0,+}(p^0))\ge \hat \lambda^1(\lambda),\quad \lambda:=f^0(p^0)\\
f^2(\bar p^2)=f^{2,+}(\bar p^2)\ge \hat \lambda^2(f^{0,+}(\bar p^0))\ge \hat \lambda^2(\bar \lambda),\quad \bar \lambda:=f^0(\bar p^0)\le \lambda.
\end{array}\right.$$
This implies in particular (using $\hat \lambda^1(\bar \lambda)+\hat \lambda^2(\bar \lambda)=\bar \lambda=f^1(\bar p^1)+f^2(\bar p^2)$) that
$$f^1(\bar p^1)\le \hat \lambda^1(\bar \lambda).$$
Using the monotonicity of the map $\lambda\mapsto \hat \lambda^1(\lambda)$, we get
$$f^1(\bar p^1)\le \hat \lambda^1(\bar \lambda) \le \hat \lambda^1(\lambda)\le f^1(p^1).$$
Contradiction with $f^1(\bar p^1)> f^1(p^1)$.\\
This completes the proof of the lemma.
\end{proof}

\subsubsection{Proof of Theorem \ref{th:1}} 
\begin{proof}[Proof of Theorem \ref{th:1}]
The proof of Theorem \ref{th:1} is a straightforward application of Lemma \ref{lem::r37}, which says that $\mathcal G_\Lambda$ is a germ, and of Lemma \ref{lem::r36}, which proves at the same time its maximality and the fact that it is generated by $E^+_{\Lambda}=\Gamma\cup \left\{P_1,P_2,P_3\right\}$. 
\end{proof}


\section{Construction of the correctors} \label{sec:correctors}

In this section, we build a corrector associated to a density at $-\infty$ equal to some $p^0\in [a^0,c^0]$ such that 
\be\label{condp0}
\int_0^1 A(t) dt \geq f^0(p^0).
\ee
Let us recall that a corrector is a time-periodic solution to the mesoscopic model \eqref{eq.meso}, which is equal to $p^0$ at $-\infty$. 

The construction of the corrector relies, on the one hand, on the equivalence between Hamilon-Jacobi equations and conservation laws in one space dimension and, on the other hand, on representation formulas for solutions of Hamilon-Jacobi equations for concave Hamiltonians. We proceed in four steps. We start with a general construction of a periodic in time solution to a Hamilton-Jacobi equation on a half-line $(0,+\infty)$, with a periodic Dirichlet condition at $x=0$ (Lemma \ref{lem.periodicHJgeneric}). We apply this construction to the entry line (road 0) for a junction condition problem (Lemma \ref{lem.w0}). The surprising fact is that this construction can be achieved independently of the outgoing roads 1 and 2. The reason for this is that, in the periodic regime, the flux entering roads 1 and 2 will be at each time the maximal flux coming from road 0: thus no information coming from the outgoing roads is needed to build the solution on the incoming road. Given the  flux exiting road 0,  one can solve the Hamilton-Jacobi problem on the exit lines 1 and 2 (Lemma \ref{lem.defwBIS}) thanks again to the general construction of Lemma \ref{lem.periodicHJgeneric}. In the fourth step we glue the solutions together and show that they form a periodic solution to the conservation law \eqref{eq.meso} (Proposition \ref{prop.correctorfluid} for the fluid regime and Proposition \ref{prop.correctorcongested} for regimes in which one of the outgoing branches is fully congested). 

\subsection{A periodic solution to a HJ equation on a half-line} 

In this section, we assume that 
\be\label{hypfgen}
\begin{array}{c}
\text{$f:[a,b]\to \R$ is a strictly increasing map which is of class $C^2$ and strongly concave:}\\
\text{$f''(p)\leq -\delta<$ for any $p\in [a,b]$, for some constant $\delta>0$}
\end{array}
\ee 
and 
\be\label{hyppsigen}
\begin{array}{c}
\text{$\psi:\R\to \R$ is a Lipschitz continuous map, which is $1-$periodic }\\
\text{and satisfies $\psi'(t)\in [-f(b),-f(a)] \qquad \text{a.e.}\; t\in \R.$}
\end{array}
\ee 
We consider the  Hamilton-Jacobi equation 
\be\label{eq.HJbisTER}
\left\{\begin{array}{lll}
(i) & \partial_xw \in [a,b] & \text{a.e. in} \; \R\times (0,+\infty),\\
(ii) & \partial_t w +  f(\partial_xw)=0& {\rm for }\; t\in \R, \; x>0,\\
(iii)  & w(t,0)=  \psi(t) & {\rm for }\;  \; t\in \R.
\end{array}\right.
\ee

Inspired by the Lax-Oleinik formula and by optimal control on junctions (see for instance \cite{IM17}), we can guess a representation of the solution. The following result checks afterwards that the candidate is indeed the unique solution.

\begin{Lemma}{\bf (Explicit time-periodic solution of the HJ equation)}\label{lem.periodicHJgeneric} \\
Under the assumptions \eqref{hypfgen} and \eqref{hyppsigen} on $f$ and $\psi$, there exists a unique time-periodic Lipschitz continuous viscosity solution $w:\R\times [0,+\infty)\to \R$ to \eqref{eq.HJbisTER} which is of time period equal to $1$. It is given by
\be\label{def.wtxTER}
w(t,x)= \sup_{t_1\leq t}  \psi(t_1)-\xi( t-t_1,x)
\ee
where the map $\xi:[0,\infty)^2\to \R$ is defined by 
$$
\xi(s,y)= \max_{p\in [a,b]} -py+ s f(p)\qquad \forall s\geq 0,\;y\geq 0 . 
$$
\end{Lemma}

\begin{proof}  \noindent{\bf Step 1: Uniqueness of the solution to \eqref{eq.HJbisTER}}.\\
We only sketch the proof, arguing as if the two solutions  $w$ and $\bar w$ of \eqref{eq.HJbisTER} were smooth: the general case can be treated by standard viscosity techniques. Arguing by contradiction, we assume that $\sup w-\bar w>0$. Then we look at the maximum of $w(t,x)-\bar w(t,x)-\ep x^2$ (for $\ep>0$ small). At the maximum point $(t,x)$ one gets $ \partial_t w= \partial_t \bar w$ and $\partial_x w= \partial_x \bar w+2\ep x$ with $x>0$ (since $w=\bar w$ at $x=0$), so that 
$$
0\geq \partial_t w+ f(\partial_x w) -  \partial_t \bar w- f(\partial_x \bar w) =  f(\partial_x \bar w+2\ep x) - f(\partial_x \bar w)>0,
$$
as $x>0$ and $ f$ is increasing. This leads to a contradiction.  \medskip

In order to proceed, we first need to rule out the case in which $\psi$ is constant. In this case the solution to \eqref{eq.HJbisTER} is given by $w(t,x)=\psi+p^*x$ where $p^*\in[a,b]$ is such that $f(p^*)=0$. On the other hand we have by \eqref{hyppsigen} that $0\in [f(a),f(b)]$. Using Lemma \ref{lem.xiC11TER}
below, one can easily check that the optimal $s$ in the expression of $w(t,x)=\psi-\inf_{s\ge0}\xi(s,x)$ is given by $s^*=x/f'(p^*)$ and then $\xi(s^*,x)=-p^*w$ which gives the correct expression for $w$. \medskip

From now on we assume  that $\psi$ is not constant. We note for later use that this  implies that $f(a)< 0$ and $f(b)>0$ because $-\psi'\in [f(a),f(b)]$ and $\psi$ is periodic and not constant. We  suppose in addition that  $\psi$ is of class $C^1$ and satisfies $\psi'(t)<-f(a)$ for any $t\in \R$. This extra condition is removed at the very end of the proof. 
\medskip

\noindent{\bf Step 2: $w$ is globally Lipschitz continuous  on $\R\times [0,+\infty)$}\\
We first note that the sup in the definition of $w$ is in fact a max, because $\psi$ is bounded and, as $\sup_{p\in [a,b]}  f(p)$ is positive, 
\begin{equation}\label{eq::r3TER}
\lim_{t_1\to-\infty} \left\{\inf_{x\in [0,R]} \xi(t-t_1,x)\right\} = +\infty \qquad \forall R>0. 
\end{equation}
In particular $w$ is uniformly bounded on any strip $\R\times [0,R]$.
As explained in  Lemma \ref{lem.xiC11TER}, the map $(s,y)\to \xi(s,y)$ is globally Lipschitz continuous and bounded in $C^{1,1}$ in $ [0,\infty)\times[\ep,+\infty)$ (for any $\ep>0$), 
with
$$
\partial_y\xi(s,y)= -\bar  p, \qquad \partial_s\xi(s,y)=  f(\bar p),
$$
where $\bar p$ is the unique maximum in the definition of $\xi(s,y)$. Since $w$ can be rewritten as 
$$
w(t,x)= \sup_{t_1\in \R}  \psi(t_1\wedge t)-\xi( t-(t_1\wedge t),x)
$$
it is globally Lipschitz continuous on $[0,+\infty)^2$. 
\medskip

\noindent{\bf Step 3: $w$ is locally semiconvex in {time-space}}\\
Next we check that $w$ is locally semiconvex in time-space in $\R\times (0,+\infty)$: we  use this property below to check that $w$ is a solution. This local semiconvexity is not straightforward because $w(t,x)$ is defined as a supremum of an expression on the interval $(-\infty,t]$ which itself depends on the variable $t$. 
In order to overcome this difficulty, we will show that the maximum time $\hat t^1_{t,x}$ in the definition of $w(t,x)$ is indeed strictly less than $t$ (with some bound), which will allow us to replace locally the interval $(-\infty,t]$ by some smaller interval locally independent on $t$. For the proof, let us introduce a few notation. Given $(t,x)\in \R\times (0,\infty)$, let $\hat t^1_{t,x}\leq t$ be a maximum point in the definition of $w(t,x)$ and $\hat p_{t,x}\in [a,b]$ be the unique maximum point in the definition of $\xi(t-\hat t^1_{t,x},x)$. We next claim that, for any $0<\ep<1$, there exists $\eta>0$ such that, if $x\in [\ep,1/\ep]$, then $\hat t^1_{t,x}\leq t-\eta$. Indeed, otherwise, there exists a sequence $(t_n,x_n)$ such that $x_n\in [  \ep, 1/\ep]$, $\hat t^1_{t_n,x_n}> t_n-1/n$. By periodicity we can assume without loss of generality that $t_n\in [0,1]$ and converges to some $t$ and that $(x_n)$ converges to some $x\in [\ep,1/\ep]$. Then $\hat t^1_{t_n,x_n}$ converges to $t$, which is a maximum point in the definition of $w(t,x)$, and $\hat p_{t_n,x_n}$ converges to some $\bar p\in [a,b]$, which is the unique maximum point in the definition of $\xi(0,x)$. As $\hat t^1_{t,x}=t$ is a maximum for $w(t,x)$, we get by the optimality conditions (using the additional regularity $ \psi \in C^1$), 
$${
 \psi'(t)\geq - \partial_s \xi(0, x) =  -f(\bar p)= - f(a), 
}$$
because the unique maximizer $\bar p$ of $p\to -px$ on $[a,b]$ is $\bar p=a$. This contradicts our additional assumption that $ \psi'< -f(a)$ and  shows that there exists $\eta>0$ such that, if $x\in [\ep,1/\ep]$, then $\hat t^1_{t,x}\leq t-\eta$. 

As a consequence, given $(t,x)\in \R\times (0,\infty)$, there exists a neighborhood $\mathcal V$ of $(t,x)$ and $\eta'>0$ such that, 
$$
w(s,y)=  \sup_{t_1\leq t-\eta'}  \psi(t_1)-\xi(s-t_1,y),\qquad y\geq x/2\qquad \forall (s,y)\in \mathcal V.
$$
Note that the upper bound for $t_1$ in the above problem is now independent of $(s,y)$. Recalling that $\xi$ is bounded in $C^{1,1}$ in $ [0,\infty)\times[x/2,\infty)$, this shows the semiconvexity of $w$ in $\mathcal V$. 
\medskip

\noindent{\bf Step 4: $w$ is solution of \eqref{eq.HJbisTER}.}\\
As $ f$ is uniformly concave and $w$ locally semiconvex, $w$ satisfies the equation in \eqref{eq.HJbisTER} in the viscosity sense if and only if it satisfies this equation  at any point of differentiability. Let $(t,x)\in \R\times (0,\infty)$ be  a point of differentiability of $w$. By the envelop theorem (Theorem \ref{th::r1}), for any   optimizer $\hat t^1_{t,x}<t$ for $w(t,x)$ and if $\hat p_{t,x}\in [a,b]$  is the unique maximizer for $\xi(t-\hat t^1_{t,x},x)$, we get
$$
\partial_xw(t,x)=-\partial_y \xi(t-\hat t^1_{t,x},x)= \hat p_{t,x} ,\qquad \partial_t w(t,x)= -\partial_s \xi(t-\hat t^1_{t,x},x)= - f(\hat p_{t,x}).
$$
Thus 
$$
\partial_t w+ f(\partial_xw) = -  f(\hat p_{t,x}) +  f( \hat p_{t,x} )=0.
$$
This shows that $w$ satisfies the equation in \eqref{eq.HJbisTER} and that $\partial_xw\in [a,b]$ a.e..

For the boundary condition, we first note that (choosing $t_1=t$ as a competitor)
$$
w(t,0)\geq - \psi(t)-\xi(0,0)=  \psi(t).  
$$
Moreover, 
$$
w(t,0)=    \psi(\hat t^1_{t,0})-(t-\hat t^1_{t,0}) \max_{p\in[a,b]}  f(p)=  \psi(\hat t^1_{t,0})-(t-\hat t^1_{t,0}) f(b). 
$$
If, contrary to our claim, we had $w(t,0)> \psi(t)$, then one would have $\hat t^1_{t,0}<t$ and 
$$
(t-\hat t^1_{t,0}) f(b) <  \psi(\hat t^1_{t,0}) -  \psi(t)= - \int_{\hat t^1_{t,0}}^t\psi'(s)ds \leq  (t-\hat t^1_{t,0}) f(b), 
$$
which is impossible since $  f(b) >0$. Hence $w(t,0)=  - \psi(t)$.
\medskip

\noindent {\bf Step 5: Conclusion.}\\
We finally remove the extra assumption that $\psi\in C^1$ and satisfies $\psi'<-f(a)$: let $(\psi^n)$ be a sequence of smooth periodic maps satisfying $-f(b)\leq (\psi^n)'<-f(a)$ and which converges to $\psi$ (such a sequence exists since $-f(b)\leq \psi'\leq -f(a)$ a.e.). Let $w^n$ be given by \eqref{def.wtxTER} for $\psi^n$ in place of $\psi$. Then $w^n$ solves the HJ equation for $\psi^n$ and, by stability, converges locally uniformly to the unique viscosity solution of the problem with $\psi$. Note that \eqref{eq.HJbisTER}-(i) holds as well by $L^\infty-*$ convergence of $\partial_x w^n$ to $\partial_xw$. 
\end{proof}

It remains to state and check the intermediate lemma. 

\begin{Lemma}{\bf (Properties of the fundamental solution $\xi$)}\label{lem.xiC11TER}\\
The map $\xi$ defined by 
$$
\xi(s,y)= \max_{p\in [a,b]} -py+ s f(p)\qquad \forall s\geq 0,\; y\geq 0. 
$$
is globally Lipschitz continuous in $[0,\infty)\times [0,\infty)$ and  bounded in $C^{1,1}$ in $[0, \infty)_s\times [\ep,\infty)_y$ for any $\ep>0$.   Moreover,  $\xi$ is differentiable at any $(s,y)$ with $s>0$ and
\begin{equation}\label{eq:001TER}
\partial_y\xi(s,y) = -\hat p_{s,y} , \qquad \partial_s\xi(s,y) = f( \hat p_{s,y}),
\end{equation}
where $\hat p_{s,y}$ is the unique point of maximum in the definition of $\xi(s,y)$ and is given by  
\begin{equation}\label{eq::r4TER}
\hat p_{s,y} = \left\{\begin{array}{cl}
( f')^{-1}(y/s) & {\rm if } \; y/s\in ( f'(b),  f'(a)) \\ 
b & {\rm if } \;  y/s \leq  f'(b) \\
a & {\rm if }  \;  y/s \geq  f'(a)
\end{array}\right.
\end{equation}
\end{Lemma}

\begin{proof} As $ f$ is increasing and strongly concave, the point of maximum $\hat p_{s,y}$ in the definition of $\xi(s,y)$ is unique for $s>0$ and $y\in [0,\infty)$ and given by \eqref{eq::r4TER}. Thus, by the envelope theorem (Theorem \ref{th::r1}), $\xi$ is differentiable at any $(s,y)$ with $s>0$ and its derivatives are given by \eqref{eq:001TER}.
As $\hat p_{s,y}$ is bounded, this implies that $\xi$ is globally Lipschitz continuous in $[0,\infty)\times[0,\infty)$.\medskip

It remains to show that $(s,y)\to \hat p_{s,y} $ is Lipschitz continuous in $[0, \infty)\times [\ep,\infty)$  (where $\ep>0$ is fixed). As $ f$ is strongly concave, $ f'$ is decreasing. Since $f$ is increasing, this implies that $ f'(a)>0$.

Using again that $ f$ is strongly concave with $ f''\le -\delta<0$, we see that $-C_0\leq (( f')^{-1})'\leq 0$ with $C_0=1/\delta$. Let $(s,y),(s',y')\in (0,\infty)\times [\varepsilon,\infty)$ be such that (to fix the ideas) $y/s\leq y'/s'$, and then $\hat p_{s',y'}\le \hat p_{s,y}$. The idea consists in using $\varepsilon$ in order to control $y,y'$, which will in turn control also $s,s'$ in some sense.

Without loss of generality we can also assume that $y/s<f'(a)$ since otherwise  $\hat p_{s,y}=a=\hat p_{s',y'}$. We have
$$
\left| \hat p_{s',y'} -\hat p_{s,y} \right| \leq C_0 \left|\frac{y'}{s'}\wedge  f'(a)-\frac{y}{s}\vee  f'(b)\right| \leq C_0 \left(\frac{y'}{s'}\wedge  f'(a)-\frac{y}{s}\right). 
$$
Let us first suppose that $y/s,y'/s'\leq  f'(a)$. As $y,y'\geq \ep$, we get $1/s'\leq  f'(a)/y'\leq  f'(a)/\ep$. Hence
$$
\left| \hat p_{s',y'} -\hat p_{s,y} \right| \leq C_0  \left( \frac{1}{s'}|y'-y| + \frac{y}{ss'}|s-s'| \right) 
\leq C_0  \left( \frac{ f'(a)}{\ep}|y'-y| + \frac{( f'(a))^2}{\ep}|s-s'| \right) . 
$$
Finally, if $y'/s' \geq  f'(a)$ and $y/s <  f'(a)$, then 
\begin{align*}
\left| \hat p_{s',y'} -\hat p_{s,y} \right| & \leq C_0 \left( f'(a)-\frac{y'}{s}+\frac{y'}{s}-\frac{y}{s}\right) 
 \leq
C_0 \left( f'(a)(1- \frac{s'}{s})+\frac{ f'(a)}{\ep}|y'-y| \right) \\
& \leq  C_0 \left(\frac{( f'(a))^2}{\ep}|s'-s|+\frac{ f'(a)}{\ep}|y'-y| \right) .
\end{align*}
This shows that the map $(s,y)\to \hat p_{s,y}$ is Lipschitz continuous in  $(0, \infty)\times [\ep,\infty)$, and thus on $[0,+\infty)\times [\varepsilon,+\infty)$. Therefore $\xi$ is bounded in $C^{1,1}$ in this set.
\end{proof}

In order to show that the correctors will have the good behavior at infinity, we have to examine carefully the behavior of the solution of the HJ equation at infinity.
\begin{Lemma}[behavior of the solution at $\infty$]\label{lem.behaviorinfty} Assume that conditions \eqref{hypfgen} and \eqref{hyppsigen} on $f$ and $\psi$ hold and that $0\in [a,b]$ with $f(0)=0$. Then the solution $w$ of \eqref{eq.HJbisTER} is bounded and there exists a constant $C>0$ such that 
\be\label{condinftyw}
\|\partial_x w\|_{L^\infty(\R\times (M,\infty))} \leq \frac{C}{M} \qquad \forall M\geq C.
\ee
\end{Lemma}

\begin{Remark} We can actually show that there exists a constant $C>0$ such that 
$$
\|w-\max {\psi}\|_{L^\infty(\R\times (M,\infty))}\leq  \frac{C}{M} \qquad \forall M\geq C.  
$$
The bound $w\leq \max {\psi}$ follows by comparison, while the other bound is {obtained}  using the  uniform concavity of $f$ in the representation formula. 
\end{Remark}

\begin{proof} We can assume without loss of generality that $a<0<b$ since, if $a=0$ or $b=0$, then by \eqref{hyppsigen} $\psi$ must be constant and therefore, since $f(0)=0$, $w=\psi$ is also constant.\medskip

As $w^+(t,x)=\| \psi\|_\infty$ and $w^-(t,x)= -\| \psi\|_\infty$ are respectively time-periodic  super- and sub-solution of the equation, we have  $|w|\leq \| \psi\|_\infty$  by comparison. 
\medskip

We now turn to the proof of \eqref{condinftyw}. Given any $(t,x)\in \R\times (0,+\infty)$ a point of differentiability of $w$, consider some optimizer $\hat t_1\leq t$ for $w(t,x)$ and $\hat p$ the optimizer in the definition of $\xi(t-\hat t_1,x)$. From the proof of Lemma \ref{lem.periodicHJgeneric}, we know that $\hat t_1<t$ and that $\partial_x w(t,x)= \hat p$. So, to prove \eqref{condinftyw}, we just need to estimate $\hat p$. 

Recalling  Lemma \ref{lem.xiC11TER} again,  we have
$$
\partial_s\xi(s,y)= f(\hat p)\quad \mbox{where}\quad \hat p=-\partial_x \xi(s,y)=  (f')^{-1}((T_{ f'(b)}^{ f'(a)}(y/s)),\quad T_\alpha^\beta(z)=\max(\alpha,\min(\beta,z)). 
$$
Hence $ f(0)=0$ implies
$$\left\{\begin{array}{ll}
\partial_s\xi(s,y)<0 &\quad \mbox{if}\quad y/s> f'(0),\\
\partial_s\xi(s,y)>0 &\quad \mbox{if}\quad y/s< f'(0).\\
\end{array}\right.$$

We claim that $x/(t-\hat t_1+1) \leq  f'(0)$. Indeed, otherwise, $x/(t-\hat t_1)\geq x/(t-\hat t_1+1) >  f'(0)$ and thus $\xi(\cdot, x)$ is decreasing on $[t-\hat t_1,t-\hat t_1+1]$. This implies, as $\psi$ is $1-$periodic, that 
$$
\psi(t-\hat t_1+1)-\xi(t-\hat t_1+1, x) > \psi(t-\hat t_1)-\xi(t-\hat t_1, x) = w(t,x), 
$$
a contradiction because $t_1= \hat t_1-1$ is a competitor in the definition of $w(t,x)$. Thus $x/(t-\hat t_1+1) \leq  f'(0)$. 

In the same way one can check that, if $\hat t_1+1<t$, then $x/(t-\hat t_1-1) \geq  f'(0)$, using $t_1= \hat t_1+1$ as a competitor in the definition of $w(t,x)$. Let us check that indeed $\hat t_1+1<t$ if $x$ is large enough: otherwise, $|t-\hat t_1|\leq 1$ and therefore
\begin{align*} 
w(t,x) & =  \psi(t-\hat t_1) +  \min_{p\in [a,b]} \{px-(t-\hat t_1) f(p)\}  \leq \| \psi\|_\infty+\| f\|_\infty+ \min_{p\in [a,b]} px \\
&=\| \psi\|_\infty+\| f\|_\infty +ax, 
\end{align*}
which yields to a contradiction if $x$ is large enough, because $a<0$ and $w$ is bounded.

The two estimates on $x/(t-\hat t_1+1)$ and $x/(t-\hat t_1-1)$  imply that, for $x$ large enough, 
$$
|x- f'(0)(t-\hat t_1)| \leq  f'(0), 
$$
where $ f'(0)>0$. 
Thus, for $x$ large enough,  $x/(t-\hat t_1)$ is close to $  f'(0)\in (f'(b), f'(a))$ and therefore for $x$ large enough
$$
|\hat p| = \left| ( f')^{-1}\left(T_{ f'(b)}^{ f'(a)}\left(\frac{x}{t-\hat t_1}\right)\right)\right|= \left| ( f')^{-1}\left(\frac{x}{t-\hat t_1}\right)-( f')^{-1}( f'(0)) \right| \leq C\left|\frac{x}{t-\hat t_1}- f'(0)\right| \leq 
\frac{C}{t-\hat t_1}\leq \frac{C}{x}. 
$$
\end{proof}

\subsection{Periodic solutions to a HJ equation on the entry line} 

We  build in  this part an antiderivative of the corrector on the incoming road $\mathcal R^0$. We suppose here that $f^0$ satisfies \eqref{hypflux} and that the flux limiter $A$ satisfies \eqref{hypIj} and \eqref{hypfluxlimiterA}. For $p^0\in [a^0,c^0]$ such that \eqref{condp0} holds, let
\be\label{deftildef0}
\tilde f^0_{p^0}(p)= f^0(p+p^0)-f^0(p^0)\qquad \text{\rm for}\; p\in [a^0-p^0, c^0-p^0],
\ee
so that $\tilde f^0_{p^0}(0)=0$ and $0\in [a_0-p_0,c_0-p_0]$.
We consider the periodic in time viscosity solution $w^0_{p^0}$ to the HJ equation 
\be\label{HJ0}
\left\{\begin{array}{lll}
(i) & \partial_xw^0 \in [a^0-p^0, c^0-p^0] & \text{a.e. in}\; \R\times (-\infty,0),\\
(ii) & \partial_t w^0 + \tilde f^0(\partial_xw^0)=0& {\rm for }\; t\in \R, \; x<0\\
(iii) & \partial_t w^0+ \min\{ A(t)-f^0(p^0),  \tilde f^{0,+} (\partial_xw^0)\}=0 & {\rm for }\; t\in \R, \; x=0
\end{array}\right.
\ee
By a solution, we mean that $w^0_{p^0}$ is continuous on $[0,+\infty)\times (-\infty,0]$ and is a viscosity solution  in the sense of \cite{IM17} to \eqref{HJ0}-(i)-(ii)-(iii) on  each open interval on which $A$ is constant. It is easy to check that the whole theory developed in \cite{IM17} generalizes to this simple time-dependent setting. { Notice that, if $w^0$  is a solution of (\ref{HJ0}), then $w^0+c$ is also a solution for any constant $c\in \R$.} Still we have the following existence result.

\begin{Lemma}[Explicit time-periodic solution in the incoming road] \label{lem.w0} Assume that $f^0$ satisfies \eqref{hypflux} and that \eqref{hypIj} and \eqref{hypfluxlimiterA} hold. Let $p^0$ be such that $p^0\in [a^0,b^0]$ and \eqref{condp0} holds, or  $p^0= (f^{0,-})^{-1}\left(\int_0^1A(s)ds\right)$. Then there exists a  bounded, Lipschitz continuous and time-periodic solution $w^0_{p^0}$ to \eqref{HJ0}, with period $1$, which is given by the representation formula
\be\label{def.w0}
w^0_{p^0}(t,x)= 
\left\{ \begin{array}{ll}
\max \Bigl\{ 0, \max_{t_2\leq t} \{ \psi_{p^0}(t_2)  - \xi^0_{p^0}(t-t_2,x)\} \Bigr\}, & \text{if $p^0\in [a^0,b^0]$,}\\
\max_{ t_2\leq t} \Bigl\{ \psi_{p^0}(t_2) - \xi^0_{p^0}(t-t_2,x)\} \Bigr\}, & \text{if $p^0=(f^{0,-})^{-1}\left(\int_0^1A(s)ds\right)$,}
\end{array}\right.
\ee
where
$$\xi^0_{p^0}(s,y)= \max_{p\in [b^0-p^0,c^0- p^0]} -py+ s\tilde f^0_{p^0}(p)\qquad \forall s\geq  0,\; y\leq 0$$
and 
\be\label{def.psiTER}
\psi_{p^0}(t) = \left\{\begin{array}{ll}
\displaystyle \max_{t_1\leq t} \left\{\int_{t_1}^{t} (f^0(p^0)-A(s))ds\right\} & \text{if $p^0\in [a^0,b^0]$,}\\
\displaystyle \int_0^{t} (f^0(p^0)-A(s))ds & \text{if $p^0=(f^{0,-})^{-1}\left(\int_0^1A(s)ds\right)$.}
\end{array}\right.
\ee
In addition,
 there exists a constant $C>0$ (depending on $p^0$), such that 
\be\label{condinftyw0}
\begin{array}{ll}
w^0_{p^0}(t,x) = 0 \qquad \text{for}\; x\leq -C, \; t\in \R & \text{if}\; p^0\in [a^0,b^0],\\
\;\\
\text{and} \qquad \|\partial_x w^0_{p^0}\|_{L^\infty(\R\times (-\infty,M))}\leq \frac{C}{M} \qquad \text{for}\; M\geq C & \text{if}\; p^0=(f^{0,-})^{-1}\left(\int_0^1A(s)ds\right) {\quad \mbox{with}\quad p^0\in  (b^0,c^0]}. 
\end{array}
\ee
Finally,  
$$
w^0_{p^0}(t,0)=  \psi_{p^0}(t)
\qquad \forall t\in \R, 
$$
and, if $p^0\in [a^0,b^0]$, the map $x{\mapsto}  w^0_{p^0}(t,x)$ is nondecreasing on $(-\infty,0]$ for any $t\in \R$. 
\end{Lemma}

Recall that the map $\psi_{p^0}$ (for $p^0\in [a^0,b^0]$) was introduced in Lemma \ref{lem.analysew00} when building the homogenized germ $\mathcal G_{\bar \Lambda}$. 

{\begin{Remark}\label{rem::b205}
Notice that in case $p^0\in (b^0,c^0]$, it is possible {to} construct explicit examples of solutions where $\partial_x w^0_{p_0}(t,x)$ has no compact support in the space variable $x$, but tends to zero as $x\to -\infty$.
\end{Remark}}

\begin{proof} Note first that, if $p^0=b^0$ satisfies \eqref{condp0}, then $w^0_{p^0}={0}$ is the solution to \eqref{HJ0} because in this (very particular) case, assumption \eqref{hypfluxlimiterA} implies that $A(t)=f^0_{\max}$. From now on we assume that 
$${p^0\neq b^0.}$$
As $p^0$ is fixed, we remove the subscript $p^0$ throughout the proof for simplicity of notation. Note that, if $p^0\in [a^0,b^0)$,  $0< b^0-p^0<c^0-p^0$ and thus the map $y{\mapsto} \xi^0(s,y)$ is decreasing on $(-\infty,0]$ for any $s\geq 0$. Hence  the map $x{\mapsto} w^0_{p^0}(t,x)$ is nondecreasing on $(-\infty,0]$ for any $t\in \R$. \\

\noindent{\bf Step 1: $w^0$ is a viscosity solution to the HJ equation \eqref{HJ0}-(ii).} If $p^0\in [a^0,b^0)$ is such that \eqref{condp0} holds, Lemma \ref{lem.analysew00} states that the map $\psi_{p^0}$ is Lipschitz continuous and $1-$periodic and we can rewrite ${w^0=w^0_{p^0}}$ in the form 
\begin{align*}
w^0(t,x) & = \max \Bigl\{ 0, \tilde w^0(t,x) \Bigr\} \qquad {\rm with}\qquad 
\tilde w^0(t,x)= \max_{t_2\leq t} \{ \psi_{p_0}(t_2) - \xi^0(t-t_2,x)\}. 
\end{align*}
In the case $p^0=(f^{0,-})^{-1}(\int_0^1A(s)ds)$, the map $\psi_{p^0}$ is also Lipschitz continuous and $1-$periodic and we set $\tilde w^0:= w^0$. 
%
%
Our aim is to use Lemma \ref{lem.periodicHJgeneric} to check that $\tilde w^0$ is a viscosity solution to the HJ equation \eqref{HJ0}-(ii). For this we change variable and set 
$$
\hat w^0(t,x)= \tilde w^0(t,-x)= \max_{t_2\leq t} \{ \psi_{p^0}(t_2) -\hat  \xi^0_{p^0}(t-t_2,x)\}, \qquad t\geq 0, \; x\geq 0,
$$
where 
$$
\hat \xi^0(s,y) = \max_{p\in [-c^0+p^0,-b^0+p^0]} -py +s \tilde f^0(-p)\qquad  s\geq 0, \; y\geq 0. 
$$
Note that the map $p {\mapsto} \tilde f^0(-p)$ is uniformly concave and strictly increasing on $[-c^0+p^0,-b^0+p^0]$. In addition, the maps $\psi_{p^0}$ defined in \eqref{def.psiTER} is Lipschitz continuous, $1-$periodic and satisfies, by Lemma \ref{lem.analysew00}, 
$$
\psi_{p^0}'(t)\in\{0, f^0(p^0)-A(t)\} \subset [-\tilde f^0(-(-b^0+p^0)), -\tilde f^0(-(-c^0+p^0))]\qquad \text{a.e.}\; t\in \R,
$$
where, for the proof of the inclusion, we used \eqref{hypfluxlimiterA} and the equality 
$$
[-\tilde f^0(-(-b^0+p^0)), -\tilde f^0(-(-c^0+p^0))] = [-f^0_{\max}+f^0(p^0), f^0(p^0)].
$$
Therefore we can apply Lemma \ref{lem.periodicHJgeneric} which states that $\hat w^0$ is globally Lipschitz continuous,  $1-$periodic in time, and satisfies the HJ equation \eqref{eq.HJbisTER} in $\R\times (0,+\infty)$ for $f(p)= \tilde f^0(-p)$ and the boundary condition $\hat w^0(\cdot,0)= \psi_{p^0}$.
This implies that $\tilde w^0(t,x)= \hat w^0(t,-x)$ is a Lipschitz continuous viscosity solution of \eqref{HJ0}-(i) and \eqref{HJ0}-(ii) in $\R\times (-\infty,0)$, with $\tilde w^0(\cdot,0)=\psi_{p^0}$.  

Assume now that $p^0\in [a^0,b^0)$. As $\tilde f^0$ is concave and the constant map $(t,x){\mapsto} 0$ is also a solution of \eqref{HJ0}-(ii) in $\R\times (-\infty,0)$, $w^0$ is also a viscosity solution of \eqref{HJ0}-(ii) in $\R\times (-\infty,0)$. In addition, \eqref{HJ0}-(i) holds since $0\in [a^0-p^0,c^0-p^0]$ and $\partial_x \tilde w^0\in [a^0-p^0,c^0-p^0]$. Note finally that { $w^0(\cdot,0)= \psi_{p^0}(\cdot)$} as $\psi_{p^0}\geq 0$. 
\medskip

\noindent{\bf Step 2: $w^0$ bounded and satisfies \eqref{condinftyw0}.} 
As $\tilde f(0)=0$, Lemma \ref{lem.behaviorinfty} states that $\hat w^0$ and thus $w^0$ are bounded.\\
Let us {first} assume that $p^0\in [a^0,b^0)$. 
Fix $x<0$ and $t_2\leq t$.  Then 
\begin{align*}
\min_{p\in [b^0-p^0, c^0-p^0]} px -(t-t_2)\tilde f^0(p) & \leq  (b^0-p^0)x + \min_{p\in [b^0-p^0, c^0-p^0]}\{ -(t-t_2) \tilde f^0(p)\}\\
& =  (b^0-p^0)x -(t-t_2)  (f^0_{\max}-f^0(p^0)). 
\end{align*}
Thus 
$$\tilde w^0(t,x)\leq 
(b^0-p^0)x +  \max_{t_1\leq t_2\leq t}  \{ -\int_{t_1}^{t_2}(A(s)-f^0(p^0))ds-(t-t_2)  (f^0_{\max}-f^0(p^0))\}.
$$
We note that the map 
$$
t {\mapsto} \max_{t_1\leq t_2\leq t}  \{ -\int_{t_1}^{t_2}(A(s)-f^0(p^0))ds-(t-t_2) (f^0_{\max}-f^0(p^0))\}
$$
is a continuous, periodic function. Hence it is bounded. As $p^0<b^0$, this shows the existence of $C>0$ such that, for any $x\leq -C$ and $t\in \R$,  $
\tilde w^0(t,x) \leq 0$. Therefore \eqref{condinftyw0} holds {in this case.}\\
{
In the case $p^0=b^0$, it is easy to see that $A=f^0_{\max}$ and then $\psi_{p^0}=0$, which implies that $w^0_{p^0}=0$ is solution.
Hence \eqref{condinftyw0} holds in this case.\\
Finally, we consider the case $p^0=(f^{0,-})^{-1}\left(\int_0^1A(s)ds\right)> b^0$.  Then  \eqref{condinftyw0}  follows from Lemma \ref{lem.behaviorinfty} and a change of variables.} \medskip

\noindent {\bf Step 3: $w^0$ satisfies the boundary condition  \eqref{HJ0}-(iii).}\\
 For proving the supersolution property, we just need to check that $w^0(\cdot,0)$ is Lipschitz continuous and satisfies $\partial_t w^0(t,0)+ A(t)-f^0(p^0)\geq 0$ a.e. (cf. \cite[Theorem 2.11]{IM17}). Recalling {that} ${w^0(\cdot,0)= \psi_{p^0}(\cdot)}$, this inequality is obvious if $p^0=(f^{0,-})^{-1}(\int_0^1A(s)ds)$. If $p^0\in [a^0,b^0)$, it holds thanks to Lemma \ref{lem.analysew00}. 
%
%

Next we turn to the subsolution property. Assume that $\varphi(t,x):=\alpha(t)+q^0 x$  is a $C^1$ test function touching $w^0$  from above at $(t_0,0)$, where $t_0$ is a point of continuity of $A$ and with (condition (2.12) in  \cite[Theorem 2.7]{IM17})
\be\label{zkejnsrdkgBIS}
A(t_0)-f^0(p^0)= \tilde f^0(q^0)=\tilde f^{0,-}(q^0).
\ee 
We have to prove that $\alpha'(t_0)+ A(t_0)-f^0(p^0)\leq 0$. Without loss of generality, we can assume that $\alpha(t_0)=w^0(t_0,0)$. 

Assume first that $p^0=(f^{0,-})^{-1}(\int_0^1A(s)ds)$. Then, for $h\in \R$,  
$$
\alpha(t_0+h)=\varphi(t_0+h,0)  \geq w^0(t_0+h,0) = w^0(t_0,0)+ \int_{t_0}^{t_0+h}  (f^0(p^0)-A(s)) ds= \alpha(t_0)+ \int_{t_0}^{t_0+h}  (f^0(p^0)-A(s)) ds, 
$$
which proves that $\alpha'(t_0)=  f^0(p^0)-A(t_0)$. 

We now assume that $p^0\in [a^0,b^0)$. Let $\hat t_1\leq t_0$ be optimal in the definition of $\psi_{p^0}(t_0)$ in \eqref{def.psiTER}. We claim that $\hat t_1< t_0$. Indeed, otherwise, $\hat t_1=t_0$ and thus $w^0(t_0,0)=0=\alpha(t_0)$. So, for any $x<0$, 
$$
\varphi(t_0, x) =  q^0 x \geq w^0(t_0,x) \geq 0,
$$
which implies that $q^0\leq 0$. But \eqref{zkejnsrdkgBIS} says that $q^0\in [b^0-p^0,c^0-p^0]$, where $b^0-p^0>0$, a contradiction. 

As $\hat t_1<t_0$, for any $h\in \R$ with $|h|$ small, 
\begin{align*}
\alpha(t_0+h)=\varphi(t_0+h,0) & \geq w^0(t_0+h,0)=\psi_{p^0}(t_0+h)\geq  \int_{\hat t_1}^{t_0+h}  (f^0(p^0)-A(s)) ds \\
& = w^0(t_0,0) + \int_{t_0}^{t_0+h}  (f^0(p^0)-A(s)) ds=\alpha(t_0) + \int_{t_0}^{t_0+h}  (f^0(p^0)-A(s))ds.
\end{align*}
Hence $\alpha'(t_0)+A(t_0)-f^0(p^0)= 0$.
\end{proof}

The next step is the computation of the trace $f^{0,+}(\partial_x w^0_{p^0}(t,0^-)+p^0)$, where $w^0_{p^0}$ is the solution of \eqref{HJ0}. The computation of this trace will be useful for gluing the correctors on each branch. Let us note that, as $w^0_{p^0}$ is a Lipschitz continuous viscosity solution to \eqref{HJ0}, Lemma \ref{lem::r2} states that $\partial_x w^0_{p^0}$ is a Krushkov entropy solution to the scalar conservation law 
$$
\left\{\begin{array}{lll}
(i) & \rho \in [a^0-p^0, c^0-p^0] & \text{a.e. in}\; \R\times (-\infty,0),\\
(ii) & \partial_t \rho + \partial_x( \tilde f^0(\rho))=0& {\rm for }\; t\in \R, \; x<0.
\end{array}\right.
$$
Thus $\partial_x w^0_{p^0}$ possesses a trace, denoted as $\partial_x w^0_{p^0}(\cdot, 0^-)$, at $x=0$ (Theorem \ref{th::r7homo}), in the sense that there exists a set $\mathcal N$ of  measure zero in $(-\infty,0)$ such that, for any $t_1<t_2$,  
\be\label{tracew0x}
\lim_{\varepsilon\to 0}\sup_{x\in (-\varepsilon,0)\backslash \mathcal N} \|\partial_x w^0_{p^0}(\cdot,x )-\partial_x w^0_{p^0}(\cdot, 0^-)\|_{L^1(t_1,t_2)}=0.
\ee
By continuity of $f^{0,+}$, we infer the existence of the trace $f^{0,+}(\partial_x w^0_{p^0}(t,0^-)+p^0)$.

\begin{Lemma}[Computation of the trace $f^{0,+}(\partial_x w^0_{p^0}(t,0^-)+p^0)$]\label{lem.f0+}  Under the assumption of Lemma \ref{lem.w0}, let $w^0_{p^0}$ be the solution of \eqref{HJ0} given in Lemma \ref{lem.w0}. Then 
\be\label{repf0+}
f^{0,+}(\partial_x w^0_{p^0}(t,0^-)+p^0)
=\left\{\begin{array}{ll}
\displaystyle f^0(p^0) & \text{if $w^0_{p^0}(t,0)=0$ and $p^0\in [a^0,b^0)$}\\
\displaystyle f^0_{\max} & \text{if $w^0_{p^0}(t,0)>0$ or $p^0= (f^{0,-})^{-1}\left(\int_0^1A(s)ds\right)$ {or $p^0=b^0$}} 
\end{array}\right.\; \text{a.e.}\; t\in \R
\ee
and
\be\label{w0p0t}
\begin{array}{ll}
\partial_t w^0_{p^0}(t,0)& = - \min\Bigl\{A(t)-f^0(p^0) \ , \ \tilde f^{0,+}(\partial_x w^0_{p^0}(t,0^-))\Bigr\} \qquad \text{a.e. in } \R,\\
& = f^0(p^0)-F_{p^0}(t) \qquad \text{a.e. in } \R,
\end{array}
\ee
where the  flux $F_{p^0}$ is defined in \eqref{defFp0} for $p^0\in [a^0, b^0)$ and by 
\be\label{defFp0BIS}
\text{$F_{p^0}(\cdot)=A(\cdot)$ if $p^0= (f^{0,-})^{-1}\left(\int_0^1A(s)ds\right)$ {or $p^0=b^0.$}}
\ee
\end{Lemma}

\begin{proof}[Proof of Lemma \ref{lem.f0+}] In the case where $p^0= (f^{0,-})^{-1}\left(\int_0^1A(s)ds\right)$ {or $p^0=b^0$} 
the proof is quite simple.
Indeed, in those two cases, we have saturation, i.e. $\partial_x w^0_{p^0} +p^0 \in [b^0,c^0]$ a.e., and then $f^{0,+}(\partial_x w^0_{p^0} +p^0)=f^0_{\max}$, which shows (\ref{repf0+}). Moreover, we have 
$$\partial_tw^0_{p^0}(t,0^-)=-\psi_{p^0}'(t)= \left\{\begin{array}{l}
f^0(p^0)-F_{p^0}(t)=0= f^0(p^0)-A(t)  \qquad \mbox{by Lemma \ref{lem.analysew00} and (\ref{hypfluxlimiterA})} \; \mbox{if $p^0=b^0$}\\
f^0(p^0)-A(t)  \qquad \mbox{by (\ref{def.psiTER})}\; \mbox{if $p^0= (f^{0,-})^{-1}\left(\int_0^1A(s)ds\right)$}\\
\end{array}\right.$$
which shows (\ref{w0p0t}).\\
We now prove the results in the case 
$$p^0\in [a^0,b^0).$$

\medskip

\noindent{\bf Step 1: Proof of \eqref{repf0+}.} We first claim that 
\be\label{aezksrjdhfn}
f^{0,+}(\partial_x w^0_{p^0}(t,x)+p^0) =\left\{\begin{array}{ll}
f^0(p^0) & \text{if} \; w^0_{p^0}(t,x)=0\\
f^0_{\max} & \text{if} \; w^0_{p^0}(t,x)>0
\end{array}\right.\; \text{a.e.}\; (t,x)\in \R\times (-\infty,0).
\ee
To prove \eqref{aezksrjdhfn}, let $(t,x)\in \R\times (-\infty,0)$ be a point of differentiability of the Lipschitz map $w^0_{p^0}$.  Then, if $w^0_{p^0}(t,x)=0$, we get $\partial_x w^0_{p^0}(t,x)=0$ since $w^0_{p^0}\geq0$, and thus \eqref{aezksrjdhfn} holds in this case. Let us now assume that $w^0_{p^0}(t,x)>0$. Let $\hat t_2\leq t$ be optimal in the definition of $w^0_{p^0}$ in \eqref{def.w0}.  We have already proved (see Step 3 in the proof of Lemma \ref{lem.periodicHJgeneric}) that $\hat t_2<t$. Then $\xi_{p^0}$ is differentiable at $(t-\hat t_2,x)$ with, by the envelope theorem  \ref{th::r1} used twice, 
$$
\partial_x w^0_{p^0}(t,x)= -\partial_x \xi^0_{p^0}(t-\hat t^2,x) = \hat p\in [b^0-p^0,c^0-p^0],
$$
where $\hat p$ is optimal for $\xi^0_{p^0}(t-\hat t^2,x)$.  As $f^{0,+}([b^0, c^0])=\{f^0_{\max}\}$, this shows  \eqref{aezksrjdhfn}. \\

Fix $t\in\R$. Recalling that $x{\mapsto} w^0_{p^0}(t,x)$ is nondecreasing and nonnegative, equality $w^0_{p^0}(t,0)=0$ implies that $w^0_{p^0}(t,x)=0$ for any $x\leq 0$. Thus 
$$
\lim_{x\to 0^-} {\bf 1}_{\{w^0_{p^0}(t,x)>0\} } = {\bf 1}_{\{w^0_{p^0}(t,0)>0\}}.  
$$
Combining the remark above with \eqref{tracew0x} and \eqref{aezksrjdhfn} gives \eqref{repf0+}. \medskip

\noindent{\bf Step 2: proof of \eqref{w0p0t}.} We recall that $w^0_{p^0}(\cdot,0)=\psi_{p^0}(\cdot)$. Thus, by Lemma \ref{lem.analysew00},  
$$
\partial_t w^0_{p^0}(t,0)= -(A(t)-f^0(p^0)) \qquad \text{for a.e. $t\in \R$ with $w^0_{p^0}(t,0)>0$.}
$$
On the other hand, if $w^0_{p^0}(t,0)>0$, then by \eqref{repf0+}
$$
 \tilde f^{0,+}(\partial_x w^0_{p^0}(t,0^-))=f^0_{\max}-f^0(p^0) \geq A(t)-f^0(p^0), 
 $$
thanks to \eqref{hypfluxlimiterA}.  This proves that \eqref{w0p0t} holds a.e. in $\{w^0_{p^0}(\cdot,0)>0\}$. Fix now $t\in\R$ a point of continuity of $A$, of derivability of $w^0_{p^0}(\cdot,0)$ and such that  $w^0_{p^0}(t,0)=0$ and \eqref{repf0+} holds. Then $\partial_t w^0_{p^0}(\cdot,0)=0$ since $w^0_{p^0}\geq 0$. As $\hat t_1=t$ is optimal in \eqref{def.psiTER} and $A$ is continuous at $t$, one necessarily has $A(t)-f^0(p^0)\geq 0$ by optimality, so that by \eqref{repf0+}
 $$
 \min\Bigl\{A(t)-f^0(p^0) \ , \ \tilde f^{0,+}(\partial_x w^0_{p^0}(t,0^-))\Bigr\}  
 = \min\Bigl\{A(t)-f^0(p^0) \ , \ 0\Bigr\} =0.
 $$
 This proves the first equality in \eqref{w0p0t} in $\{w^0_{p^0}(\cdot,0)=0\}$. The second one is just the last statement of Lemma \ref{lem.analysew00} since $w^0_{p^0}(t,0)= \psi_{p^0}(t)$. 

\end{proof}

\subsection{Periodic solutions to a HJ equation on the exit lines} 

We proceed with our construction of correctors, now building the correctors on the exit lines. Again we use a representation formula of the solution in terms of a Hamilton-Jacobi equation. 

Let  $p^0\in [a^0, b^0]$ satisfying \eqref{condp0} or $p^0= (f^{0,-})^{-1}\left(\int_0^1A(s)ds\right)$. Let  $w^0_{p^0}$ be defined in Lemma \ref{lem.w0}. We fix $j=1,2$ and assume that $f^j$ satisfies \eqref{hypflux}.   Recalling the definition of the flux $F_{p^0}$ in \eqref{defFp0} and \eqref{defFp0BIS}, we  introduce the flux entering  the exit-line $j$ (where $j=1,2$) as
$$
 F^j_{p^0}(t)= F_{p^0}(t){\bf 1}_{I^j}(t)= \left\{ \begin{array}{ll}
\min\{ A(t),f^{0,+}(\partial_x w^0_{p^0}(t,0^-)+p^0)\} & {\rm if}\; t\in I^j, \\
0 & \text{otherwise},
\end{array}\right. 
$$
where the second equality  comes from \eqref{w0p0t} in Lemma \ref{lem.f0+}. Let us also recall the definition of $\hat p^j_{p^0}$ introduced in \eqref{defpkp0BIS} in the case $p^0\in [a^0, b^0]$. 

\begin{Definition}[The notation $\hat p^j_{p^0}$]\label{defhatpi}  Given $p^0\in [a^0, b^0]$ satisfying \eqref{condp0} or $p^0= (f^{0,-})^{-1}\left(\int_0^1A(s)ds\right)$, let $\hat p^j_{p^0}\in [a^j, b^j]$ (for $j=1,2$) be the unique solution to 
$$
f^j(\hat p^j_{p^0})=f^{j,+}(\hat p^j_{p^0})= \int_0^1 F^j_{p^0}(t)dt. 
$$
\end{Definition} 

\begin{Remark} Note that $\hat p^j_{p^0}$ indeed exists and is unique since, by \eqref{hypfluxlimiterA} and the definition of $F^j_{p^0}$, $0\leq \int_0^1 F^j_{p^0}(t)dt\leq f^j_{\max}$ and $f^j$ is one-to-one from $[a^j,b^j]$ to $[0, f^j_{\max}]$.  
\end{Remark} 

Let us now set $\tilde f^j_{p^0}(p)= f^j(p+\hat p^j_{p^0})-f^j(\hat p^j_{p^0})$ for $p\in [a^j-\hat p^j_{p^0},b^j - \hat p^j_{p^0}]$, 
and \begin{equation}\label{eq::rr5BIS}
\tilde \psi^j_{p^0}(t)=- \int_0^t (F^j_{p^0}(s)-f^j(\hat p^j_{p^0}))ds.
\end{equation}
Note that $\tilde \psi^j_{p^0}$ is a $1-$periodic, Lipschitz continuous map, satisfying  
\be\label{zekrhfjdnj}
(\tilde \psi^j_{p^0})'\in - [ - f^j(\hat p^j_{p^0}), \max\tilde f^j_{p^0}]= -[ \tilde f^j_{p^0}(a^j-\hat p^j_{p^0}),\tilde f^j_{p^0}(b^j-\hat p^j_{p^0})]\; \text{a.e..}
\ee
Let us consider the time-periodic viscosity solution $w^j_{p^0}$ to the Hamilton-Jacobi equation  
\be\label{eq.HJbisBIS}
\left\{\begin{array}{lll}
(i) & \partial_x w^j\in [a^j-\hat p^j_{p^0}, b^j-\hat p^j_{p^0}] & \text{a.e. in}\; \R\times (0,+\infty),\\
(ii) & \partial_t w^j + \tilde f^j(\partial_xw^j)=0& {\rm for }\; t\in \R, \; x>0,\\
(iii) & w^j(t,0)=  \tilde \psi^j_{p^0}(t) & {\rm for }\;  \; t\in \R.
\end{array}\right.
\ee

%
\begin{Lemma}[Explicit time-periodic solution in the outgoing roads] \label{lem.defwBIS} Fix $j=1,2$. Assume that $f^j$ satisfies \eqref{hypflux} and that \eqref{hypIj} and \eqref{hypfluxlimiterA} hold. Let $p^0\in [a^0, b^0]$ satisfying \eqref{condp0} or $p^0= (f^{0,-})^{-1}\left(\int_0^1A(s)ds\right)$ and let ${\tilde \psi^j_{p^0}}$ be defined  in (\ref{eq::rr5BIS}). 
Then, there exists a unique time-periodic Lipschitz continuous viscosity solution $w^j_{p^0}$ to \eqref{eq.HJbisBIS}, of time period equal to $1$. It is given by
\be\label{def.wtxBIS}
w^j_{p^0}(t,x)= \sup_{t_1\leq t}  \{ \tilde \psi^j_{p^0}(t_1)-\xi^j_{p^0}( t-t_1,x)\}, 
\ee
where the map $\xi^j_{p^0}:[0,\infty)^2\to \R$ is defined by 
$$
\xi^j_{p^0}(s,y)= \max_{p\in [a^j-\hat p^j,b^j-\hat p^j]} -py+ s\tilde f^j_{p^0}(p)\qquad \forall s\geq 0,\;y\geq 0 . 
$$
 Finally, there exists a constant $C>0$ such that 
\be\label{condinftywi}
\|\partial_x w^j_{p^0}\|_{L^\infty(\R\times (M,\infty))} \leq \frac{C}{M} \qquad \forall M\geq C.
\ee
\end{Lemma}

\begin{proof} Following Lemma \ref{lem.periodicHJgeneric}, $w^j_{p^0}$ { defined in \eqref{def.wtxBIS}} is the unique solution to \eqref{eq.HJbisBIS} and is Lipschitz continuous because by construction $\tilde f^j_{p^0}:  [a^j-\hat p^j_{p^0},b^j - \hat p^j_{p^0}]\to \R$ is increasing and uniformly concave and $\tilde \psi_{p^0}$ satisfies \eqref{zekrhfjdnj}. As $\tilde f^j_{p^0}(0)=0$, Lemma \ref{lem.behaviorinfty} implies that $w^j_{p^0}$ is bounded and satisfies \eqref{condinftywi}. 
\end{proof}
\medskip

In order to  make the link with conservation laws, we need the following technical result which will allow us to glue the solutions on the different branches. Fix $p^0$ as in Lemma \ref{lem.defwBIS} and let  $w^0_{p^0}$, $\hat p^j_{p^0}\in [a^j, b^j]$  and $w^j_{p^0} $ be  {respectively} defined by  Lemma \ref{lem.w0} , Definition \ref{defhatpi} and Lemma \ref{lem.defwBIS}. The maps $w^j_{p^0}$ being a solution to the HJ equation \eqref{HJ0} (for $j=0$) and \eqref{eq.HJbisBIS} (for $j=1,2$), $\partial_x w^j_{p^0}$ is a solution to the corresponding conservation law (Lemma \ref{lem::r2}). Therefore $\partial_x w^j_{p^0}$ has a trace at $x=0$ in the sense of Panov (Theorem \ref{th::r7homo}).

\begin{Lemma}[Expression of the flux of the traces]\label{lem.viscjunc} On $I^j$ (for $j=1,2$), the trace $(\partial_x  w^0_{p^0}(\cdot,0^-), \partial_x w^j_{p^0}(\cdot,0^+))$ satisfies 
\be\label{tracecondIj}
\begin{array}{c}
\min\Bigl\{ A(t),  f^{0,+}(\partial_x  w^0_{p^0}(t,0^-)+p^0),  f^{j,-}(\partial_x w^j_{p^0}(t,0^+)+\hat p^j_{p^0})\Bigr\}\\
\qquad =  f^0(\partial_x  w^0_{p^0}(t,0^-)+p^0)=  f^j(\partial_x w^j_{p^0}(t,0^+)+\hat p^j_{p^0}) \; \text{a.e.,}
\end{array}
\ee
while on $\R\backslash I^j$, the trace $\partial_x w^j_{p^0}(\cdot,0^+)$ satisfies 
\be\label{tracecondI-j}
 f^j(\partial_x w^j_{p^0}(t,0^+)+\hat p^j_{p^0}) =0 \qquad \text{a.e..}
\ee
\end{Lemma}

\begin{proof} {\bf Step 1: proof of \eqref{tracecondIj}.} The main idea is to reduce the problem to a HJ equation on a junction and then to use the equivalence between HJ and conservation law on a simple junction with only two branches given by Lemma \ref{lem::r3}. Fix $j=1,2$ and let $(\tau_1,\tau_2)\subset I^j$ on which $A$ is constant. We set 
\be\label{defWW}
 \left\{\begin{array}{ll}
W^0(t,x)=w^0_{p^0}(t,x) +p^0x -f^0(p^0)(t-\tau_1)&\text{in}\; (\tau_1,\tau_2)\times (-\infty,0),\\
W^j(t,x)=w^j_{p^0}(t,x)+\hat p^j_{p^0} x-f^j(\hat p^j_{p^0})(t-\tau_1) -w^j_{p^0}(\tau_1,0) +w^0_{p^0}(\tau_1,0)& \text{in}\; (\tau_1,\tau_2)\times (0,+\infty),
\end{array}\right.
\ee
We claim that  $W=(W^0,W^j)$ is a viscosity solution to the problem on the 1:1 junction (in the sense of \cite{IM17}): 
\be\label{HJ0junc}
\left\{\begin{array}{lll}
(i) & \partial_t W^0 +  f^0(\partial_xW^0)=0& {\rm for }\; t\in (\tau_1,\tau_2), \; x<0,\\
(ii) & \partial_t W^j +  f^j(\partial_xW^j)=0& {\rm for }\; t\in (\tau_1,\tau_2), \; x>0,\\
(iii) & W(t,0):=W^0(t,0^-)= W^j(t,0^+) & {\rm for}\; t\in (\tau_1,\tau_2), \\
(iv) & \partial_t W(t,0)+ \min\{ A(t),  f^{0,+} (\partial_xW(t,0^-)),  f^{j,-} (\partial_xW(t,0^+)) \}=0 & {\rm for }\; t\in (\tau_1,\tau_2). \end{array}\right.
\ee
Indeed, by construction, $W(\tau_1,0^-)= W(\tau_1,0^+)$. By Lemma \ref{lem.f0+},  we have
$$
\partial_t W(t,0^-)= \partial_t w^0_{p^0}(t,0^-)-f^0(p^0)= -F_{p^0}(t)=-F^j_{p^0}(t)  \qquad \text{a.e. in} \; (\tau_1,\tau_2),
$$
while, by the boundary condition satisfied by $w^j_{p^0}$ and the definition of $\tilde \psi^j_{p^0}$ in \eqref{eq::rr5BIS},  
$$
\partial_t W(t,0^+)=\partial_t w^j_{p^0}(t,0^+)-f^j(\hat p^j_{p^0})= (\tilde \psi^j_{p^0})'(t)-f^j(\hat p^j_{p^0}) =-F^j_{p^0}(t) \qquad \text{a.e. in} \; (\tau_1,\tau_2).
$$
Thus equality (iii) in \eqref{HJ0junc} holds. Note also that, by the equation satisfied by $w^0_{p^0}$ and $w^j_{p^0}$, \eqref{HJ0junc}-(i) and \eqref{HJ0junc}-(ii) hold. Let us finally check that the junction condition \eqref{HJ0junc}-(iv) holds in the viscosity sense.  As, by the definition of  $F^j_{p^0}$, 
$$
\partial_t W(t,0)+ A(t) = -F^j_{p^0}(t) +A(t)\geq 0\qquad \text{a.e.}, 
$$
\cite[Theorem 2.11]{IM17} implies that $W$ is a supersolution. 
For the subsolution property, assume that $\varphi(t,x):=\alpha(t)+q^0 x{\bf 1}_{x<0}+q^j{\bf 1}_{x>0}$  is a test function touching $w^0$  from above at $(t_0,0)$, where $t_0\in (\tau_1,\tau_2)$ and with (condition (2.12) in  \cite[Theorem 2.7]{IM17})
\be\label{zkejnsrdkgTER}
A(t)=  f^0(q^0)= f^{0,-}(q^0)= f^j(q^j)= f^{j,+}(q^j).
\ee 
We have to prove that $\alpha'(t_0)+A(t)\leq 0$. As the map $(t,x){\mapsto} \alpha(t)+q^0 x{\bf 1}_{x<0}$ touches locally $W$ from above on $(\tau_1,\tau_2)\times (-\infty,0]$ at $(t_0,0)$, the map $(t,x) {\mapsto} \alpha(t)+(q^0-p^0) x{\bf 1}_{x<0} +f^0(p^0)(t-\tau_1)$ touches locally $w^0_{p^0}$ from above on $\R\times (-\infty,0]$ at $(t_0,0)$. By the equation satisfied by $w^0_{p^0}$, this implies that 
$$
\alpha'(t_0)+f^0(p^0)+\min\{ A(t_0)-f^0(p^0), \tilde f^{0,+}( q^0- p^0)\} \leq 0.
$$  
Recalling the definition of $\tilde f^0$ in \eqref{deftildef0}, the inequality above yields  
$$
\alpha'(t_0) +\min\Bigl\{A(t_0), f^{0,+}(q^0)\Bigr\}\leq 0, 
$$
where, because of \eqref{zkejnsrdkgTER} and \eqref{hypfluxlimiterA}, $f^{0,+}(q^0)= f^0_{\max}\geq A(t_0)$. Hence $\alpha'(t_0)+A(t_0)\leq 0$. This proves that $W$ is a viscosity solution to \eqref{HJ0junc}. 

We now rely on Lemma \ref{lem::r3}, which implies that the trace $\partial_x W(\cdot, t)$  satisfies 
$$
\partial_x W(t,0) \in G \qquad \text{a.e. in}\; (\tau_1,\tau_2), 
$$
where 
$$
G=\{(u^0,u^j)\in [a^0,c^0]\times [a^j,c^j], \; \min\left\{A(t), f^{0,+}(u^0),f^{j,-}(u^j)\right\}=f^0(u^0)=f^j(u^j) \}. 
$$
This is \eqref{tracecondIj}. \medskip

\noindent{\bf Step 2: proof of \eqref{tracecondI-j}.} Fix $j\in \{1,2\}$ and  let $(\tau_1,\tau_2)\subset \R\backslash I^j$ on which $A$ is constant and let $W=(W^0,W^j):\R\to \R$ be given by 
\be\label{defWW2}
\left\{\begin{array}{ll}
W^0(t,x) = w^j_{p^0}(t,0) +a^j x-f^j(\hat p^j_{p^0})(t-\tau_1)&\text{in}\; (\tau_1,\tau_2)\times (-\infty,0),\\
W^j(t,x)=w^j_{p^0}(t,x)+\hat p^j_{p^0} x-f^j(\hat p^j_{p^0})(t-\tau_1) & \text{in}\; (\tau_1,\tau_2)\times (0,+\infty),
\end{array}\right.
\ee
We claim that $W$ is a viscosity solution of the HJ equation on the 1:1 junction 
\be\label{HJ0junc2}
\left\{\begin{array}{lll}
(i) & \partial_t W^0 +  f^j(\partial_xW^0)=0& {\rm for }\; t\in (\tau_1,\tau_2), \; x<0,\\
(ii) & \partial_t W^j +  f^j(\partial_xW^j)=0& {\rm for }\; t\in (\tau_1,\tau_2), \; x>0,\\
(iii) & W(t,0):=W^0(t,0^-)= W^j(t,0^+) & {\rm for}\; t\in (\tau_1,\tau_2), \\
(iv) & \partial_t W(t,0)+ \min\{ 0 ,   f^{j,+} (\partial_xW^0(t,0^-)) ,  f^{j,-} (\partial_xW^j(t,0^+)) \}=0 & {\rm for }\; t\in (\tau_1,\tau_2). \end{array}\right.
\ee
Indeed,  by construction, $W$ is continuous and conditions (ii) and (iii) hold. On $(\tau_1,\tau_2)\times (-\infty,0)$, we have (in the a.e. sense and thus, by the smoothness {of $W^0$ which is affine}, in the viscosity sense)
$$
\partial_t W^0(t,x) +  f^j(\partial_xW^0(t,x))= \partial_t w^j_{p^0}(t,0)-f^j(\hat p^j_{p^0}) +f^j(a^j)= 0
$$
since $\partial_t w^j_{p^0}(t,0) =(\tilde \psi^j_{p^0})'(t) =  f^j(\hat p^j_{p^0})$ as $F^j_{p^0}=0$ on $(\tau_1,\tau_2)\subset \R\backslash I^j$. Thus $(i)$ holds. 
The same proof shows that $\partial_t W(t,0)= 0$, which implies condition $(iv)$. As $W$ is a viscosity solution of \eqref{HJ0junc2} we infer from Lemma \ref{lem::r3} that the trace at $x=0$ of ${\partial_x} W$ satisfies 
$$
\partial_xW(t,0) \in G,
$$
where 
\begin{align*}
G&=\{(u^0,u^j)\in [a^j,c^j]^2, \; \min\left\{0, f^{j,+}(u^0),f^{j,-}(u^1)\right\}=f^j(u^0)=f^j(u^j) \}\\
& = \{(u^0,u^j)\in [a^j,c^j]^2, \; 0=f^j(u^0)=f^j(u^j) \}.
\end{align*}
This implies \eqref{tracecondI-j}. 
\end{proof}

\subsection{Construction of the correctors}

We are now ready to build the correctors, i.e., the time-periodic solutions to \eqref{eq.meso} with a specific behavior at infinity. Throughout this part, assumptions   \eqref{hypflux}, \eqref{hypIj} and \eqref{hypfluxlimiterA} are in force.

\subsubsection{The correctors in the fluid case}

We build here a corrector when $(p^0,p^1,p^2)$ is as in case (i) of the definition \eqref{defG0} of $E_{\bar\Lambda}$. 

\begin{Proposition}\label{prop.correctorfluid}
Assume that $p^0\in [a^0,b^0]$ satisfies $f^0(p^0)\leq \int_0^1 A(t)dt$. Then there exists a bounded  solution $u_{p^0}=(u^j_{p^0})$ to \eqref{eq.meso} on $\R\times {\mathcal R}$, which is time-periodic of period $1$ and satisfies, for a constant $C>0$ depending on the data and on $p^0$,  
$$
u^0_{p^0}(t,x)=p^0 \;\text{for a.e.}\; t\in \R, \ x\leq -C, \qquad \|u^j_{p^0}(t, \cdot)-\hat p^j_{p^0}\|_{L^\infty(\R\times (M,\infty))} \leq \frac{C}{M}\;\text{for any}\; M\geq C.
$$
\end{Proposition}

\begin{proof} Let us set 
$$
 \left\{\begin{array}{ll}
u^0_{p^0}(t,x)= \partial_x w^0_{p^0}(t,x)+p^0 & \text{on} \; \R\times (-\infty,0), \\
u^j_{p^0}(t,x)=\partial_x w^j_{p^0}(t,x)+\hat p^j_{p^0} & \text{on} \; \R\times (0,+\infty), j=1,2,\\
\end{array}\right.
$$
where $w^0_{p^0}$, $\hat p^j_{p^0}\in [a^j, b^j]$  and $w^j_{p^0} $ are defined in  Lemma \ref{lem.w0}, Definition \ref{defhatpi} and Lemma \ref{lem.defwBIS} respectively. By construction, $u_{p^0}$ is bounded and time-periodic of period $1$ as $w^0_{p^0}$ and $w^j_{p^0}$ are Lipschitz continuous and $1-$periodic in time. As $w^0_{p^0}$ and $w^j_{p^0}$ solve \eqref{HJ0}-(i)-(ii) and \eqref{eq.HJbisBIS}-(i)-(ii) respectively, $u_{p^0}$ satisfies \eqref{eq.meso}-(i)-(ii) thanks to the local correspondance between viscosity solution and conservation laws in $1-$space dimension recalled in Lemma \ref{lem::r2}. The behavior at infinity of $u_{p^0}$ is a consequence of \eqref{condinftyw0} and \eqref{condinftywi}. As for the junction condition \eqref{eq.meso}-(iii), it is proved in Lemma \ref{lem.viscjunc}. 
\end{proof}

\subsubsection{The correctors in the fully congested case} 

In this part we assume that the second exit road is fully congested (case (ii) in \eqref{defG0}): 

\begin{Proposition}\label{prop.correctorcongested} Assume that $(p^0,p^1,p^2)\in Q$ satisfies 
$$
p^2=c^2, \; f^0(p^0)= f^{0,-}(p^0)= \int_0^1{\bf 1}_{I^1}(t)A(t)dt = f^1(p^1)={f^{1,+}(p^1).}
$$
Then there exists a bounded  solution $u=(u^j)$ to \eqref{eq.meso} on $\R\times {\mathcal R}$, which is time-periodic of period $1$ and satisfies, for a constant $C>0$ depending on the data and on $p^0$,  $u^2=c^2$ and 
$$
\|u^0-p^0\|_{L^\infty(\R\times (-\infty,M))}+\|u^1- p^1\|_{L^\infty(\R\times (M,\infty))} \leq \frac{C}{M}\;\text{for any}\; M\geq C.
$$
\end{Proposition}

\begin{proof} Let us define a new flux limiter by setting $\tilde A:= A{\bf 1}_{I^1}$. We note that $p^0= (f^{0,-})^{-1}\left(\int_0^1 \tilde A(s)ds\right)$. 
Let us consider $w^0_{p^0}$ the solution introduced in Lemma \ref{lem.w0}  and $w^1_{p_0}$  the solution given for $j=1$  in Lemma \ref{lem.defwBIS} for the the new flux limiter $\tilde A$. We set 
$$
 \left\{\begin{array}{ll}
u^0_{p^0}(t,x)= \partial_x w^0_{p^0}(t,x)+p^0 & \text{on} \; \R\times (-\infty,0), \\
u^1_{p^0}(t,x)=\partial_x w^1_{p^0}(t,x)+\hat p^j_{p^0} & \text{on} \; \R\times (0,+\infty).\\
\end{array}\right.
$$
As $w^0_{p^0}$ and $w^1_{p^0}$ solve \eqref{HJ0}-(i)-(ii) and \eqref{eq.HJbisBIS}-(i)-(ii) respectively (with flux limiter $\tilde A$), $(u^0_{p^0},u^1_{p^0},c^2)$ satisfies \eqref{eq.meso}-(i)-(ii) thanks to the local correspondance between viscosity solution and conservation laws in $1-$space dimension recalled in Lemma \ref{lem::r2}. The behavior at infinity of $(u^0_{p^0},u^1_{p^0})$ is a consequence of \eqref{condinftyw0} and \eqref{condinftywi}. As for the junction condition \eqref{eq.meso}-(iii), it is proved in Lemma \ref{lem.viscjunc}. 
\end{proof}

\section{Proof of the homogenization}\label{sec:homog}

The section is dedicated to the proof of the existence of a solution to the mesoscopic model and of  the homogenization for the 1:2 junctions ({Subsection \ref{proofs1:2}}) and  for the 2:1 junctions (Subsection \ref{proof2:1}). 

\subsection{Proof for a 1:2 junction}\label{proofs1:2}

In this part, we prove Lemma \ref{lem.existeqmeso}, Theorem  \ref{thm.corrector} and Theorem \ref{thm.main}.

\begin{proof}[Proof of Lemma \ref{lem.existeqmeso}]
We  show the existence of a solution to \eqref{eq.meso} with initial condition $\bar \rho$ by induction on the time intervals $[0,\tau^{k+1})$, $k\in \N$, where ($[\tau^k,\tau^{k+1})$) form a partition of $[0,+\infty)$ such that, for any $k\in \N$, $A$ is constant on the interval $(\tau^k,\tau^{k+1})$ and $(\tau^k,\tau^{k+1})\subset I^i$ for some $i=1,2$.\\

\noindent {\bf Step 1: existence on $(0,\tau^1)$}\\
To fix the ideas we assume here that $(0,\tau^1)\subset I^1$, as the case where $(0,\tau^1)\subset I^2$ can be treated in a symmetric way. Let  $A$ denote the (constant) restriction of the flux limiter $A(\cdot)$ to $(0,\tau^1)$. 

Let $\bar w$ be an antiderivative of the initial data $\bar \rho$, i.e. $\bar w:\R\to \R$ is Lipschitz continuous and such that $\partial_x \bar w= \bar \rho$.

 On the time interval $[0,\tau^1)$ we set 
$$
(\rho^0,\rho^1,\rho^2)= (\partial_x w^0, \partial_x w^1, \partial_x w^2)\qquad \text{on}\;   (0,\tau^1)
$$
where $(w^0, w^1)$ solves the HJ equation,  with a junction condition at $x=0$, 
$$
\begin{array}{llll}
\partial_t w^j+f^j(\partial_x w^j)=0&\qquad \text{on}\; &(0,\tau^1)\times \mathcal R^j, \;& j=0,1,\\
w(t,0):=w^0(t,0^-)=w^1(t,0^+)&\qquad \text{on }\; &(0,\tau^1)\times \{x=0\},&\\
\partial_t w+ \min\{A, f^{0,+}(\partial_x w^0), f^{1,-}(\partial_x w^1)\}=0 &\qquad \text{on }\; &(0,\tau^1)\times \{x=0\},&\\
w^j=\bar w^j& \qquad \text{on}\; &\left\{t=0\right\}\times \mathcal R^,\; &j=0,1,
\end{array}
$$
and $w^2$ is the solution to 
$$
\begin{array}{lll}
\partial_t w^2+f^2(\partial_x w^2)=0& \qquad \text{on}\; &(0,\tau^1)\times \mathcal R^2,\\
\partial_t w^2 +\min\left\{0,f^{2,-}(\partial_x w^2)\right\}=0& \qquad \text{on}\; &(0,\tau^1)\times \left\{x=0\right\},\\
w^2= \bar w^2 & \qquad \text{on}\; &\left\{t=0\right\}\times  \mathcal R^2,
\end{array}
$$
where the solutions are given by the theory developed in \cite{IM17}\footnote{In  \cite{IM17}, the Hamiltonian is coercive. To cover this case, we just have to extend each $f^j$ as a concave function on $\R$ such that $-f^j$ is coercive. Then using the comparison principle and suitable barriers (built on the initial data), it is quite standard that we can show that $\partial_t w^j\le 0$ for our initial data satisfying $f^j(\partial_x \bar w^j)\ge 0$. Then using the PDE itself, we can show that the solution satisfies $f^j(\partial_x w^j)\ge 0$ and then $\partial_x w^j\in [a^j,c^j]$ almost everywhere.}.

From Lemma \ref{lem::r3}, we know that $\tilde \rho:=(\rho^0,\rho^1)$ is an entropy solution to  
$$\begin{array}{llll}
\rho^j \in [a^j,c^j]&\qquad \mbox{a.e. on}\quad &(0,\tau^1)\times \mathcal R^j,&\quad j=0,1\\
\partial_t \rho^j+ \partial_x( f^j( \rho^j))=0&\qquad \text{in}\; &(0,\tau^1) \times \mathcal R^j,&\quad j=0,1,  \\ 
\tilde \rho(t,0)\in G^{0,1}&\qquad \text{a.e. on}\; &(0,\tau^1)\times \left\{0\right\},&
\end{array}$$
with initial condition ${(\bar \rho^0,\bar \rho^1)}$, where the (maximal) germ $G^{0,1}$ is given  by 
$$G^{0,1}:= \{ (p^0,p^1)\in [a^0,c^0]\times [a^1,c^1], \quad   \min \left\{A, f^{0,+}(p^0),f^{1,-}(p^1)\right\}=f^0(p^0)= f^1(p^1)\}.$$
Moreover, introducing $\bar w^{\emptyset}\equiv w^2(0,0)$, $f^{\emptyset}\equiv 0\equiv f^{\emptyset,+}$, $\mathcal R^{\emptyset}:=(-\infty,0)$, we see that $(w^{\emptyset}\equiv  w^2(0,0), w^2)$ is solution to
$$
\begin{array}{llll}
\partial_t w^j+f^j(\partial_x w^j)=0 &\qquad \text{on}\; &(0,\tau^1) \times \mathcal R^{j},& j=\emptyset,2\\
w(t,0)=w^{\emptyset}(t,0)=w^2(t,0)&\qquad \text{at}\; & (0,\tau^1)\times \left\{x=0\right\},&\\
\partial_t w +\min \{f^{\emptyset,+}(\partial_x w^\emptyset),f^{2,-}(\partial_x w^2)\}=0 &\qquad \text{at}\; &(0,\tau^1)\times \left\{x=0\right\}.&\\
w^j=\bar w^j& \qquad \text{on}\; &\left\{t=0\right\}\times \mathcal R^j,& j=\emptyset,2.
\end{array}
$$
Setting $\bar \rho^{\emptyset}\equiv 0$, $\rho^{\emptyset}=\partial_x w^{\emptyset}=0$ and $a^\emptyset=0=c^\emptyset$, we see from Lemma \ref{lem::r3} that $\hat \rho=(\rho^{\emptyset},\rho^2)$ is an entropy solution of 
$$\begin{array}{llll}
\rho^j\in [a^j,c^j]&\quad \mbox{a.e. on}\quad &(0,\tau^1)\times \mathcal R^j,&\quad j=\emptyset,2\\
\partial_t \rho^j+ \partial_x( f^j(\tilde \rho^j))=0&\quad \mbox{on}\quad &(0,\tau^1) \times \mathcal R^j,&\quad j=\emptyset,2\\
\hat \rho(t,0)\in G^{\emptyset,2}& \quad \mbox{on}\; &(0,\tau^1)\times \left\{x=0\right\},&
\end{array}$$
with initial condition $\bar \rho^j$, where the germ $G^{\emptyset,2}$ is given by 
 $$
 G^{\emptyset,2}=\{(p^\emptyset ,p^2) \in [a^\emptyset,c^\emptyset] \times [a^2,c^2], \quad \min \left\{f^{\emptyset,+}(p^\emptyset),f^{2,-}(p^2)\right\}=f^\emptyset(p^\emptyset)= f^2(p^2)\}.
 $$
 This shows that $\rho^2$ is an entropy solution of
 $$\begin{array}{lll}
\rho^2\in [a^2,c^2]&\quad \mbox{a.e. on}\quad &(0,\tau^1)\times \mathcal R^2,\\
\partial_t \rho^2+ \partial_x( f^2(\tilde \rho^2))=0&\quad \mbox{on}\quad &(0,\tau^1) \times \mathcal R^2,\\
\rho^2(t,0)\in G^2& \quad \mbox{on}\; &(0,\tau^1)\times \left\{x=0\right\},
\end{array}$$
 with initial condition $\bar \rho^2$ and with
 $$G^2:=\left\{p^2\in \R\quad\mbox{such that $(0,p^2)\in G^{\emptyset,2}$}\right\}=\left\{p^2 \in [a^2,c^2], \quad f^2(p^2)=0\right\}=\left\{a^2,c^2\right\}$$
 Therefore  $\rho=(\rho^0,\rho^1,\rho^2)$ solves \eqref{eq.meso} on $(0,\tau^1)$ with initial condition $\bar \rho$.\\
 
\noindent {\bf Step 2: existence on $[0,\tau^k)$ given the solution on $[0,\tau^{k-1})$, $k\geq 2$}\\
Assume that we have built $\rho$ on $[0,\tau^{k-1})$. Recall that $\rho$ has a continuous in time representative with values in $L^1_{loc}$ {(see \cite[Theorem 6.2.2]{Da05})}. Let us set $\underline \rho := \rho(\tau^{k-1},\cdot)$. We can then build as in the previous step a solution $\tilde \rho=(\tilde \rho^0,\tilde \rho^1,\tilde \rho^2)$ of \eqref{eq.meso} on $(\tau^{k-1},\tau^k)$ with initial condition $\underline \rho$ at time $\tau^{k-1}$. It remains to check that the concatenation 
 $$
 \hat \rho(t,\cdot) =\left\{\begin{array}{ll}
 \rho(t,\cdot) & {\rm in}\; [0,\tau^1)\\
 \tilde \rho(t,\cdot) & {\rm in}\; [\tau^1, \tau^2)
 \end{array}\right.
 $$
 is an entropy solution to \eqref{eq.meso} on $[0,\tau^2)$ with initial condition $\bar \rho$. Note that the junction condition \eqref{eq.meso}-(iii) at $x=0$ is  satisfied because this is the case for $\rho$ on $(0,\tau^{k-1})$ and for $\tilde \rho$ on $(\tau^{k-1},\tau^k)$. It remains to check that $\hat \rho^j$ is an entropy solution on $[0,\tau^k)\times \mathcal R^j$ for any $j=0,1,2$. The argument is standard and we only sketch it. To fix the ideas, we do the proof for $j=1$, the argument for $j=0$ and $j=2$ being symmetric. Fix a $C^1_c([0,\tau^k)\times (0,+\infty))$ function $\varphi\ge 0$. Let $\theta_n:[0,\tau^{k-1})\to [0,1]$ be smooth, nonincreasing map, with a compact support and such that $\theta_n\to 1$ and $\theta_n'\to 0$ uniformly in $[0, \tau^{k-1}-\delta]$ for any $\delta>0$. As $\rho^1$ is an entropy solution on $[0,\tau^{k-1})\times \mathcal R^1$,   we have, for all $c\in \R$, 
$$
\int_{(0,\tau^{k-1})} \ \int_{(0,+\infty)} \ |\rho^1-c| (\varphi_t \theta_n+\varphi \theta_n') + \left\{\mbox{sign}(\rho^1-c)\right\}\cdot (f(\rho^1)-f(c)) \varphi_x \theta_n+ \int_{\left\{0\right\}\times (0,+\infty)} |\bar \rho^1-c|\varphi \theta_n(0) \ge 0.
$$
By the continuity of $t{\mapsto} \rho^1(t,\cdot)$ in $L^1_{loc}((0,\infty))$, we find, when letting $n\to\infty$, 
$$
- \int_{\{\tau^{k-1}\}\times (0,+\infty)} \ |\underline \rho^1-c| \varphi  +\int_{(0,\tau^{k-1})} \ \int_{(0,+\infty)} \ |\rho^1-c| \varphi_t +  \left\{\mbox{sign}(\rho^1-c)\right\}\cdot (f(\rho^1)-f(c)) \varphi_x + \int_{\left\{0\right\}\times (0,+\infty)} |\bar \rho^1-c|\varphi \ge 0.
$$
As $\tilde \rho^1$ is an entropy solution on $[\tau^{k-1},\tau^k)\times \mathcal R^1$ with initial condition $\underline \rho^1$, we also have 
$$
\int_{(\tau^{k-1},\tau^k)} \ \int_{(0,+\infty)} \ |\tilde \rho^1-c| \varphi_t +  \left\{\mbox{sign}(\tilde \rho^1-c)\right\}\cdot (f(\tilde \rho^1)-f(c)) \varphi_x + \int_{\left\{\tau^{k-1}\right\}\times (0,\infty)} |\underline \rho^1 -c|\varphi \ge 0. 
$$
Putting together the two previous inequalities   proves that $\rho^1$ is an entropy solution on $[0,\tau^k)\times \mathcal R^1$ with initial condition $\bar \rho^1$. \\
\noindent {\bf Step 4: existence on $[0,+\infty)$}\\
By  induction this proves the existence of a solution of the whole time interval $[0,\infty)$. \\
\noindent {\bf Step 5: Kato's inequality (\ref{Kato}) and uniqueness}\\
We claim that  $(\rho^0,\rho^1,\rho^2)$  satisfies Kato's inequality \eqref{Kato}. Indeed, as the sets $\mathcal G_{{\Lambda}^1}$ and $\mathcal G_{\Lambda^2}$ introduced in \eqref{germGL1bis} and \eqref{germGL2bis} are maximal germs (see Lemma \ref{lem::c15}), we just need to apply Kato's inequality given in \cite{AKR11} on each time interval $(\tau^k, \tau^{k+1})$ for $k\in \N$ and then proceed as above to glue the solution together. The uniqueness of the solution $\rho$ is then an obvious consequence of Kato's inequality. 
\end{proof}

\begin{proof}[Proof of Theorem  \ref{thm.corrector}] Let $p\in E_{\bar \Lambda}$. The existence of a corrector when $p$ satisfies (i) in the definition \eqref{defG0} of  $E_{\bar \Lambda}$ is given by  Proposition  \ref{prop.correctorfluid}. The case $(ii)$ is the aim of Proposition \ref{prop.correctorcongested}. The cases $(iii)$ is symmetric to the case (ii), exchanging the indices $1$ and $2$.
The case $(iv)$ is obvious because then one can choose $u^j_p=c^j$ for $j=0,1,2$.  
 \end{proof}
 
 \begin{proof}[Proof of Theorem \ref{thm.main}] 
 Recall that the construction of $\mathcal G_{\bar {\Lambda}}$ and the proof that it is a maximal germ are given in Subsection \ref{subsec.germmacro}. 
 
 We now prove the homogenization. It is known that the sequence $(\rho^\ep)$ is relatively compact in $L^1_{loc}((0,+\infty)\times \mathcal R)$ (Proposition \ref{prop.intregu} in the Appendix). 
  
 Let $\rho=(\rho^i)_{i=0, 1,2}$ be a limit (in $L^1_{loc}((0,+\infty)\times \mathcal R)$ and up to a subsequence) of $(\rho^\ep)$. We have to check that $\rho$ is the unique solution to \eqref{eq.homo}. By stability, $\rho^i$ is  an entropy solution on $[0,+\infty)\times \mathcal R^i$ and satisfies $\rho^i\in [a^i,c^i]$ a.e. on $(0,+\infty)\times \mathcal R^i$ for $i=0,1,2$. 

Let $p=(p^i)\in E_{\bar \Lambda}$. By Theorem \ref{thm.corrector}  there exists a time-periodic solution $u_p$ of \eqref{eq.meso} and $C>0$ such that for $M\ge C$, we have
\be\label{keypptcorrbis}
\|u^0_p- p^0\|_{L^\infty(\R\times (-\infty,-M))} +\|u^i_p-p^i\|_{L^\infty(\R\times (M,\infty))}\leq CM^{-1},\qquad i=1,2.
\ee
We set $u^\ep_p(t,x)= u_p(t/\ep, x/\ep)$. Note that the scaled function $u^\ep_p= (u^{\ep,k}_p)$ is a solution to \eqref{eq.mesoep} (without the initial condition). Thus, by  Kato's inequality (\ref{Kato}), we have
\begin{align*}
\sum_{i=0}^2\bigg\{ \int_0^\infty \int_{\mathcal R^i} |\rho^{\ep,i}- u^{\ep,i}_p|\phi^i_t+ \left\{\mbox{sign}(\rho^{\ep,i}- u^{\ep,i}_p)\right\}\cdot (f^i(\rho^{\ep,i})- f^i(u^{\ep,i}_p) )\partial_x\phi^i \\
+ \int_{\mathcal R^i} |\bar \rho_0^{i}(x)- u^{\ep,i}_p(0,x)| \phi^i(0,x)\bigg\} \geq 0
\end{align*}
for any continuous nonnegative test function $\phi:[0,\infty)\times \mathcal R\to [0,\infty)$ with a compact support and such that $\phi^j:=\phi_{|[0,+\infty)\times (\mathcal R^j\cup \left\{0\right\})}$ is $C^1$ for any $j=0,1,2$. Letting $\ep\to0$ and recalling \eqref{keypptcorrbis}, which implies that  $u^\ep_p$ converges in $L^1_{loc}$ to $p$ as $\ep\to 0$, this gives for any test function $\phi$ as above: 
$$
\sum_{i=0}^2\left\{ \int_0^\infty \int_{\mathcal R^i} |\rho^i- p^i|\phi^i_t+ \left\{\mbox{sign}(\rho^i- p^i)\right\}\cdot (f^i(\rho^i)- f^i(p^i)) \partial_x\phi^i +\int_{\mathcal R^i} |\bar \rho_0^i(0,x)- p^i| \phi^i(x) \right\}\geq 0. 
$$
Following the argument in \cite[Proposition 2.12]{MFR22}, this implies that, for a.e. $t\geq 0$, 
$$
q^0(p^0,\rho^0(t,0^-))\geq q^1(p^1,\rho^1(t,0^+))+q^2(p^2,\rho^2(t,0^+)).
$$
This inequality holds for any $p\in E_{\bar \Lambda}$ and for a.e. $t\ge 0$, and we have $(\rho^0(t,0^-),\rho^1(t,0^+),\rho^2(t,0^+))\in Q$ for a.e. $t\ge 0$. Therefore Lemma \ref{lem.reduction} implies that $\rho(t,0)=(\rho^i(t,0))\in \mathcal G_{\bar{\Lambda}}$. 
It follows that $\rho$ solves \eqref{eq.homo}, which has a unique solution $\rho$. Therefore the whole sequence $(\rho^\varepsilon)$ converges to $\rho$. {Moreover the $L^\infty$ bound on $\rho^\ep$ implies its convergence in $L^1_{loc}([0,+\infty)\times \mathcal R)$.}
\end{proof}

\subsection{Proof for a 2:1 junction}\label{proof2:1}

The main idea of the proof is to derive Theorem \ref{thm.main2:1} from Theorem \ref{thm.main} by a simple change of variable{s}, transforming 2:1 junctions into 1:2 junctions.

\subsubsection{A general framework for junctions with three roads}

We first introduce a general class of germs, defined for  fluxes $f^j$ for $j=0,1,2$ satisfying (\ref{hypflux}). 
The entropy flux associated to $f^j$ is defined for $\bar c,c \in [a^j,c^j]$ as
$$q^{f^j}(\bar c,c):=(f^j(\bar c)-f^j(c))\mbox{sign}(\bar c-c)$$
and let
$$\mbox{sign}(\mathcal R^j)=\left\{\begin{array}{ll}
+ &\quad \mbox{if}\quad \mathcal R^j=(0,+\infty)\\
- &\quad \mbox{if}\quad \mathcal R^j=(-\infty,0)\\
\end{array}\right.$$
 with a general set of three roads
$$\mathcal R=(\mathcal R^0,\mathcal R^1,\mathcal R^2).$$
Given $f=(f^0,f^1,f^2)$, the dissipation for $\bar P=(\bar p^0,\bar p^1,\bar p^2)$ and $P=(p^0,p^1,p^2)$ is defined by
$$D_{f,\mathcal R}(\bar P,P)= - \sum_{j=0,1,2} \mbox{sign}(\mathcal R^j)\cdot  q^{f^j}(\bar p^j,p^j).$$
We now build associated germs. Let us define the roots $u^{f^j}_{\pm}$ of $f^{j,\pm}(\cdot)=\lambda$ as 
$$\left\{\begin{array}{l}
\left[a^j,b^j\right] \ni u^{f^j}_+(\lambda):=r \quad \mbox{such that}\quad f^{j,+}(r)=\lambda\in \left[0,f^j_{\max}\right]\\
\left[b^j,c^j\right]\ni u^{f^j}_-(\lambda):=r  \quad \mbox{such that}\quad f^{j,-}(r)=\lambda\in \left[0,f^j_{\max}\right].\\
\end{array}\right.$$
For $\Lambda=(\bar \lambda^0,\bar \lambda^1,\bar \lambda^2,\hat \lambda^1,\hat \lambda^2)$ satisfying (\ref{eq::c12}), and for $\sigma\in \left\{\pm\right\}$, we consider  the curve
\begin{equation}\label{eq::c5bis}
\Gamma_{f,\Lambda}^{\sigma}:=\left\{P=(u^{f^0}_\sigma(\lambda),u^{f^1}_\sigma(\lambda^1),u^{f^2}_\sigma(\lambda^2))\quad \mbox{with}\quad \lambda^k:=\hat \lambda^k(\lambda)
\quad \mbox{for}\quad k=1,2\quad \mbox{and}\quad \lambda\in [0,\bar \lambda^0] \right\}
\end{equation}
and the points
\begin{equation}\label{eq::c6bis}
\left\{\begin{array}{lllll}
P_{0}^{f,\Lambda,\sigma}:=(&u^{f^0}_\sigma(0),&u^{f^1}_\sigma(0),&u^{f^2}_\sigma(0)&)\quad  \in \Gamma_{f,\Lambda}^\sigma\\
P_{3}^{f,\Lambda,\sigma}:=(&u^{f^0}_{-\sigma}(0),&u^{f^1}_{-\sigma}(0),&u^{f^2}_{-\sigma}(0)&)\\
\\
P_{1}^{f,\Lambda,\sigma}:=(&u^{f^0}_{-\sigma}(\bar \lambda^1),&u^{f^1}_\sigma(\bar \lambda^1),&u^{f^2}_{-\sigma}(0)&)\\
P_{2}^{f,\Lambda,\sigma}:=(&u^{f^0}_{-\sigma}(\bar \lambda^2),&u^{f^1}_{-\sigma}(0),&u^{f^2}_\sigma(\bar \lambda^2)&)\\
\end{array}\right.
\end{equation}
We also define
\begin{equation}\label{eq::c25}
E^\sigma_{f,\Lambda}:=\Gamma_{f,\Lambda}^{\sigma} \cup \left\{P^{f,\Lambda,\sigma}_1,P^{f,\Lambda,\sigma}_2,P^{f,\Lambda,\sigma}_3\right\}
\end{equation}

The case $\sigma=+$ corresponds to the divergent 1:2 junction, while the case $\sigma=-$ corresponds to the convergent 2:1 junction.

We consider the following general set (using notation $Q^{RH}$ defined in (\ref{eq::e**1}) and (\ref{eq::e**2}))
\begin{equation}\label{eq::c1bis}
{\mathcal G}_{f,\Lambda}^\pm:=\left\{P=(p^0,p^1,p^2)\in Q^{RH},\quad \left|\begin{array}{ll}
0\le f^j(p^j)\le\bar \lambda^j ,&\quad j=0,1,2\\
\\
f^{k,\pm}(p^k)\ge \hat \lambda^k(f^{0,\pm}(p^0)),&\quad k=1,2\\
\end{array}\right.\right\}.
\end{equation}

\subsubsection{Germs for  2:1 junctions,  by reversion} 

Consider the convergent 2:1 junction (with an abuse of notation)
\begin{equation}\label{eq::e60}
\check {\mathcal R}:=(\check {\mathcal R}^0,\check {\mathcal R}^1,\check {\mathcal R}^2)\quad \mbox{with}\quad \left\{\begin{array}{ll}
\check {\mathcal R}^j=(-\infty,0)&\quad \mbox{for}\quad j=1,2\\
\check {\mathcal R}^0=(0,+\infty)&
\end{array}\right.
\end{equation}
and associated fluxes $\check f^j$ for $j=0,1,2$ satisfying (\ref{hypflux}), (\ref{eq::e*1}), (\ref{eq::e*2}) and (\ref{eq::e*3}), with $\check a^j,\check b^j,\check c^j$ instead of $a^j,b^j,c^j$. Similarly, we consider the divergent 1:2 junction  denoted by ${\mathcal R}$ and defined as (also with an abuse of notation)
$${\mathcal R}:=({\mathcal R}^0,{\mathcal R}^1,{\mathcal R}^2)\quad \mbox{with}\quad \left\{\begin{array}{ll}
{\mathcal R}^0=(-\infty,0)&\\
{\mathcal R}^j=(0,+\infty)&\quad \mbox{for}\quad j=1,2\\
\end{array}\right.$$
We now explain how to transform fluxes $(\check f^j)$ defined on the convergent junction  $\check{\mathcal R}$ into fluxes $(f^j)$ defined on the divergent junction $\mathcal R$: 
we set
\begin{equation}\label{eq::c19}
f^j(v):=\check f^j(-v)\quad \mbox{with}\quad (a^j,b^j,c^j):=(-\check c^j,-\check b^j,-\check a^j). 
\ee
As before we set $Q=[a^0,c^0]\times [a^1,c^1]\times[a^2,c^2]$ and $\check Q=[\check a^0,\check c^0]\times [\check a^1,\check c^1]\times[\check a^2,\check c^2]$. 
\begin{Lemma}\label{lem::c23}{\bf (Effect of reversion on the dissipation)} For $P, \bar P\in Q$, 
$$
D_{\check f,\check{\mathcal R}}(-\bar P,-P)=D_{f,\mathcal R}(\bar P,P).
$$
\end{Lemma}

\begin{proof} Using
$$q^{\check f^j}(-\bar p^j,-p^j)=-q^{f^j}(\bar p^j,p^j)$$
we deduce that
$$D_{\check f,\check{\mathcal R}}(-\bar P,-P)=D_{f,\mathcal R}(\bar P,P).$$
\end{proof}

Let us now explain how to build germs for the fluxes  $\check f^j$ on the junction $\check{\mathcal R}$. 

\begin{Lemma}\label{cor::c24}{\bf (Germ for a convergent 2:1 junction)}\\
Let $\check {\mathcal R}$ be defined in (\ref{eq::e60}), and fluxes $\check f^j$ for $j=0,1,2$ satisfying (\ref{hypflux}), (\ref{eq::e*1}), (\ref{eq::e*2}) and (\ref{eq::e*3}).
Under assumption (\ref{eq::c12}) on $\Lambda$,  let us consider the set ${\mathcal G}_{\check f,\Lambda}^-$ defined in (\ref{eq::c1bis}).
Then this set ${\mathcal G}_{\check f,\Lambda}^- \subset \check Q$ is a maximal germ (for dissipation $D_{\check f,\check{\mathcal R}}$) determined by its subset $E^-_{\check f,\Lambda}$ defined in (\ref{eq::c25}). Recall here that for $P=(p^0,p^1,p^2)$ and $\bar P=(\bar p^0,\bar p^1,\bar p^2)$
$$D_{\check f,\check{\mathcal R}}(\bar P,P)=q^{\check f^1}(\bar p^1,p^1)+q^{\check f^2}(\bar p^2,p^2)-q^{\check f^0}(\bar p^0,p^0)=\mbox{IN}-\mbox{OUT}$$
\end{Lemma}

\begin{proof}
Lemma \ref{cor::c24} follows from Theorem \ref{th:1} for $\mathcal G_\Lambda=\mathcal G_{f,\Lambda}^+$. Applying reversion transform (\ref{eq::c19}) 
for $P=(p^0,p^1,p^2)\in \mathcal G^+_{f,\Lambda}$, which consists here in the transform
$$(P,\check f)\mapsto (-P, f)$$
and using the fact that
\be\label{eq.germU-U}
-P\in \mathcal G_{\check f,\Lambda}^- \quad \Longleftrightarrow\quad P\in \mathcal G_{f,\Lambda}^+
\ee
we see from Lemma \ref{lem::c23} that $\mathcal G_{\check f,\Lambda}^-$ is a germ.\\
Moreover, because for $\sigma\in \left\{\pm\right\}$ we have
$$-u^{f^j}_\sigma(\lambda)=u^{\check f^j}_{-\sigma}(\lambda),$$
 we see that 
$$-\Gamma_{\check f,\Lambda}^{-}=\Gamma_{f,\Lambda}^{+}\quad -P^{\check f,\Lambda,-}_{\ell}=P^{f,\Lambda,+}_{\ell}\quad \mbox{for}\quad \ell=1,2,3.$$
Now recall that $\mathcal G^+_{f,\Lambda}$ is determined by $E^+_{f,\Lambda}$.
Hence for $\displaystyle -P\in \check Q:= \prod_{j=0,1,2} [\check a^j,\check c^j]$, we have
$$\left(D_{\check f,\check{\mathcal R}}(-\bar P,-P)\ge 0 \quad \mbox{for all}\quad 
-\bar P\in E^-_{\check f,\Lambda}\right)\quad \Longrightarrow\quad -P\in \mathcal G_{\check f,\Lambda}^-\ ,$$
which shows that $\mathcal G_{\check f,\Lambda}^-$ is determined by the set $E^-_{\check f,\Lambda}$. 
This gives the desired result for dissipation $D_{\check f,\check {\mathcal R}}$ and  completes the proof of the lemma.
\end{proof}

By this simple change of variables and \eqref{eq.germU-U}, we have immediately
\begin{Corollary}\label{eq::c18}{\bf (Reversion of the germ)} Given $\check \rho, \check \rho^\ep\in L^\infty ((0,\infty)\times \check{\mathcal R})$, let 
\be\label{eq.defrhorhoep}
 \rho^{j}(t,x):=-\check\rho^{j}(t,-x) , \qquad  \rho^{j,\ep}(t,x):=-\check\rho^{j,\ep}(t,-x),\qquad \bar \rho_0^{j}(x):=-\check{\bar \rho}^{j}(-x), \quad \mbox{for}\quad x\in {\mathcal R}^j.
\ee
Given $(\check f^j)$, let $(f^j)$ be given by reversion transform (\ref{eq::c19}).
Then $\check \rho^\ep$ solves \eqref{eq::c20bis} (with initial data $\check{\bar \rho}$) with  germ $\check{\mathcal G}(\cdot)$ given by \eqref{eq::c21bis}, if and only if $\rho^\ep$ solves \eqref{eq.mesoep} (with initial data $\bar \rho$) with germ $\mathcal G(\cdot)$ given by \eqref{defGt}.  In the same way, if $\Lambda$ satisfies (\ref{eq::c12}), then $\check \rho$ solves \eqref{eq.macro2:1} for the germ ${\mathcal G}_{\check f,\Lambda}^-$ given by \eqref{eq::c1bis}, if and only if $\rho$ solves \eqref{eq.homo} for the germ  $\mathcal G_\Lambda=\mathcal G^+_{f,\Lambda}$ given by \eqref{eq::c1bis}. 
\end{Corollary}

\subsubsection{Proof of Theorem \ref{thm.main2:1}}

\begin{proof}[Proof of Theorem \ref{thm.main2:1}] The existence and the uniqueness of a solution to  \eqref{eq::c20bis} is a consequence of Lemma~\ref{lem.existeqmeso} and Corollary \ref{eq::c18}. Given $\check \rho^\ep$ a solution to \eqref{eq::c20bis}, let $\rho^\ep$  and $\bar \rho_0$ be defined by \eqref{eq.defrhorhoep}. We know from Corollary~\ref{eq::c18} that $\rho^\ep$ solves \eqref{eq.mesoep}, with $(f^j)$ defined by \eqref{eq::c19}, $\mathcal G(\cdot)$ given by \eqref{defGt}. Then Theorem \ref{thm.main} says that the $(\rho^\ep)$ converges in $L^1_{loc}$ as $\ep\to 0^+$ to the solution $\rho$ of \eqref{eq.homo}  for the germ $\mathcal G_{\bar{\Lambda}}={\mathcal G}_{ f,\bar \Lambda}^+$ defined in  \eqref{eq::c1BIS}, where $\bar{\Lambda}$ is given in Subsection  \ref{subsec.germmacro}. Let $\check \rho$ be defined from $\rho$ by the transform \eqref{eq.defrhorhoep}. Then, by Corollary~\ref{eq::c18},  $\check \rho$ is a solution of \eqref{eq.macro2:1} for the germ ${\mathcal G}_{\check f,\bar \Lambda}^-$. This shows that the $(\check \rho^\ep)$ converges in $L^1_{loc}$ as $\ep\to 0^+$ to $\check \rho$, which is the unique solution to \eqref{eq.macro2:1} for the germ ${\mathcal G}_{\check f,\bar \Lambda}^-$.
\end{proof}

\appendix

\section{Appendix} 

In this appendix, we collect several results needed throughout the paper. 

\subsection{Panov's theorem on strong traces}

Let $T>0$ and let us consider the following equation
\begin{equation}\label{eq::h1}
\partial_t u+ \partial_x(f(u))=0\quad \mbox{on}\quad (0,T)_t\times (0,+\infty)_x 
\end{equation}

We recall the following result.
\begin{Theorem}\label{th::r7homo}{\bf (Existence of strong traces;  \cite[Theorem 1.1]{Ps07})}\\
Assume that $f:\R\to \R$ is continuous and that $u\in L^\infty((0,T)_t\times (0,+\infty)_x)$ is a standard Krushkov entropy solution of (\ref{eq::h1}) on $(0,T)_t\times (0,+\infty)_x$. Assume moreover that $f$ satisfies the following nondegeneracy condition:
\begin{equation}\label{eq::h2}
\mbox{the map}\quad v\mapsto f(v) \quad \mbox{is not constant on  intervals of positive length.} 
\end{equation}
Then there exists $w \in L^\infty(0,T)$ and a measurable set $\mathcal N\subset (0,+\infty)_x$ of measure zero, such that
$$ \lim_{\varepsilon\to 0}\sup_{x\in (0,\varepsilon)\backslash \mathcal N} \|u(\cdot,x )-w\|_{L^1(0,T)}=0$$
and we write
$$\mbox{ess}\lim_{(0,+\infty)\ni x\to 0^+} u(\cdot,x )=w\quad \mbox{in}\quad L^1(0,T).$$
We call $w$ the  strong trace of $u$ on the interface $(0,T)\times \left\{0\right\}$ and we denote it by $u(\cdot, 0^+)$.
\end{Theorem}

\subsection{Local regularity of scalar conservation in one space dimension}

We assume for $a<c$, $\delta>0$, 
\begin{equation}\label{hypflux-ter}
\begin{array}{c}
\text{$f:[a, c]\to \R$ is  $C^2$ with $f''\le -\delta<0$ on $[a,c]$.}
\end{array}
\end{equation}
Given $(t,x)\in \R^2$ and $R>0$, let 
$$
Q_R(t,x)=[t-R, t+R]\times [x-2\|f'\|_\infty R, x+2\|f'\|_\infty R].
$$
We are interested in BV estimates of solutions to the the scalar conservation law $\partial_t u+\partial_x(f(u))=0$ in $Q_R(t,x)$. 

\begin{Proposition}{\bf (Local BV bound for a conservation law with a convex flux)}\label{prop.intregu}\\
Under assumption (\ref{hypflux-ter}), there exists  a constant $C>1$, depending on $c-a$, on $\|f'\|_\infty$ and on $\delta>0$ (the concavity constant of $f$), such that, for any $R\in (0,1]$ and any $(t,x)\in \R^2$, if $u:Q_R(t,x)\to [a,c]$ is an $L^\infty$ entropy solution to the scalar conservation law $\partial_t u+\partial_x(f(u))=0$ in $Q_R(t,x)$, then the total variation $V(u; Q_{R/3}(t,x))$ of $u$ in $Q_{R/3}(t,x)$ is bounded by 
$$
{V(u; Q_{R/3}(t,x)) \leq CR.} 
$$
\end{Proposition}

\begin{proof} We only sketch the proof, as it is standard (we just did not find a reference giving the formulation above needed in the paper). Without loss of generality we can assume that $(t,x)=(0,0)$ and $a=0$, so that $\|u\|_\infty\leq c$. We abbreviate $Q_R(0,0)$ into $Q_R$. By finite speed of propagation, the restriction of $u$ to $Q_{R/3}$ depends only on the value of $u(-R, \cdot)$ in $[-R', R']$, where $R':=2R\|f'\|_\infty$. Let us denote by $\tilde u$ the solution of $\partial_t \tilde u+\partial_x(f(\tilde u))=0$ in $(-\infty,\infty)\times \R$ starting from $\tilde u_0$ at time $-R$, where $\tilde u_0=u(-R,\cdot)$ on $[-R', R']$ and $\tilde u_0=0$ otherwise. Then $\tilde u=u$ in $Q_{R/3}$ and $\tilde u$ satisfies the Lax-Oleinik bound:
\be\label{laxOleinick}
\partial_x \tilde u(s,\cdot) \geq - \frac{1}{\delta (s+R)}\geq -\frac{3}{2\delta R} \qquad \text{for}\; s\in [-R/3,R/3],
\ee
in the sense of distributions. Thus, for any smooth test function $\phi$ with a compact support in $Q_{R/3}$, we have, at least formally, ($C$ denoting a constant depending on $c-a$, $\|f'\|_\infty$ and $\delta$ and possibly changing from line to line):  
\begin{align*}
\iint_{Q_{R/3}} (\partial_x \phi)  u & = \iint_{Q_{R/3}}(\partial_x\phi)  \tilde u  =
-\iint_{Q_{R/3}} \phi \partial_x\tilde u  
=-\iint_{Q_{R/3}} \phi \left( \partial_x\tilde u + \frac{3}{2\delta R}\right) + \frac{3}{2\delta R} \iint_{Q_{R/3}} \phi\\
& \leq CR  \|\phi\|_\infty + \|\phi\|_\infty \iint_{Q_{R/3}}  \left( \partial_x\tilde u + \frac{3}{2\delta R}\right) \qquad \text{(by \eqref{laxOleinick})}\\
& \leq CR\|\phi\|_\infty + \|\phi\|_\infty \int_{-\frac{R}{3}}^{\frac{R}{3}}\left[ \tilde u(s, \cdot)\right]_{-R'/3}^{R'/3}ds \leq CR\|\phi\|_\infty .
\end{align*}
The rigorous derivation of the above inequality can be achieved by regularization. 
It implies that 
$$
\int_{-R/3}^{R/3} V_x( u(s,\cdot); [-R'/3,R'/3]) ds \leq CR,
$$ 
where $V_x$ denotes the total variation in the $x$ variable. 
On the other hand, by the equation satisfied by $u$, 
\begin{align*}
\iint_{Q_{R/3}} (\partial_t \phi)  u&  = -\iint_{Q_{R/3}} (\partial_x \phi)  f( u) \leq \|\phi\|_\infty \int_{-R/3}^{R/3} V_x(f( u(s,\cdot)); [-R'/3,R'/3])ds \\
& \leq \|f'\|_\infty \|\phi\|_\infty \int_{-R/3}^{R/3} V_x(u(s,\cdot); [-R'/3,R'/3])ds  \leq CR \|\phi\|_\infty. 
\end{align*}
This implies the result. 
\end{proof}

\subsection{Local correspondence: viscosity solutions versus entropy solution} 

Equivalence between Hamilton-Jacobi equation and scalar conservation laws in one space dimension has been discussed in several papers: see for instance \cite{CP20, KR02} {(see also Lemma \ref{lem::r3} below)}. The following statement can be deduced from these reference combined with a localization argument in the  spirit of the proof of Proposition \ref{prop.intregu}: 

\begin{Lemma}{\bf (Local correspondence viscosity solution versus entropy solution)}\label{lem::r2}\\
Let $a<c$, $\delta>0$ and $f:[a,c]\to \R$ be $C^2$ such that $f''\le -\delta$. Let $T>0$ and $R>0$
Let $v:\Omega \to \R$ be a Lipschitz continuous function with $\Omega:=(0,T)\times (-R,R)$ and $\partial_x v\in [a,c]$ a.e. on $\Omega$.\\
If $v$ is a viscosity solution of
$$\partial_t v + f(\partial_x v)=0\quad \mbox{on}\quad \Omega$$
then $u=\partial_x v$ is an entropy solution of
$$\partial_t u+ \partial_x f(u)=0\quad \mbox{on}\quad \Omega.$$
\end{Lemma}

\subsection{Correspondence for a junction: viscosity solutions versus entropy solution}

\begin{Lemma}{\bf (Correspondence for a junction: viscosity solution versus entropy solution, \cite{CFGM})}\label{lem::r3}\\
For $i=L,R$, let real numbers $a^i< b^i< c^i$ and functions $f^i:[a^i,c^i]\to \R$ be $C^2$ satisfying $(f^i)''\le -\delta<0$, increasing on $[a^i,b^i]$ and decreasing on $[b^i, c^i]$ {and such that $f^i(a^i)=f^i(c^i)=0$}. We define the monotone envelopes
$$f^{i,+}(p)=\left\{\begin{array}{ll}
f^i(p)&\quad \mbox{for}\quad p\in [a^i,b^i]\\
f^i(b^i) &\quad \mbox{for}\quad p\in [b^i,c^i]\\
\end{array}\right.\quad \mbox{and}\quad
f^{i,-}(p)=\left\{\begin{array}{ll}
f^i(b^i)&\quad \mbox{for}\quad p\in [a^i,b^i]\\
f^i(p) &\quad \mbox{for}\quad p\in [b^i,c^i]\\
\end{array}\right.
$$
Let $T>0$ and a flux limiter $A\ge 0$. Let  $v=(v^L,v^R)$ be a viscosity solution (in the sense of \cite{IM17}) of
$$\left\{\begin{array}{ll}
v^L_t+f^L(v^L_x)=0&\quad \mbox{on}\quad (0,T)\times (-\infty,0)\\
v^R_t+f^R(v^R_x)=0&\quad \mbox{on}\quad (0,T)\times (0,+\infty)\\
v(t,0):=v^L(t,0)=v^R(t,0)&\quad \mbox{on}\quad (0,T)\times \left\{0\right\}\\
\partial_t v(t,0)+\min\left\{A,f^{L,+}(v^L_x(t,0^-)),f^{R,-}(v^R_x(t,0^+))\right\} =0 &\quad \mbox{on}\quad (0,T)\times \left\{0\right\}\\
v=v_0 & \quad \mbox{on}\quad \left\{0\right\}\times \R
\end{array}\right.$$
with $v$ uniformly Lipschitz continuous on $[0,T)\times \R$.\\
{\noindent {\bf i) (Natural result)}}\\
Then $u=\partial_xv=(u^L,u^R)$ is an entropy solution of 
$$\left\{\begin{array}{ll}
\partial_t u^L+\partial_x(f^L(u^L))=0&\quad \mbox{on}\quad (0,T)\times (-\infty,0)\\
\partial_t u^R+\partial_x(f^R(u^R))=0&\quad \mbox{on}\quad (0,T)\times (0,+\infty)\\
(u^L(t,0^-),u^R(t,0^+))\in \mathcal G_A &\quad  \mbox{a.e. on}\quad  (0,T)\times \left\{0\right\}\\
u=u_0 & \quad \mbox{a.e. on}\quad \left\{0\right\}\times \R
\end{array}\right.$$
with $u_0=\partial_x v_0$ and with
$$\mathcal G_A:=\left\{(p^L,p^R)\in [a^L,c^L]\times [a^R,c^R],\quad \min\left\{A,f^{L,+}(p^L),f^{R,-}(p^R)\right\}=f^L(p^L)=f^R(p^R)\right\}.$$
{\noindent {\bf ii) (A variant)}\\
The result is still true for $f^i=f^{i,\pm}=0=a^i=b^i=c^i$ for $i=L$, or for $i=R$ or for both.}
\end{Lemma}

\subsection{The envelope theorem} 

We recall the following result (which is easy to prove directly).
\begin{Theorem}{\bf (Envelope theorem)}\label{th::r1}\\
Let $\Omega\subset \R^n$ be an open set for $n\ge 1$and $Y$ be a compact set. Consider a function $\varphi:\Omega_x\times Y_y \to \R$ and let
$$h(x)=\max_{y\in Y} \varphi(x,y)$$
We make the following assumptions on $\varphi$
\begin{equation}\label{eq::rr1}
\left\{\begin{array}{l}
\mbox{the map $\varphi$ is continuous on $\Omega_x\times Y_y$}\\
\mbox{the map $\varphi(\cdot,y)$ is differentiable on $\Omega_x$ for each $y\in Y$, with derivative $\varphi_x(\cdot,y)$}\\
\mbox{the map $\partial_x \varphi$ is continuous on $\Omega_x\times Y_y$}\\
\end{array}\right.
\end{equation}
\noindent {\bf i) (The directional derivative)}\\
For any $v\in \R^n$, the function $h$ has directional derivative at each point $x_0\in \Omega$ which is defined by
$$D^+_vh(x_0):=\lim_{\varepsilon\to 0^+} \frac{h(x+\varepsilon v)-h(x)}{\varepsilon}$$
and we have
$$D^+_vh(x_0)=\max_{y_0\in \mbox{\footnotesize Argmax} \ \varphi(x_0,\cdot)}\ v\cdot \partial_x\varphi(x_0,y_0)$$
with
$$\mbox{Argmax} \ \varphi(x_0,\cdot):=\left\{y_0\in Y,\quad  \varphi(x_0,y_0)=\max_{y\in Y} \varphi(x_0,y)\right\}.$$
\noindent {\bf ii) (When $h$ has already a derivative)}\\
Assume that $h$ has a derivative at $x_0\in \Omega$. Then we have
$$\partial_xh(x_0)=\partial_x\varphi(x_0,y_0)\quad \mbox{for all}\quad y_0\in \mbox{Argmax} \ \varphi(x_0,\cdot).$$
\noindent {\bf iii) (Existence of a derivative for $h$)}\\
Let $x_0\in \Omega$. If the map $v\mapsto D^+_vh(x_0)$ is linear, then $h$ has a derivative at $x_0$.\\
\noindent {\bf iv) (The basic result)}\\
Let $x_0\in \Omega$.
$$\mbox{If $\mbox{Argmax} \ \varphi(x_0,\cdot)= \left\{y_0\right\}$ is a singleton},$$
then  $h$ has a derivative at $x_0$ and
$$\partial_xh(x_0)=\partial_x\varphi(x_0,y_0).$$
\end{Theorem}

\paragraph{\textbf{Acknowledgement.}}
This research was partially funded by l'Agence Nationale de la Recherche (ANR), project ANR-22-CE40-0010 COSS. 
For the purpose of open access, the authors have applied a CC-BY public copyright licence to any Author Accepted Manuscript (AAM) version arising from this submission.


\begin{thebibliography}{ABC}


%

\bibitem{ACCT13} 
\textsc{Y. Achdou, F. Camilli, A. Cutr\`{\i}, and N. Tchou},
{\it  Hamilton-Jacobi equations constrained on networks}. Nonlinear Differential Equations and Applications NoDEA, 20(3) (2013), 413-445.

\bibitem{AOT15} 
\textsc{Y. Achdou, F. Camilli, S. Oudet, and N. Tchou},
{\it Hamilton-Jacobi equations for optimal control on junctions and networks}. 
ESAIM: Control, Optimisation and Calculus of Variations, 21(3) (2015), 876-899.


\bibitem{AcTc15} 
\textsc{Y. Achdou, and N. Tchou}, 
{\it  Hamilton-Jacobi equations on networks as limits of singularly perturbed problems in optimal control: dimension reduction}. Communications in Partial Differential Equations, 40(4) (2015), 652-693.



%
\bibitem{AGDV11}
\textsc{Adimurthi, S. S. Ghoshal, R. Dutta, and G.D. Veerappa Gowda}, 
{\it Existence and nonexistence of TV bounds for scalar conservation laws with discontinuous flux}. Communications on pure and applied mathematics, 64(1) (2011), 84-115.

%


%
\bibitem{ACD}
\textsc{B. Andreianov, G. M. Coclite and C. Donadello},
{\it Well-posedness for vanishing viscosity solutions of scalar conservation laws on a network.}
Discrete Contin. Dyn. Syst.  37 (2017), 5913-5942.

\bibitem{AKR11}  
\textsc{B. Andreianov, K.H. Karlsen and  N.H. Risebro}, 
{\it A theory of $L^ 1$-dissipative solvers for scalar conservation laws with discontinuous flux.} 
Arch. Ration. Mech.  Anal. 201 (2011), 27-86.

%
\bibitem{AuPe05}
\textsc{E. Audusse, and B. Perthame},
{\it  Uniqueness for scalar conservation laws with discontinuous flux via adapted entropies}. Proceedings of the Royal Society of Edinburgh Section A: Mathematics, 135(2) (2005), 253-265.

%
\bibitem{BaJe97}
\textsc{P. Baiti, and H.K. Jenssen}, 
{\it Well-Posedness for a Class of $2\times 2$ Conservation Laws with  $L^\infty$Data}. Journal of differential equations, 140(1)  (1997), 161-185.

%
\bibitem{BaCh} 
\textsc{G. Barles, and E. Chasseigne}, 
{\sc An illustrated guide of the modern approches of Hamilton-Jacobi equations and control problems with discontinuities}. {(2023)}. arXiv preprint arXiv:1812.09197.


%
\bibitem{BCGHP}
\textsc{A. Bressan, S. Canic, M. Garavello, H. Herty and P. Piccoli},
{\it Flows on networks: Recent results and perspectives}, 
EMS Surv. Math. Sci.  1 (2014), 47-111.

%
\bibitem{CaMa13} 
\textsc{F. Camilli, and C. Marchi}, 
{\it A comparison among various notions of viscosity solution for Hamilton-Jacobi equations on networks}. Journal of Mathematical Analysis and applications, 407(1) (2013), 112-118.

%

%
\bibitem{CFGM} 
\textsc{P. Cardaliaguet, N.  Forcadel, T.  Girard  and R. Monneau},
{\it {Conservation law and Hamilton-Jacobi equations on a junction: the convex case}}. {Preprint, https://hal.science/hal-04279829.}

%
\bibitem{CFarma}
\textsc{P. Cardaliaguet, and N.  Forcadel} 
{\it Microscopic derivation of a traffic flow model with a bifurcation}. To appear in ARMA. 

\bibitem{CFsiam21}
\textsc{P. Cardaliaguet, and N.  Forcadel},
{\it From heterogeneous microscopic traffic flow models to macroscopic models}. 
SIAM Journal on Mathematical Analysis, 53(1) (2021), 309-322.
%


\bibitem{CP20}
\textsc{R.M. Colombo, and V. Perrollaz},
{\it Initial data identification in conservation laws and Hamilton-Jacobi equations}. 
Journal de Math\'ematiques Pures et Appliqu\'ees, 138 (2020), 1-27.


%
\bibitem{CHM20}  
\textsc{R.~M. Colombo, H.~Holden, and F.~Marcellini},  On the microscopic modeling of
vehicular traffic on general networks.  SIAM J. Appl. Math. 80 (2020), no. 3, 1377-1391.

\bibitem{Da05} 
\textsc{C.M. Dafermos}, 
{\sc Hyperbolic conservation laws in continuum physics} (Vol. 3) (2005). Berlin: Springer.

\bibitem{Da09} 
\textsc{A.L. Dalibard}, 
{\it Homogenization of non-linear scalar conservation laws}. 
Archive for rational mechanics and analysis, 192(1) (2009), 117-164.


%


%

%
\bibitem{Ew92}
\textsc{W. E}. 
{\it Homogenization of scalar conservation laws with oscillatory forcing terms}. 
SIAM Journal on Applied Mathematics (1992), 959-972.


%
\bibitem{rigorousLWR}
\textsc{M.~Di~Francesco and M.~D. Rosini}, 
{\it Rigorous derivation of nonlinear
  scalar conservation laws from follow-the-leader type models via many particle
  limit}, Arch. Ration. Mech. Anal., 217 (2015), pp.~831--871.

\bibitem{FMR22b} 
\textsc{U.S. Fjordholm, M. Musch and N.H. Risebro},
{\it Well-posedness and convergence of a finite volume method for conservation laws on networks.} 
SIAM J. Numer. Anal. 60 (2) (2022), 606-630.

%
\bibitem{FSZ18} 
\textsc{N. Forcadel, W. Salazar, and M. Zaydan}, 
{\it Specified homogenization of a discrete traffic model leading to an effective junction condition}. 
Communications on Pure and Applied Analysis, 17(5)  (2018), 2173-2206.

%
\bibitem{FoSa20}
\textsc{N. Forcadel, and W. Salazar}, 
{\it Homogenization of a discrete model for a bifurcation and application to traffic flow}. Journal de Math\'ematiques Pures et Appliqu\'ees, 136 (2020), 356-414.

%
\bibitem{GaImMo15} 
\textsc{G. Galise, C. Imbert, and R. Monneau}, 
{\it A junction condition by specified homogenization and application to traffic lights}. 
Analysis \& PDE, 8(8) (2015), 1891-1929.

%
\bibitem{GP1}
\textsc{M. Garavello and B. Piccoli}, 
{\it Traffic flow on networks}, 
American Institute of Mathematical Sciences (AIMS), Springfield, MO, (2006).
%
\bibitem{GP2}
\textsc{M. Garavello and B. Piccoli}, 
{\it Conservation laws on complex networks}, Ann. Inst. H. Poincare Anal. Non Lin\'eaire  26 (2009), 1925-1951.

%
\bibitem{GNPT07}
\textsc{M. Garavello, N. Natalini, B. Piccoli and A. Terracina}, 
{\it Conservation laws with discontinuous flux}. Networks and Heterogeneous Media, 2(1) (2007), 159.


%
\bibitem{GR20}
\textsc{P. Goatin and E. Rossi},
{\it  Comparative study of macroscopic traffic flow models at road junctions.} 
Netw.  Heterog. Media 15(2) (2020), 261-279.

%

\bibitem{HR}
\textsc{H. Holden and N. H. Risebro}. 
{\it A mathematical model of traffic flow on a network of unidirectional roads}. SIAM J. Math. Anal., 26 (4) (1995), 999-1017.


\bibitem{H22}
\textsc{Y. Holle},
{\it Entropy Dissipation at the Junction for Macroscopic Traffic Flow Models.}
SIAM J. Math. Anal. 54 (1) (2022), 954-985.

\bibitem{IM17} 
\textsc{C. Imbert and R. Monneau}, 
{\it Flux-limited solutions for quasi-convex Hamilton-Jacobi equations on networks.} 
Annales scientifiques de l'ENS 50 (2) (2017), 357-448

\bibitem{IMZ}
\textsc{C. Imbert, R. Monneau and H. Zidani},
{\it  A Hamilton-Jacobi approach to junction problems and application to traffic flows}. 
ESAIM: Control, Optimisation and Calculus of Variations 19 (2013), 129-166.
%
\bibitem{KR02}
\textsc{K.H. Karlsen and  N.H. Risebro},
{\it A note on Front tracking and the Equivalence between Viscosity Solutions of Hamilton-Jacobi 
Equations And Entropy Solutions of scalar Conservation Laws.}
Nonlin. Anal. TMA 50 (4) (2002), 455-469.

%


%
\bibitem{LiSo16} 
\textsc{P.-L. Lions and P. Souganidis},
{\it Viscosity solutions for junctions: well posedness and stability}. 
Rendiconti Lincei-matematica e applicazioni, 27(4) (2016), 535-545. 


{\bibitem{LS2}
\textsc{P.-L. Lions, P. Souganidis},
{\it Well-posedness for multi-dimensional junction problems with Kirchoff-type conditions},
Rend. Lincei Mat. Appl. 28 (2017), 807-816.}


%
\bibitem{LS}
\textsc{P.-L. Lions and P. Souganidis},
{\it Scalar conservation laws: Initial and boundary value problems
revisited and saturated solutions},
C. R. Acad. Sci. Paris, Ser. I 356 (2018), 1167-1178.

\bibitem{MFR22} 
\textsc{M. Musch,  U.S. Fjordholm  and N.H. Risebro},  
{\it Well-posedness theory for nonlinear scalar conservation laws on networks.} 
Netw. Heterog. Media, 17 (2022), 101-128.

%


\bibitem{Ps07} 
\textsc{E.Y. Panov},   
{\it Existence of strong traces for quasi-solutions of multidimensional conservation laws.}
J. Hyperbolic Differ. Equ. 4 (2007), 729-770.

%
\bibitem{ShCa13} 
\textsc{D. Schieborn, D., and F. Camilli}, 
{\it Viscosity solutions of Eikonal equations on topological networks}. 
Calculus of Variations and Partial Differential Equations, 46(3) (2013), 671-686.

%
\bibitem{Se92}
\textsc{D. Serre}, 
{\it Correctors for the homogenization of conservation laws with oscillatory forcing terms}. 
Asymptotic analysis, 5(4) (1992), 311-316.


%
\bibitem{To00}
\textsc{J.D. Towers},
{\it Convergence of a difference scheme for conservation laws with a discontinuous flux}. SIAM journal on numerical analysis, 38(2) (2000), 681-698.

%
\bibitem{To01}
\textsc{J.D. Towers},
{\it A difference scheme for conservation laws with a discontinuous flux: the nonconvex case}. SIAM journal on numerical analysis, 39(4) (2001), 1197-1218.

%
\bibitem{To20}
\textsc{J.D. Towers},
{\it Well-posedness of a model of merging and branching traffic flow},
ResearchGate preprint (2020).


%
\bibitem{V01}
\textsc{A. Vasseur},
{\it Strong Traces for Solutions of Multidimensional Scalar Conservation Laws.}
Arch. Rational Mech. Anal. 160 (2001), 181-193.





\end{thebibliography}
\end{document}